\documentclass[a4paper,12pt,oneside,reqno]{amsart}
\usepackage{hyperref}
\usepackage[headinclude,DIV13]{typearea}
\areaset{15.1cm}{25.0cm}
\usepackage{amsfonts,amssymb,amsmath,amsthm,bbm}
\usepackage[latin1] {inputenc}
\usepackage{graphicx, psfrag}
\usepackage{xcolor}

\newtheorem{theorem}{Theorem}[section]
\newtheorem{lemma}[theorem]{Lemma}
\newtheorem{proposition}[theorem]{Proposition}
\newtheorem{corollary}[theorem]{Corollary}

\theoremstyle{definition}
\newtheorem{definition}[theorem]{Definition}

\theoremstyle{remark}
\newtheorem*{remark}{Remark}

\def\paragraph#1{\noindent \textbf{#1}}

\numberwithin{equation}{section}

\def\dist{\mathop{\rm dist}\nolimits}

\def\d{\mathrm{d}}

\def\rf{\rfloor}

\def\lf{\lfloor}

\def\<{\langle}
\def\>{\rangle}

\def\a{\alpha}
\def\b{\beta}
\def\e{\epsilon}
\def\ve{\varepsilon}

\def\g{\gamma}

\def\l{\lambda}

\def\s{\sigma}
\def\t{\tau}

\def\z{\zeta}
\def\o{\omega}

\def\L{\Lambda}
\def\G{\Gamma}
\def\O{\Omega}

\def\R{{\Bbb R}}  
\def\N{{\Bbb N}}  
\def\P{{\Bbb P}}  
  
\def\Q{{\Bbb Q}}  
 
\def\E{{\Bbb E}}

\let\cal=\mathcal
\def\AA{{\cal A}}

\def\CC{{\cal C}}
\def\DD{{\cal D}}
\def\EE{{\cal E}}
\def\FF{{\cal F}}
\def\GG{{\cal G}}
\def\HH{{\cal H}}

\def\MM{{\cal M}}

\def\OO{{\cal O}}
\def\PP{{\cal P}}

\def\RR{{\cal R}}

\def\TT{{\cal T}}
\def\VV{{\cal V}}

\def\VV{{\cal V}}

 \def \G {{\Gamma}}
 \def \L {{\Lambda}}
 
 \def \b {{\beta}}
 \def \s {{\sigma}}
 
 \def \z {{\zeta}}

 \def \t {{\tau}}

 \def \g {{\gamma}}
 \def \l {{\lambda}}
 \def \d {{\delta}}
 \def \a {{\alpha}}
 \def \o {{\omega}}
 \def \O {{\Omega}}

 \def \x {{\xi}}

 \def \ba {\begin{array}}
 \def \ea {\end{array}}

 \newcommand{\be}{\begin{equation}}
 \newcommand{\ee}{\end{equation}}

\newcommand{\bea}{\begin{eqnarray}}
 \newcommand{\eea}{\end{eqnarray}}
\def\TH(#1){\label{#1}}\def\thv(#1){\ref{#1}}
\def\Eq(#1){\label{#1}}\def\eqv(#1){(\ref{#1})}

\def\sfrac#1#2{{\textstyle{#1\over #2}}}

 \def \1{\mathbbm{1}}
\def\wt {\widetilde}
\def\wh{\widehat}

\def\PRM{\hbox{\rm PRM}}
\def\asl{\hbox{\rm Asl}}

\begin{document}

\title[Dynamic phase diagram of the REM]
{Dynamic phase diagram of the REM}
\author[V. Gayrard]{V\'eronique Gayrard}
\address{V. Gayrard\\ Aix Marseille Univ, CNRS, Centrale Marseille, I2M, Marseille, France}
\email{veronique.gayrard@math.cnrs.fr}
\author[L. Hartung]{Lisa Hartung}
 \address{L. Hartung,Institut für Mathematik, Johannes Gutenberg-Universität Mainz, Staudingerweg 9,
55099 Mainz, Germany }
\email{lhartung@uni-mainz.de}

\date{\today}

\begin{abstract}
By studying the  two-time overlap correlation function, we give a comprehensive analysis of the phase diagram of
the Random Hopping Dynamics of the Random Energy Model (REM) on time-scales that are exponential in the volume.
These results are derived from the convergence properties of the clock process associated to the dynamics and fine properties of the simple random walk in the $n$-dimensional discrete cube.
\end{abstract}
 
\maketitle


\section{Introduction.}
\label{S1}

Sometimes called the simplest spin glass, the Random Energy Model (REM) played a decisive r\^ole in the understanding of \emph{aging}, a characteristic slowing down of the relaxation dynamics of spin glasses (see \cite{BBG1}, \cite{BBG2}, \cite{BC06b},  \cite{G10b}, \cite{CW15}, \cite{G15}, \cite{G18}, \cite{MM13}, for mathematical works, and \cite{BD}, \cite{Bou92}, \cite{BB02} and the review \cite{BCKM98} for those of theoretical physics). This phenomenon is  quantified through two-time correlations functions.
In this paper, we study the two-time \emph{overlap} correlation function of the REM evolving under the simplest Glauber dynamics,  the so-called \emph{Random Hopping Dynamics} (hereafter, RHD), and give its complete (dynamic) phase diagram as a function of the inverse temperature,  $\beta>0$, and of the time-scale, $c_n$, when  the latter is exponential in the dimension $n$ of the state space, $\{-1,1\}^n$.
The objectives of this paper are twofold: to give the complete picture for a key mean-field spin glass model for which only part of the picture was known to date, and to do it by means of an effective and unifying technique.

More specifically, the proof is based on a well-established universal aging scheme, first put forward in \cite{BC06b}, which links aging to the arcsine law for stable subordinators  through a partial sum process called \emph{clock-process}.
The latter is then analyzed through powerful techniques  drawn from  Durrett and Resnick's  work on convergence of partial sum processes of dependent random variables  to subordinators \cite{DuRe}. These techniques were first introduced in the context of aging dynamics in \cite{G12} and  have since proved very effective in more complex spin-glass models or dynamics  \cite{G15}, \cite{G18}, \cite{BG13}, \cite{BGS13}, for which the universality of the REM-like aging (or \emph{arcsine-law aging}) was confirmed. 
It should be noted here  that this paper is in large part based on the unpublished work \cite{G10b} which is complemented
 by new results (in particular, analysis of the overlap correlation function is new as well as the study of the high temperature and short time-scale transition line between aging and stationarity).

\subsection{The setting.}
\label{S1.1}
We now specify the model.
Denote by $\VV_n=\{-1,1\}^n$ the n-dimensional discrete cube and by $\EE_n$ its edges set.
The Hamiltonian  of the REM is a collection of independent  Gaussian random variables,  
$(\HH_n(x), x\in\VV_n)$, with $\E \HH_n(x)=0$ and $\E \HH^2_n(x)=n$.
Assigning to each site $x$ the Boltzman weight 
\be
\t_n(x)\equiv \exp\{-\b \HH_n(x)\},
\Eq(S1.1.1)
\ee
the  Random Hopping Dynamics in the  \emph{random landscape}  
$(\t_n(x), x\in\VV_n)$ is the Markov jump process  $(X_{n}(t), t>0)$ with rates
\be
\l_n(x,y)=(n\t_n(x))^{-1},\quad\text{if $(x,y)\in\EE_n$},
\Eq(S1.1.2)
\ee
and $\l_n(x,y)=0$ else. Clearly, it is reversible with respect to Gibbs measure.
The sequence of random landscapes $(\t_n(x), x\in\VV_n)$, $n\geq 1$,  or  \emph{random environment}, is defined on a common probability space denoted $(\O^{\t}, \FF^{\t}, \P)$.  We  refer to the $\s$-algebra generated by the variables $X_n$ as $\FF^X$. We denote by $\mu_n$ the initial distribution of $X_n$ and write  $\PP_{\mu_n}$ for the law of $X_n$ started in $\mu_n$, conditional on $\FF^{\t}$, i.e. for fixed realizations of the random environment.

 To study the dynamic phase diagram of the process $X_n$ we must choose three quantities:
\begin{itemize}
\item[(1)]
the time-scale of observation, 
\item[(2)]
 a two-time correlation function, 
\item[(3)]
and the initial distribution.
\end{itemize}
We are  interested in time-scales that are exponential in $n$. We further must distinguish two types of exponential time-scales called {\it intermediate} and {\it extreme},  defined as follows. Given a time-scale $c_n$, let $a_n$ be defined through
\be
a_n\P(\t_n(x)\geq c_n)=1.
\Eq(1.1')
\ee

\begin{definition}
\TH(1.def1)
We say that a diverging sequence $c_n$ is (i) an {\it intermediate}  time-scale if there exists a constant $0<\varepsilon\leq 1$ such that
\be
\lim_{n\rightarrow\infty}\frac{\log a_n}{n\log2}=\varepsilon  \quad\text{and}\quad \lim_{n\rightarrow\infty}\frac{a_n}{2^n}=0.
\Eq(1.2)
\ee
 (ii) It is an {\it extreme} time-scale if ($\varepsilon=1$ and) there exists a constant $0<\bar\varepsilon<\infty$ such that 
\be
\lim_{n\rightarrow\infty}\frac{a_n}{2^n}= \bar\varepsilon.
\Eq(1.3)
\ee
\end{definition}

The natural  two-time correlation function  of interest in mean-field spin glasses is the  \emph{overlap} correlation function,
$\CC^\circ_{n}(t,s)$: given two times $t,s>0$ and a parameter $0<\rho<1$,
\be
\CC^\circ_{n}(t,s)=\PP_{\mu_n}\left(n^{-1}\big(X_n(c_n t), X_n(c_n (t+s)\big)\geq 1-\rho\right)
\Eq(TTCF.0)
\ee
where $(\cdot,\cdot)$ denotes the inner product in $\R^n$. The central  idea underlying the aging mechanism  based on the arcsine law for stable subordinators is that, as stated in Theorem \thv(TTCF.theo1) below, $\CC^\circ_{n}(t,s)$ coincides asymptotically
with the \emph{no-jump}  correlation function $\CC_{n}(t,s)$ used to quantify aging in the trap models of theoretical physics  \cite{BD},
\cite{BCKM98}, and defined as
 \be
\CC_{n}(t,s)=\PP_{\mu_n}\left(n^{-1}\big(X_n(c_n t), X_n(c_n u)\big)=1,\,\forall t\leq u<t+s\right).
\Eq(1.1.8)
\ee

\begin{theorem}[From the overlap to the no-jump correlation function]
\TH(TTCF.theo1) 
Let $c_n$ be either an intermediate or an extreme time-scale and let $\mu_n$ be any initial distribution.
For all $0<\rho<1$ and for all $0<\varepsilon\leq 1$ and all  $0<\beta<\infty$ such that
$0<\a(\varepsilon)\leq 1$ we have that $\P$-almost surely on intermediate time-scale and in 
$\P$-probability on extreme time-scales, for all $t\geq 0$ and all $s>0$
\be
\lim_{n\rightarrow\infty}\CC^\circ_{n}(t,s)=\lim_{n\rightarrow\infty}{\CC}_n(t,s).
\Eq(TTCF.theo1.1)
\ee
and, for $\a(\varepsilon)=1$ on intermediate time-scale
\be
\lim_{n\rightarrow\infty}\sqrt{n}\CC^\circ_{n}(t,s)=\lim_{n\rightarrow\infty}\sqrt{n}{\CC}_n(t,s).
\Eq(TTCF.theo1.2)
\ee
\end{theorem}

From now on we focus on ${\CC}_n(t,s)$. Unless otherwise specified, the initial distribution is the uniform distribution
\be
\pi_n(x)=2^{-n},\quad x\in \VV_n.
\Eq(1.8)
\ee
It models the experimental procedure of a \emph{deep quench} which aims to draw a typical initial state.

\subsection{Main results} 
\label{S1.3}
We are interested in the  behavior of the correlation function ${\CC}_n(t,s)$ in the limit $n\rightarrow\infty$. Whereas in stationary dynamics ${\CC}_n(t,s)$ is  asymptotically time translational invariant, in out-of-equilibrium aging dynamics a history dependence appears.  Our first theorem characterizes this \emph{aging phase}. For  $0<\varepsilon\leq 1$ and  $0<\beta<\infty$, we set
\be
\begin{split}
\b_c(\varepsilon)&=\sqrt{\varepsilon2\log 2},
\\
\a(\varepsilon)&=\b_c(\varepsilon)/\b,
\end{split}
\Eq(1.theo1.0')
\ee
and write $\b_c\equiv\b_c(1)$, $\a\equiv\a(1)$. Note that $\b_c$ is the static critical temperature at which a transition  occurs  between  distinct high and low temperature limiting  Gibbs measures (see Section 9.3 of \cite{AB12} for their description). Denote by $\asl_{\a(\varepsilon)}(\cdot)$ the probability distribution function of the generalized  arcsine law of parameter 
$\a(\varepsilon)$,
\be
\asl_{\a(\varepsilon)}(u)=\frac{\sin{(\a(\varepsilon)}\pi)}{\pi}\int_0^u (1-x)^{-\a(\varepsilon)}x^{-1+\a(\varepsilon)}dx,
\quad 0<\a(\varepsilon)< 1.
\Eq(1.theo1.0'')
\ee

\begin{theorem}[Aging]
{\TH(1.theo1)}
Let $c_n$ be an intermediate time-scale. For all  $0<\varepsilon\leq 1$ and all  $0<\beta<\infty$ such that
 $0<\a(\varepsilon)< 1$ the following holds $\P$-almost surely if $\sum_{n}a_n/2^n<\infty$ and  in $\P$-probability if 
 $\sum_{n}a_n/2^n=\infty$: for all $t\geq 0$ and all $s>0$,
\be
\lim_{n\rightarrow\infty}\CC_{n}(t,s)=\asl_{\a(\varepsilon)}(t/t+s).
\Eq(1.theo1.1)
\ee
\end{theorem}

Eq.~\eqv(1.theo1.1) was first proved in \cite{BC06b} (see Theorem 3.1) and later in \cite{CG} (see Theorem 2.1) in subregions 
of the above $\P$-almost sure convergence region.

Call $\DD(\varepsilon,\beta)$ the domain of validity of Theorem \thv(1.theo1).  It is delimited in the $(\varepsilon,\beta)$-parameter plane by three transition lines which are: the curve $\b_c(\ve)=\b$ and $0< \varepsilon\leq 1$, arising at intermediate time scales, the plateau  $\varepsilon= 1$ and  $\b>\b_c(1)$, appearing at extreme time scales, and the axis $\varepsilon=0$ and $\b>0$, corresponding to time-scales that are sub-exponential in $n$. Notice that these three  transition lines correspond, respectively, to $\a(\varepsilon)=1$, $0<\a\equiv\a(1)<1$ and $\a(\varepsilon)=0$, whereas  inside $\DD(\varepsilon,\beta)$, $0<\a(\varepsilon)<1$.  We will see in Subsection \thv(S1.2) that to these different values of $\a(\varepsilon)$ correspond different behaviors of the clock process.

The domain  $\DD(\varepsilon,\beta)$ is the optimal domain of validity of \eqv(1.theo1.1). On the one hand it is easy to prove that  on sub-exponential time-scales, i.e.~when $\varepsilon=\varepsilon(n)\downarrow 0$ as $n$ diverges, $\lim_{n\rightarrow\infty}\CC_{n}(t,s)=1$  $\P$-a.s.. A non-trivial limit can be obtained by considering a non-linear rescaling of time  \cite{Gun09} (see also  \cite{BGu12},  \cite{BGS13}) when $\varepsilon(n)$ decays slowly enough. On the other hand it is known that in the complement of $\DD(\varepsilon,\beta)$ in the upper half quadrant  $\varepsilon>0$, $\b>0$,   the process $X_n$ started in the uniform measure $\pi_n$ is asymptotically stationary \cite{PM00}. Here two distinct stationary phases must be distinguished, mirroring the two distinct static phases. As might be expected by virtue of the translational invariance of stationary dynamics,
correlations vanish in the high temperature stationary phase where  the limiting Gibbs measure resembles a uniform measure.

\begin{theorem}[High temperature stationary phase]
 \TH(hightemp) For $\b<\b_c(\ve)$ with $0<\ve\leq 1$ and $c_n$ an intermediate time-scale,  $\P$-almost surely
\be\Eq(Mainz.115)
\lim_{n\to \infty} \CC_n(t,s) =0.
\ee
\end{theorem}

In the complementary low temperature stationary phase, namely when $\b>\b_c$, $\CC_n(t,s)$ converges for all $t,s>0$ to a random function ${\CC}^{sta}_{\infty}(s)$ that reflects the Poisson-Dirichlet nature of the limiting Gibbs weights. We postpone the precise statement to Theorem \thv(1.theo2).

\begin{remark}
Note  that \cite{PM00} only provides upper bounds (see (1.8) and (1.9) therein) on the time needed for the process to be at a distance less than a constant from equilibrium. These bounds correspond precisely to the two transition lines delimiting  $\DD(\varepsilon,\beta)$ on exponential time-scales. It thus follows from Theorem \thv(1.theo1) that they are accurate, i.e.~that at shorter times the process is not in equilibrium.
\end{remark}

As shown in the remainder of this subsection, the two distinct (low and high temperature) static phases give rise to two distinct dynamical phase transitions between aging  and stationarity. We begin by examining  the high temperature critical line $\b=\b_c(\ve)$ and $0<\ve\leq 1$, focusing on the subregion of intermediate time-scales defined by $\b_c(\ve)=\b$ and
\be
\lim_{n\to\infty}\b\sqrt{n}-\frac{\log c_n}{\b\sqrt{n}} =\theta
\Eq(TTCF.1)
\ee
for some constant $\theta\in(-\infty,\infty)$. The reasons for this restriction, which are technical, are discussed in the remark below  \eqv(1'.theo3.2').

\begin{theorem}[High temperature critical line]
\TH(TCor.2)
Let $\b=\b_c(\ve)$ with $0<\ve\leq 1$. Let $c_n$ be an intermediate time-scale satisfying \eqv(TTCF.1) for some constant $\theta\in (-\infty,\infty)$. Then, for all $t,s>0$, $\P$-almost surely if $\sum_{n}a_n/2^n<\infty$ and in $\P$-probability else
\be
\Eq(TCor.2.1)
\lim_{n\to\infty}\sqrt{n}\CC_n(t,s) =   \frac{e^{-\theta^2/2}}{\Phi(\theta)}\log\left(1+\frac{t}{s}\right)\frac{1}{\b\sqrt{2\pi}},
\ee
where $\Phi(\theta)$ is the standard gaussian distribution function. 
\end{theorem}
\begin{remark}
A main motivation behind Theorem \thv(TCor.2) is the paper \cite{BB02}, where Bouchaud's trap model \cite{BD} is studied along its high temperature critical line, which predicts that the scaling form of  its correlation function  presents dynamical ultrametricity in the sense of Cugliandolo and Kurchan \cite{CuKu94}.
This result, that corresponds to the setting of i.i.d.~random variables in the domain of attraction of a one stable law,  easily follows from \cite{Er}.
 Since the limiting correlation functions of Bouchaud's trap model  and that of the REM (for both the RHD and Metropolis dynamics\cite{G18}) are the same in their aging phases, it is natural to ask whether the REM also exhibits dynamical ultrametricity along its high temperature critical line. Since $\CC_n(t,s) $ decays to zero as $n$ diverges whatever the choices of $t,s>0$, Theorem \thv(TCor.2) answers in the negative.
\end{remark}

We now turn to the low temperature critical line $\varepsilon= 1$ and  $\b>\b_c(1)$ at extreme time-scales. To describe the transition across this line we use two different double limiting procedures:  we first take the limit $n\rightarrow\infty$ and then, either take the further small time limit $t\rightarrow 0$, in which case the process falls back to aging (Theorem \thv(1.theo1.ext)), or take the large time limit $t\rightarrow \infty$, in which case the process crosses over to stationarity  (Theorem \thv(1.theo2)). We do not have an expression for the single $n\rightarrow\infty$ limit.

\begin{theorem}(Low temperature critical line: crossover to aging)
\TH(1.theo1.ext)
Let $c_n$ be an extreme time-scale. For all $\b>\b_c$ and all $\rho>0$, in $\P$-probability
\be
\lim_{t\rightarrow 0}\lim_{n\rightarrow\infty}\CC_{n}(t,\rho t)=\asl_{\a}(1/1+\rho).
\Eq(1.theo1.2)
\ee
\end{theorem}
Theorem \thv(1.theo1.ext) was first proved in \cite{BBG2}.
The proof based on  clock process that we give here is radically simpler than the metastability-based approach of \cite{BBG1}, \cite{BBG2}.

This result was proved again in  \cite{FL09} along a very  different route, namely by first constructing the scaling limit of the process $X_n$ at extreme time-scale, which is given by an ergodic process called $K$-process, and then, constructing the clock processes from which \eqv(1.theo1.2) can be derived.

To state the next theorem  let $\PRM(\mu)$ be the Poisson random measure on $(0,\infty)$  with marks $\{\g_k\}$ and mean measure $\mu$ satisfying $\mu(x,\infty)=x^{-\a}$, $x>0$, and define the function
\be
{\CC}^{sta}_{\infty}(s)=
\sum_{k=1}^{\infty}\frac{\g_k}{\sum_{k=1}^{\infty}\g_k}e^{-s/\g_k},\quad s\geq 0.
\Eq(1.theo2.1)
\ee

\begin{theorem}[Low temperature critical line: crossover to stationarity]  
{\TH(1.theo2)} 
Let $c_n$ be an extreme time-scale. The following holds for all $\b>\b_c$. Let $\overset{d}=$ denote equality in distribution.

\item{(i)} If  $\mu_n=\GG_n$ 
where $\GG_n(x)=\t_n(x)/\sum_{x\in\VV_n}\t_n(x)$ is Gibbs
measure, then for all $s,t>0$
\be
\lim_{n\rightarrow\infty}{\CC}_n(t,s)
\overset{d}=
{\CC}^{sta}_{\infty}(s).
\Eq(1.theo2.2)
\ee
\item{(ii)} If  $\mu_n=\pi_n$ then for all $s\geq 0$
\be
\lim_{t\rightarrow\infty}\lim_{n\rightarrow\infty}{\CC}_n(t,s)\overset{d}={\CC}^{sta}_{\infty}(s).
\Eq(1.theo2.3)
\ee
\end{theorem}

\subsection{Convergence of clock processes.} 
\label{S1.2}
This section gathers the clock-process convergence results that are behind the proofs of the results of Subsection \thv(S1.1). 
An alternative construction of the  process $X_n$ consists in writing it as a time-change of its \emph{jump chain}, $J_n$, by the \emph{clock process}, $\wt S_n$,
\be
X_n(t)=J_n(i) \quad\text{if}\quad \wt S_n(i)\leq t<\wt S_n(i+1)  \,\,\,\text{for some}\,\,\, i,
\Eq(1.1.7)
\ee
where  $(J_n(k), k\in\N)$ is the simple random walk on $\VV_n$ and, given a family  of independent mean one exponential random variables, $(e_{n,i},n\in\N, i\in\N)$, independent of $ J_n$, $\wt S_n$ is the partial sum process
\be
\wt S_n(k)=\sum_{i=0}^{k}\t_n(J_n(i))e_{n,i},\quad k\in \N.
\Eq(1.1.6)
\ee
Given sequences $c_n$ and $a_n$ define the rescaled clock process
\be
S_n(t)=c_n^{-1}\wt S_n(\lfloor a_n t\rfloor),\quad t\geq 0.
\Eq(G1.3.2')
\ee
We now state our results on $S_n$.  For this denote by
\be
\g_n(x)=c_n^{-1}\t_n(x),\quad x\in\VV_n
\Eq(4.prop3.1)
\ee
the rescaled landscape variables. Also denote by $\Rightarrow$ weak convergence in the c\`adl\`ag space $D([0,\infty))$ equipped with the Skorohod $J_1$-topology. 
\begin{theorem}[Intermediate scales] 
\TH(1'.theo1)
Let $c_n$ be an intermediate time-scale.
\item{(i)} For all $0<\varepsilon\leq 1$ and all  $0<\beta<\infty$ such that  $0<\a(\varepsilon)< 1$ the following holds: $\P$-almost surely if $\sum_{n}a_n/2^n<\infty$ and  in $\P$-probability  if $\sum_{n}a_n/2^n=\infty$
\be
S_n\Rightarrow S^{int},
\Eq(1.prop1.2)
\ee
where $S^{int}$ is a subordinator with zero drift with  L\'evy measure, 
$\nu^{int}$, defined on $(0,\infty)$  through
\be
\nu^{int}(u,\infty)=
u^{-\a(\varepsilon)}\a(\varepsilon)\G(\a(\varepsilon)),\quad u>0.
\Eq(1.prop1.1)
\ee

\item{(ii)} For all $0<\varepsilon\leq 1$ and all  $0<\beta<\infty$ such that  $\a(\varepsilon)=1$, the following holds: $\P$-almost surely if $\sum_{n}a_n/2^n<\infty$ and  in $\P$-probability  if $\sum_{n}a_n/2^n=\infty$
\be\Eq(1'.theo3.1')
S_n- M_n\Rightarrow  S^{crit},
\ee
where $S^{crit}$ is the L\'evy process with L\'evy triple 
$ (0,0,\nu^{int})$ and 
\be
M_n(t) = \sum_{i=1}^{[a_n t]}\sum_{x \in \mathcal{V}_n} p_n \left(J_n (i-1),x\right)  \g_n (x) \left(1-e^{-1/\g_n(x)}\right).
\Eq(1'.theo3.0)
\ee
If moreover  $c_n$ satisfies \eqv(TTCF.1) for some $\theta\in(-\infty,\infty)$ then for all $ T>0$ and all $ \e>0$, $\P$-almost surely if $\sum_{n}a_n/2^n<\infty$ and  in $\P$-probability  if $\sum_{n}a_n/2^n=\infty$
\be\Eq(1'.theo3.2')
\lim_{n\to\infty}\PP\bigg(\sup_{t\in[0,T]}\Big\vert M_n(t)-\E(E(M_n(1))) t\Big\vert>\e\bigg)=0.
\ee
\end{theorem}

\begin{remark} The behavior of centering term $\E\left(E\left(M_n(t) \right)\right)$ when  $\a(\varepsilon)=1$ is studied in Appendix \thv(B).
In the regime of scaling \eqv(TTCF.1) under which \eqv(1'.theo3.2') is obtained, the centering term  
$\E\left(E\left(M_n(t) \right)\right)$ is of order $\sqrt{n}$ and the  fluctuations of $M_n(t)$  are smaller than the likelihood to observe a jump of $S_n$ over a large interval.  This in particular allows for precise error controls in the analysis of the correlation function (when averaging with respect to the jump chain), resulting in the precision of the statement of Theorem \thv(TCor.2), including the exact constant on the right-hand side of \eqv(TCor.2.1). When \eqv(TTCF.1) is not satisfied, $\E\left(E\left(M_n(t) \right)\right)$ will not diverge like $\sqrt{n}$ but exponentially fast in $n$. Obtaining a statement as in \eqv(1'.theo3.2') or Theorem \thv(TCor.2) would require a precise error control on an exponential level, which is made impossible by the rough concentration estimates used in the analysis of $M_n(t)$. That these estimates can be improved however is anything but clear.
\end{remark}

\begin{proposition}\TH(LCor.3) Let $c_n$ be an intermediate time-scale.
\item{(i)}   For all $0<\varepsilon\leq 1$ and all  $0<\beta<\infty$ such that  $\a(\varepsilon)=1$ the following holds: for all $T>0$ and for all $\e>0$, $\P$-almost surely if $\sum_{n}a_n/2^n<\infty$ and  in $\P$-probability  if $\sum_{n}a_n/2^n=\infty$
\be\Eq(1'.theo3.2'.1)
\lim_{n\to\infty}\PP\bigg(\sup_{t\in[0,T]}\Big\vert\frac{M_n(t)}{\E(E(M_n(1)))} -t\Big\vert>\e\bigg)=0,
\ee
and
\be\Eq(LCor.1)
\lim_{n\to\infty}\PP\left(\sup_{t\in[0,T]}\left\vert \frac{S_n(t)}{\E(E(M_n(1)))}-t\right\vert>\e\right)=0. \ee 
\item{(ii)} 
For all $0<\varepsilon\leq 1$ and all  $0<\beta<\infty$ such that  $\a(\varepsilon)>1$, then for all $T>0$ and for all $\e>0$, $\P$-almost surely
\be\Eq(1'.theo3.2'.1.neu)
\lim_{n\to\infty}\PP\bigg(\sup_{t\in[0,T]}\Big\vert\frac{S_n(t)}{a_ne^{n\b^2/2}/c_n} -t\Big\vert>\e\bigg)=0.
\ee
\end{proposition}

\begin{remark}
Note that Proposition \thv(LCor.3) holds without assuming \eqv(TTCF.1) due to the (stronger) rescaling by $\E(E(M_n(1)))$.
\end{remark}

\begin{remark}
In the high temperature regime of \eqv(1'.theo3.2'.1.neu) (and Theorem \thv(hightemp)), the behavior of the clock is completely dominated by its small jumps. This is to be contrasted with \eqv(1.prop1.2) where the clock is  dominated by its extreme increments, and with \eqv(1'.theo3.1') where both phenomena are competing.
Although such a result may not seem to be of primary interest in the REM analysis, it is different in the GREM where several aging behaviors can coexist at different levels of the underlying hierarchical structure \cite{FG18}.
\end{remark}

Note that $S^{int}$ is a stable subordinators of index  $0<\a(\varepsilon)< 1$. In the case $\a(\ve)=1$, $S^{crit}$ is not a subordinator but a compensated pure jump L\'evy process.  In the case of extreme time-scales the limiting process is neither a stable process nor a deterministic process but a doubly stochastic subordinator.

As before let $\PRM(\mu)$ be the Poisson random measure on $(0,\infty)$  with marks $\{\g_k\}$ and 
mean measure $\mu$ defined through
\be
\mu(x,\infty)=x^{-\a},\quad x>0.
\Eq(1.theo2.0)
\ee 

\begin{theorem}[Extreme scales] 
\TH(1'.theo2)
If $c_n$ is an extreme time-scale then both the sequence of re-scaled landscapes
$(\g_n(x),\, x\in\VV_n)$,
$n\geq 1$, and the marks of $\PRM(\mu)$ can be represented on a common probability space
$(\O, \FF, \boldsymbol{P})$ such that, in this representation, denoting by $\boldsymbol{\sigma}_n$ 
the corresponding re-scaled clock process \eqv(G1.3.2'),
the following holds: for all $\b_c<\beta<\infty$, $\boldsymbol{P}$-almost surely,
\be
\boldsymbol{\sigma}_n\Rightarrow S^{ext},
\Eq(1.prop2.1)
\ee
where
$
S^{ext}
$
is the subordinator whose L\'evy measure, $\nu^{ext}$,
is the random measure on $(0,\infty)$ defined on $(\O, \FF, \boldsymbol{P})$ through
\be
\nu^{ext}(u,\infty)=\bar\varepsilon\sum_{k=1}^{\infty}e^{-u/\g_k},\quad u>0,
\Eq(1.prop2.2)
\ee
$\bar\varepsilon$ being defined in \eqv(1.3).
\end{theorem}

A similar process was first obtained in \cite{G12} in the simpler setting of trap models (see Proposition 4.9 and Section 7 therein).

Although the limiting subordinator is not stable, the tail of the random L\'evy  measure $\nu^{ext}$
is regularly varying at $0^+$. This is a key ingredient the proof of Theorem \thv(1.theo1.ext).

\begin{lemma}{\TH(1.lemma5)} If $\b>\b_c$, then
$\boldsymbol{P}$-almost surely,
$
\nu^{ext}(u,\infty)\sim \bar\varepsilon u^{-\a}\a\G(\a)
$
as
$
u\rightarrow 0^+
$.
\end{lemma}

For future reference, the $\s$-algebra generated by the variables $J_n$ is denoted by $\FF$. We write $P_{\mu_n}$ for the law of the jump chain $J_n$ started in $\mu_n$, conditional on the $\s$-algebra  $\FF^{\t}$, i.e. for fixed realizations of the random environment. As already mentioned, we likewise  call $\PP_{\mu_n}$ the law of $X_n$ started in $\mu_n$, conditional on $\FF^{\t}$ (see paragraph below \eqv(S1.1.2)). Observe that  $\pi_n$ is the invariant measure of the jump chain.

The remainder of the paper is organized as follows.
In the next section we give sufficient conditions for the convergence of the clock process to a pure jump L\'evy process. Moreover we give sufficient conditions for \eqv(1'.theo3.2') to hold. In Sections \thv(S3) and \thv(S4)  we establish preparatory results on the random landscape and the jump chain. In Section \thv(S5), \thv(S6)  and \thv(S7)  the conditions given in Section \thv(S2)  are verified. Section  \thv(S7)  contains in particular the proof of Theorem \thv(1'.theo1) and Proposition \thv(LCor.3).
 A detailed survey how these sections are organized will be given at the end of Section  \thv(S2). Section \thv(S8) is then devoted to the study of correlation functions on intermediate time scales and contains the proofs of Theorem \thv(1.theo1) and Theorem \thv(TCor.2). Section  \thv(S9) is a self-contained section dealing with the case of extreme times-scales. Finally, the proof of Theorem \thv(TTCF.theo1) is given in Section  \thv(S10). Three short appendices complete the paper.
 
\section{Key tools and strategy.}
\label{S2}
Recall that the initial distribution is  $\pi_n$ (see \eqv(1.8)). We now formulate conditions for the sequence $S_n$ to converge. The idea of proof is taken from Theorem 1.1 of \cite{G12}. We state these conditions for given sequences $c_n$ and $a_n$, and for a fixed realization of the random landscape, i.e.~for fixed $\o\in\O^{\t}$, and do not make this explicit in the notation. For $y\in\VV_n$ and $u>0$ set
\be
h^{u}_n(y)=\sum_{x\in\VV_n}p_n(y,x)\exp\{-u/\g_n(x)\},
\Eq(4.8)
\ee
and, writing $k_n(t):=\lfloor a_n t\rfloor$, define
\bea
\nu_n^{J,t}(u,\infty)
&=& \sum_{j=1}^{k_n(t)}h^{u}_n(J_n(j-1)),
\Eq(3.10)
\\
\s_n^{J,t}(u,\infty)
&=&
\sum_{j=1}^{k_n(t)}\left[h^{u}_n(J_n(j-1))\right]^2.
\Eq(3.11)
\eea
Further set, for $u\in(0,\infty)$ and $\d>0$ 
\bea
g_\d(u) &= &u\left(1-e^{-\d /u}\right) ,\Eq(F1)\\
\Eq(F2)
f_\d(u) &= &u^2(1-e^{-\d/u})-\d ue^{-\d/u} .
\eea
\smallskip
\noindent{\bf Condition (A0).} For all $u>0$
\be
2^{-n}\sum_{x\in\VV_n}e^{-u/\g_n(x)}=o(1).
\Eq(1.A0')
\ee
\noindent{\bf Condition (A1).}
There exists a $\s$-finite measure $\nu$ on $(0,\infty)$ such that $\nu(u,\infty)$ is continuous, and such that, for all $t>0$ and all $u>0$,
\be
P\left(
\left|
\nu_n^{J,t}(u,\infty)
-t\nu(u,\infty)
\right|
<\e
\right)=1-o(1),\quad\forall\e>0.
\Eq(G1.A1)
\ee

\noindent{\bf Condition (A2).}  For all $u>0$ and all $t>0$,
\be
P\left(
\s_n^{J,t}(u,\infty)<\e
\right)=1-o(1),\quad\forall\e>0.
\Eq(G1.A2)
\ee

\noindent{\bf Condition (A3).}   For all $u>0$ and all $t>0$,
\be \Eq(G1.A3)
 \lim_{\d\to 0}\lim_{n\to\infty}\frac{[a_nt]}{2^n} \sum_{x\in \VV_n} g_\d(\g_n(x))= 0.
\ee

\noindent{\bf Condition (A3').} For all $u>0$ and all $t>0$,
\be
\lim_{\d\to 0}\lim_{n\to\infty}   \frac{[a_nt]}{2^n} \sum_{x\in\mathcal{V}_n}
 f_\d(\g_n(x))= 0.
 \Eq(G1.A3')
\ee

\begin{theorem}\TH(1.theo3)
\item{(i)} Let $\nu$ in Condition (A1) be such that $\int_{(0,\infty)}(1\wedge u)\nu(du)<\infty$. Then, for all sequences  $a_n$ and $c_n$  for which Conditions (A0), (A1),  (A2) and (A3) are satisfied $\P$-almost surely, respectively in $\P$-probability, we have that with respect to the same convergence mode
\be
 S_n\Rightarrow  S 
\Eq(G1.3.theo1.1)
\ee
as $n\rightarrow\infty$, where $S$ is a subordinator with L\'evy measure $\nu$ and zero drift.

\item{(ii)} Let  $\nu(\mbox{d}u)=u^{-2}\mbox{d}u$ in Condition (A1). Then, for all sequences  $a_n$ and $c_n$ for which Conditions (A0), (A1), (A2) and (A3') are satisfied $\P$-almost surely, respectively in $\P$-probability, we have that with respect to the same convergence mode
\be
S_n-M_n\Rightarrow S^{crit}
\Eq(1'.theo4.1)
\ee
as $n\rightarrow\infty$, where $S^{crit}$ is a L\'evy process with L\'evy triple $(0,0,\nu)$.
\end{theorem}

\begin{proof}
Let us first prove the statements of Theorem \thv(1.theo3) for a fixed  realization of the environment.
As in the proof of  Theorem 1.1 of \cite{G12}, we will do this by showing that the conditions of Theorem \thv(1.theo3) imply those of Theorem 4.1 of \cite{DuRe}. We begin with assertion (i). Under the assumption that the measure $\nu$ in Condition (A1) satisfies $\int_{(0,\infty)}(1\wedge u)\nu(du)<\infty$, Conditions (A1) and (A2) are those of Theorem 1.1 of \cite{G12} when the initial distribution is the invariant measure $\pi_n$ and imply, respectively, Conditions (a) and  (b) of Theorem 4.1 of \cite{DuRe}. Moreover in this case Condition (A0) is Condition (A0)  of  Theorem 1.1 of \cite{G12} (with $F=1$ for all $v>0$).
It thus remains to show that Condition (A3) implies Condition (c) of \cite{DuRe}, namely, implies  that
\be\textstyle
\lim_{\d\to 0}\lim_{n\to\infty}\PP\left(\sum_{i=1}^{[a_nt]} Z_{n,i}\1_{\{Z_{n,i}\leq \d\}}>\e\right) = 0
\Eq(1'.theo4.5)
\ee
where
$
Z_{n,i}=\g_n(J_n(i))e_{n,i}
$
(see \eqv(1.1.6) and \eqv(4.prop3.1)).
Now by a first order Tchebychev inequality,
\be
\nonumber
 \textstyle
\PP\left(\sum_{i=1}^{[a_nt]} Z_{n,i}\1_{\{Z_{n,i}\leq \d\}}>\e\right)\leq \e^{-1} \EE\left(\sum_{i=1}^{[a_nt]} Z_{n,i}\1_{\{Z_{n,i}\leq \d\}}\right)
=  \frac{[a_nt]}{2^n} \sum_{x\in \VV_n} g_\d(\g_n(x)).
\ee
Thus Condition (A3) yields Condition (c) of \cite{DuRe}. This completes the proof of Assertion (i) for fixed realization of the environment. The proof of  Assertion (ii) follows the same pattern with  Condition (A3') substituted for  Condition (A3). Let us establish that, under the assumption that  $\nu(\mbox{d}u)=u^{-2}\mbox{d}u$, Condition (A3') implies Condition (d) of Theorem 4.1 of \cite{DuRe}, which then implies Condition (c). To this end we must establish that, setting
\be
\Eq(1'.theo4.3)
\overline{Z}_{n,i}^\d = Z_{n,i}\1_{\left\{Z_{n,i}\le \d\right\}}
 - \EE \left(Z_{n,i}\1_{\left\{Z_{n,i}\le \d\right\}}\left \vert \right.\mathcal{F}_{n,i-1}\right)
\ee
where $\FF_{n,i-1}=\s\left(e_{n,1},\dots,e_{n,i-1},J_n(1),\dots,J_n(i-1)\right)$,
we have
\be
\Eq(1'.theo4.2)\textstyle
\lim_{\d\to 0}\lim_{n\to\infty}\mathcal{P} \Big(\sum_{i=1}^{[a_nt]} \EE 
\Bigl(\big( \overline{Z}_{n,i}^\d\big)^2\Big\vert \mathcal{F}_{n,i-1}\Bigr)>\epsilon\Big)  =0. 
\ee
By a first order Tchebychev inequality the probability in \eqv(1'.theo4.2) is bounded above by
\be
\Eq(1'.theo4.4)
\textstyle 
\e^{-1} \sum_{i=1}^{[a_nt]}\EE \left(Z_{n,i}\1_{\left\{Z_{n,i}\le \d\right\}}\right)^2   = 
2^{-n} \sum_{x\in\mathcal{V}_n}f_{\d}(\g_n(x)).
\ee
To make use of Theorem 4.1 of \cite{DuRe} we lastly have to check that  (A1) implies that as $n\to\infty$
\be
\Eq(2.Thm1.1)
\sum_{i=1}^{[a_nt]}  \EE _{\pi_n} \left(Z_{n,i}\1_{\{\d<Z_{n,i}<\g\}}\left \vert \mathcal{F}_{n,i-1}\right.\right)
\to
t\int_\d^1 x \nu(\mbox{d}x)\quad \mbox{ in }\mathcal{P}\mbox{-probability}.
\ee
This can be shown as proposed in the proof of Theorem 4.1 in \cite{DuRe} using a Riemann sum argument. Let $k\in\mathbb{N}$. Taking an equidistant partition $t_0,\dots,t_k$ of $[\d,1]$ one can bound $Z_{n,i}$ in the following way:
\be
\Eq(2.Thm1.3)
\sum_{j=0}^{k-1}t_j\1_{\{t_j\leq Z_{n,i}<t_{j+1}\}} 
\leq Z_{n,i} 
\leq \sum_{j=0}^{k-1}t_{j+1}\1_{\{t_j\leq Z_{n,i}<t_{j+1}\}} .
\ee
We now take conditional expectations w.r.t.~$\FF_{n,i-1}$ and use Condition (A1). 
This completes the proof of Assertion (ii) for fixed realization of the environment.
Arguing as in the proof of Theorem 1.1 in \cite{G12} we conclude that Assertion (i), respectively Assertion (ii), is valid $\P$-almost surely (respectively, in $\mathbb{P}$-probability) whenever Conditions (A0), (A1), (A2) and (A3), respectively Condition (A3'), are valid $\P$-almost surely (respectively, in $\mathbb{P}$-probability).
This completes the proof of Theorem \thv(1.theo3).
\end{proof}

Sections  \thv(S5), \thv(S6)  and \thv(S7)  to come are devoted to the verification of the Conditions given in Theorem \thv(1.theo3) for intermediate time-scales. They are structured as follows.  Conditions (A1) and (A2) which are of a similar nature are grouped together. To verify them a two step argument is needed. Firstly, we establish ergodic theorems to substitute chain dependent quantities by chain independent ones. This is done in Section \thv(S5). In Section \thv(S6) we then show concentration of the chain independent quantities with respect to the random environment and, finally, verify Conditions (A1) and (A2). All remaining Conditions are verified in Section \thv(S7). Extreme scales are treated separately in Section 9.

\section{Properties of the landscape.}
\label{S3}

In this section we establish the needed properties of the re-scaled landscape variables $(\g_n(x), x\in\VV_n)$ of \eqv(4.prop3.1).
We assume that $0<\b<\infty$ is fixed, and as before, drop all dependence on $\b$ in the notation. For $u\geq 0$ set $G_n(u)=\P(\tau_n(x)>u)$. Since this is a continuous monotone decreasing function, it has a well defined inverse
$
G_n^{-1}(u):=\inf\{y\geq 0 : G_n(y)\leq u\}
$.
For $ v\geq 0$ set
\be
h_n(v)=a_nG_n(c_nv).
\Eq(2.13)
\ee

\begin{lemma}{\TH(2.lemma6)} 
Let $c_n$ be any of the time-scales of Definition \thv(1.def1).

\item{(i)} For each fixed $\zeta>0$ and all $n$ sufficiently large so that $\zeta>c_n^{-1}$,
the following holds:
for all $v$ such that $\zeta\leq v<\infty$,
\be
h_n(v)= v^{-\a_n}(1+o(1)),
\Eq(2.lem6.1)
\ee
where $0\leq \a_n=\a(\varepsilon)+o(1)$.

\item{(ii)} Let $0<\d<1$.
Then, for all $v$ such that $c_n^{-\d}\leq v\leq 1$ and all large enough $n$,
\be
v^{-\a_n}(1+o(1))
\leq h_n(v)\leq
\sfrac{1}{1-\d}v^{-\a_n(1-\frac{\d}{2})}(1+o(1)),
\Eq(6.lem11.1)
\ee
where $\a_n$ is as before.
\end{lemma}

Next, for $u\geq 0$ set
\be
g_n(u)=c_n^{-1}G_n^{-1}(u/a_n).
\Eq(2.12)
\ee
Clearly $g_n(v)=h_n^{-1}(v)$. Clearly also both $g_n$ and $h_n$ are continuous monotone decreasing functions.
The following lemma is tailored to deal with the case of extreme time-scales. Recall that $\a\equiv\a(1)$.

\begin{lemma}{\TH(2.lemma4)}
Let $c_n$ be an extreme time-scale.

\item{(i)} For each fixed $u>0$, for any sequence $u_n$ such that
$|u_n-u|\rightarrow 0$  as $n\rightarrow\infty$,
\be
g_n(u_n)\rightarrow u^{-(1/\a)},\quad n\rightarrow\infty.
\Eq(2.lem4.1)
\ee
\item{(ii)} There exists a constant $0<C<\infty$
such that, for all $n$ large enough,
\be
g_n(u)\leq u^{-1/\a}C,\quad 1\leq u\leq a_n(1-\Phi(1/(\b\sqrt n))),
\Eq(2.lem4.2)
\ee
where $\Phi$ denotes the standard Gaussian distribution function .
\end{lemma}

The proofs of these two lemmata rely on Lemma \thv(2.lemma3) below. 
Denote by $\Phi$ and $\phi$ the standard Gaussian distribution function and density, respectively.
Let $B_n$ be defined through
\be
a_n\frac{\phi(B_n)}{B_n}=1,
\Eq(2.7)
\ee
and set
$
A_n=B_n^{-1}
$

\begin{lemma}{\TH(2.lemma3)}
Let $c_n$ be any time-scale.
Let $\wt B_n$ be a sequence such that, as $n\rightarrow\infty,$
\bea
\Eq(2.9)
\delta_n&:=&(\wt B_n-B_n)/A_n\rightarrow 0.
\eea
Then, for all $x$ such that $A_n x+ \wt B_n>0$ for large enough $n$,
\be
a_n(1-\Phi(A_n x+ \wt B_n))=\frac{\exp\left(-x\left[1+\sfrac{1}{2}A_n^2x\right]\right)}
{1+A_n^2x}
\left\{1+\OO\left(\delta_n[1+A_n^2+A_n^2x]\right)+\OO(A_n^2)\right\}.
\nonumber
\ee
\end{lemma}
\begin{proof}The lemma is a direct consequence of the following expressions, valid for all $x>0$,
\be\Eq(2.11)
1-\Phi(x) = x^{-1}\phi(x)-r(x) =x^{-1}(1-x^{-2})\phi(x)-s(x),
\ee
where
$
0<r(x)< x^{-3}\phi(x)
$
and 
$
0<s(x)< x^{-5}\phi(x)
$
(see e.g.~\cite{AS}, p.~932). 
\end{proof}

We now prove Lemmata \thv(2.lemma6) and \thv(2.lemma4), beginning with Lemma \thv(2.lemma6).

\begin{proof}[Proof of Lemma \thv(2.lemma6)]
By definition of $G_n$ we may write
\be
h_n(v)
=a_n\Bigl(1-\Phi\left(A_n\log(v^{\a_n})+\overline B_n\right)\Bigr),
\Eq(2.14)
\ee
where
\be
\overline B_n=(\b\sqrt n)^{-1} \log c_n, \quad\a_n=(\b\sqrt n)^{-1}B_n.
\Eq(2.15)
\ee

We first claim that for $v\geq c_n^{-1}$, which guarantees that $A_n\log(v^{\a_n})+\overline B_n>0$,  the  sequence $\overline B_n$
satisfies the assumptions of Lemma \thv(2.lemma3).
For this we use the know fact that the sequence $\wh B_n$ defined by
\be
\wh B_n=(2\log a_n)^{\frac{1}{2}}-\sfrac{1}{2}(\log\log a_n +\log 4\pi)/(2\log a_n)^{\frac{1}{2}},
\Eq(2.18)
\ee
satisfies
\be
(\wh B_n-B_n)/A_n=\OO\left(1/\sqrt{\log a_n}\right)
\Eq(2.19)
\ee
(see \cite{H}, p. 434, paragraph containing Eq. (4)).
By \eqv(2.11) we easily get that
\be
a_n\bigl(1-\Phi(\wh B_n)\bigr)=1-(\log\log a_n)^2(16\log a_n)^{-1}(1+o(1)),
\Eq(2.20)
\ee
whereas, by definition of $a_n$ (see \eqv(1.1')),
\be
a_n\bigl(1-\Phi(\overline B_n)\bigr)=1.
\Eq(2.21)
\ee
Since $\Phi$ is monotone and increasing, \eqv(2.20) and \eqv(2.21) imply that $\wh B_n>\overline B_n$.
Thus
\be
\bigl(1-\Phi(\overline B_n)\bigr)-\bigl(1-\Phi(\wh B_n)\bigr)
=\Phi(\wh B_n)-\Phi(\overline B_n)
\geq\phi(\wh B_n)(\wh B_n-\overline B_n)\geq 0.
\Eq(2.22)
\ee
This, together with \eqv(2.20) and \eqv(2.21), yields
\bea\Eq(2.23)
0<\wh B_n-\overline B_n
&<&\bigl[a_n\phi(\wh B_n)\bigr]^{-1}(\log\log a_n)^2(16\log a_n)^{-1}(1+o(1)).
\eea
Now, by \eqv(2.7),
\be
a_n\phi(\wh B_n)=B_n\bigl[\phi(\wh B_n)/\phi(B_n)\bigr]
=B_n\exp\bigl\{-\sfrac{1}{2}(\wh B_n-B_n)(\wh B_n+ B_n)\bigr\}
\leq B_n(1+o(1)),
\Eq(2.24)
\ee
where the final inequality follows from \eqv(2.19).
Finally, combining \eqv(2.23) and \eqv(2.24) yields
$
0<(\wh B_n-\overline B_n)/A_n=\OO\left((\log\log a_n)^2(16\log a_n)^{-1}\right)
$,
and using again \eqv(2.19), we obtain
$
\delta_n=(\overline B_n-B_n)/A_n=\OO\left((\log\log a_n)^2(16\log a_n)^{-1}+1/\sqrt{\log a_n}\right)
$,
which was the claim. To control the behavior of the sequences $A_n$, $\a_n$, and $c_n$, we will need an expression for the
(of course well known) solution $B_n$ of \eqv(2.7). Here is one (\cite{Cr}, p. 374):
\be
B_n=(2\log a_n)^{\frac{1}{2}}-\sfrac{1}{2}(\log\log a_n +\log 4\pi)/(2\log a_n)^{\frac{1}{2}}+\OO((\log a_n)^{-1}).
\Eq(2.16)
\ee
Note that so far we did not make use of the assumption on $c_n$:
using \eqv(2.15), \eqv(2.16), the fact that
$
2\log a_n=(2\log2)(\log a_n/n\log2)=\b_c^2(\log a_n/n\log2)n
$,
and the just established fact that
$
(\overline B_n-B_n)/A_n\rightarrow 0
$, we obtain,
for intermediate time-scales,
\bea
\Eq(2.23'')
\log a_n&=&\sfrac{1}{2}\b^2_c(\varepsilon)n(1+o(1)),
\\
\frac{\log c_n}{\b\sqrt{n}}&=&(2\log a_n)^{\frac{1}{2}}
-\sfrac{1}{2}(\log\log a_n +\log 4\pi)/(2\log a_n)^{\frac{1}{2}}+\OO((\log a_n)^{-1}),\quad
\Eq(2.23''.2)
\\
\a_n&=&(\sqrt{n}\b)^{-1}B_n=\a(\varepsilon)(1+o(1)).
\nonumber
\eea
Finally for extreme time-scales, writing $\b_c(1)\equiv\b_c$, we have that
$
2\log a_n=\b_c^2n(1-C/n)
$
for some constant $0<C<\infty$. Thus, instead of \eqv(2.23''), we get:
\bea
\Eq(2.23')
\log a_n&=&\sfrac{1}{2}\b^2_cn(1-C/n),
\nonumber\\
\log c_n&=&\b\b_c n(1-o(1)),\\
\a_n\leq \a&\text{and}&
\a_n=\a(1-o(1)).
\nonumber
\eea
We are now equipped to prove Lemma \thv(2.lemma6).
By Lemma \thv(2.lemma3), for all  $v\geq c_n^{-1}$, setting 
$R_n=\OO(A_n^2)+\OO\left(\delta_n[1+A_n^2+A_n^2\a_n\log v]\right)$,
\be
\Eq(2.24')
h_n(v)=\frac{\exp\left(-\a_n\log v\left[1+\sfrac{1}{2}A_n^2\a_n\log v\right]\right)}
{1+A_n^2\a_n\log v}\left\{ 1+R_n\right\},
\ee
where $\delta_n\downarrow 0$ as $n\uparrow\infty$. 
Plugging in the explicit form of $\a_n$, $A_n$ and $\delta_n$ we get
\be
\Eq(2.24'')
h_n(v) =\left(1+\frac{B_n^{\a_n-1}}{\b\sqrt{n}}\log v \right)^{-1}
v^{-B_n^{\a_n}/\b\sqrt{n}}e^{\log v/2n\b^2}\left\{1+\OO\left(\delta_n \right)\right\}.
\ee
Therefore, for each fixed $0< v<\infty$,
and all large enough $n$ so that $v> c_n^{-1}$,
\be
h_n(v)=v^{-\a_n}(1+o(1)).
\Eq(2.25bis)
\ee
This together with \eqv(2.23'') proves assertion (i) of the lemma.
To prove  assertion (ii) note that by \eqv(2.15), since $A_n=B_n^{-1}$,
$
A_n^2\a_n=\frac{1}{\log c_n}\frac{\overline B_n}{B_n}
$
where $\frac{\overline B_n}{B_n}=1+o(1)$ (see the paragraph following \eqv(2.24)).
Thus, for all $v$ satisfying $c_n^{-\d}\leq v\leq 1$, we have
\be
-\d\sfrac{\overline B_n}{B_n}\leq A_n^2\a_n\log v\leq 0.
\Eq(2.25ter)
\ee
Combining this and \eqv(2.24') immediately yields the bounds \eqv(6.lem11.1).
The proof of Lemma \thv(2.lemma6) is now done.
\end{proof}

\begin{proof}[Proof of Lemma \thv(2.lemma4)]
Up until \eqv(2.24') we proceed exactly as in the proof of Lemma \thv(2.lemma6).
Now, by \eqv(2.24'), for each fixed $0\leq v<\infty$, any sequence $v_n$ such that $|v_n-v|\rightarrow 0$, and all large enough $n$ (so that $v> c_n^{-1}$),
\be
h_n(v_n)=v_n^{-\a_n}(1+o(1))= v^{-\a(1-o(1))}(1+o(1)).
\Eq(2.25)
\ee
This and the relation $g_n(v)=h_n^{-1}(v)$  imply that for each fixed $0<u<\infty$,
any sequence $u_n$ such that $|u_n-u|\rightarrow 0$, and all large enough $n$ (so that $u<h_n(c_n^{-1})$),
\be
g_n(u_n)=u_n^{-(1/\a_n)}(1+o(1))= u^{-(1/\a)(1+o(1))}(1+o(1)),
\Eq(2.26)
\ee
which is tantamount to assertion (i) of the lemma.

To prove assertion (ii) assume that $c_n^{-1}\leq v\leq 1$.
Recall that $h_n$ is a monotonous function so that
if $h_n(v)=g_n^{-1}(v)$ for all $c_n^{-1}\leq v\leq 1$,
then $g_n(u)=h_n^{-1}(u)$ for all $h_n(1)\leq u\leq h_n(c_n^{-1})$.
Now $h_n(1)=a_nG_n(c_n)=1$, as follows from \eqv(1.1'), and
$h_n(c_n^{-1})=a_nG_n(1)=a_n(1-\Phi(1/(\b\sqrt n)))$.
Observe next that
$c_n^{-1}\leq v\leq 1$ is equivalent to
$
-1\leq A_n^2\log v^{\a_n} \leq 0
$.
Therefore, by \eqv(2.24'), for large enough $n$,
\be
h_n(v)\geq (1-2\delta_n)v^{-\a_n},\quad c_n^{-1}\leq v\leq 1.
\Eq(2.27)
\ee
By monotonicity of $h_n$,
\be
g_n(u)=h_n^{-1}(u)\leq (1-2\delta_n)^{1/\a_n} u^{-1/\a_n},\quad 1\leq u\leq a_n(1-\Phi(1/(\b\sqrt n))).
\Eq(2.28)
\ee
From this and the fact that $\a_n\leq \a$ (see \eqv(2.23')), \eqv(2.lem4.2) is readily obtained.
This concludes the proof of the lemma.
\end{proof}

\begin{remark} 
We see from the proof of Lemma \thv(2.lemma4) that the lemma holds true not only for extreme scales,
but for intermediate scales also provided one replaces $\a$ by $\a(\varepsilon)$ everywhere.
\end{remark}

\section{ The jump chain: some estimates.}
\label{S4}

In this section we gather the specific properties of the jump chain $J_n$ (i.e.~the simple random walk) that will be needed later to reduce Condition (A1) and Condition (A2) of Theorem \thv(1.theo3)  to  conditions that are independent from $J_n$.
Proposition \thv(3.prop0) below and its Corollary \thv(3.cor1) are central to this scheme. They will allow us to substitute the measures $\pi^\pm_n(x)$ of \eqv(3.7) for the jump chain after $\theta_n\sim n^2$ steps have been taken.

The fact that the chain $J_n$ is periodic with period two introduces a number of small complications. Let us fix the notation.
Denote by ${\bf 1}$ the vertex of $\VV_n$ whose coordinates are identically 1. Write $\VV_n\equiv \VV_n^{-}\cup\VV_n^{+}$ where $\VV_n^{-}$ and $\VV_n^{+}$ are, respectively, the subsets  of vertices that are at odd and even distance of the vertex ${\bf 1}$.
To each of these subsets we associate a chain, $J_n^{-}$ and $J_n^{+}$, obtained by observing $J_n$ at its visits to $\VV_n^{-}$ and $\VV_n^{+}$, respectively. Specifically, denoting by ${\pm}$ either of the symbols $-$ or $+$,  $(J_n^{\pm}(k)\,,k\in\N)$  is the chain on $\VV^{\pm}_n$ with transition probabilities
\be
p_n^{\pm}(x,y)=P(J_n(i+2)=y\mid J_n(i)=x)
\quad \mbox{if }\, x\in\VV^{\pm}_n, y\in\VV^{\pm}_n,
\Eq(3.6)
\ee
and $p_n^{\pm}(x,y)=0$ else. Clearly $J_n^{\pm}$ is aperiodic, reversible, and has a unique reversible invariant measure 
$\pi_n^{\pm}$ given by
\be
\pi^\pm_n(x)=2^{-n+1}, \quad x\in\VV_n^{\pm}.
\Eq(3.7)
\ee
Denote by $P^{\pm}_x$ the law of $J^{\pm}_n$ started in $x$ and set
\be
\theta_n=2
\left\lceil
\sfrac{3}{2}(n-1)\log 2/\left|\log\left(1-\sfrac{2}{n}\right)\right|
\right\rceil.
\Eq(3.prop1.1)
\ee

\begin{proposition}
\TH(3.prop0)
There exists a positive decreasing sequence $\delta_n$, satisfying $|\delta_n|\leq 2^{-n}$, such that  for all $x\in\VV^{\pm}_n$ and $y\in\VV_n$, all $l\geq l+\theta_n/2$, and  large enough  $n$,
\be
P^{\pm}_x\left(J^{\pm}_n\left(l\right)=y\right)=(1+\delta_n)\pi^{\pm}_n(y).
\Eq(3.prop0.1)
\ee
\end{proposition}

As an immediate consequence of Proposition \thv(3.prop0), we have the

\begin{corollary}
\TH(3.cor1) 
Let $\theta_n$ and $\delta_n$ be as in Proposition \thv(3.prop0).
Then, for all $x\in\VV_n$ and $y\in\VV_n$,  all $i\geq 0$, and large enough  $n$, the following holds:
\be
\sum_{k=0}^{1}P_{x}\left(J_n(i+\theta_n+k)=y\right)=2(1+\delta_n)\pi_n(y).
\Eq(3.prop1.2)
\ee
\end{corollary}

The next two propositions bound, respectively,
the expected number of returns and visits to a given vertex.
Let $p_n^{l}(\cdot,\cdot)$ denote the $l$ steps transition probabilities of $J_n$ and let $\dist(\cdot,\cdot)$ denote Hamming's distance
\be\Eq(graph.dist)
\dist(x,x')\equiv\frac 12 \sum_{i=1}^n |x_i-x'_i|.
\ee

\begin{proposition} 
\TH(3.prop3)
There exists a numerical constant $0<c<\infty$ such that for all $m\leq n^2$,
\be
\sum_{l=1}^{2m}p_n^{l+2}(z,z)\leq \frac{c}{n^2}\,,\quad\forall z\in\VV_n,
\Eq(3.prop3.1)
\ee
\end{proposition}

\begin{proposition} 
\TH(3.prop4)
There exists a numerical constant $0<c<\infty$ such that for all $m\leq n^2$, for all pairs of distinct vertices $y,z\in\VV_n$
satisfying $\dist(y,z)=\frac{n}{2}(1-o(1))$,
\be
\sum_{l=1}^{2m}p_n^{l+2}(y,z)\leq e^{-cn}\,,
\Eq(3.prop4.1)
\ee
\end{proposition}

We now prove the above results in the order in which they appear.

\begin{proof}[Proof of Proposition \thv(3.prop0)] 
The proof relies on a well know bound by Diaconis et al  \cite{DS} that relates the rate of convergence to stationarity of $J^{\pm}_n$
to the eigenvalues  of the transition matrix
$
Q^{\pm}=\left(p^{\pm}_n(x,y)\right) _{\VV^{\pm}_n\times\VV^{\pm}_n}
$.
First notice that (i) the eigenvalues of the transition matrix
$
Q=\left(p_n(x,y)\right)_{\VV_n\times\VV_n}
$ 
of $J_n$ are $1-2j/n$, $0\leq j\leq n$, (see, for example, \cite{DS} example 2.2 p.~45);
(ii) that by \eqv(3.6), with obvious notation, $Q^2=Q^{+}+Q^{-}$ and  $Q^{+}Q^{-}=Q^{-}Q^{+}=0$;
(iii) and that  $Q^{+}$ and $Q^{-}$ can be obtained from one another by permutation of their
rows and columns. Now it follows from (iii) that $Q^{+}$ and $Q^{-}$ must have the same eigenvalues.
This fact combined with (i) and (ii) imply that these eigenvalues coincide with those of $Q^2$,
so that using  (i) we conclude that both $Q^{+}$ and $Q^{-}$ have eigenvalues 
$\left(1-2\frac{j}{n}\right)^2$, $0\leq j\leq \lfloor \frac{n}{2}\rfloor$.

Since $Q^{\pm}$ is irreducible we may apply (1.9) of  Proposition 3 in \cite{DS} to the chain  $J^{\pm}_n$ with $\b_*=\left(1-\frac{2}{n}\right)^2$ and time (denoted $n$ therein)
$\theta_n/2
=\left\lceil
\frac{3}{2}(n-1)\log 2/\left|\log\left(1-\frac{2}{n}\right)\right|
\right\rceil
$.
This yields
$
P^{\pm}_x\left(J^{\pm}_n(l)=y\right)=(1+\delta_n)\pi^{\pm}_n(y)
$
where $\delta_n^2\leq \frac{1}{4}2^{3(n-1)}\left(1-\frac{2}{n}\right)^{2\theta_n}\leq 2^{-3n+1}$
for all $n$ large enough, and thus $|\delta_n|\leq  2^{-n}$.
The proposition is proven.
\end{proof}

\begin{proof}[Proof of Corollary \thv(3.cor1) ]
We prove \eqv(3.prop1.2) first.
Assume, without loss of generality, that $i+\theta_n$ is even and set $i+\theta_n=2l$.
Then, 
\be
\Delta
\equiv
\textstyle
\sum_{k=0}^{1}P_{x}\left(J_n(2l+k)=y\right)
=P_x\left(J_n(2l)=y \right)+\frac{1}{n}\sum_{z\sim x}P_z\left(J_n(2l)=y\right)
\Eq(3.prop.3)
\ee
where the sum is over all nearest neighbourgs $z$ of $x$. Thus, by Proposition \thv(3.prop0),
\be
\nonumber
\Delta
=(1+\delta_n)\Bigl[
\pi_n^{+}(y)\1_{\{x\in\VV_n^{+}\}}+\pi_n^{-}(y)\1_{\{x\in\VV_n^{-}\}}
+\pi_n^{-}(y)\1_{\{x\in\VV_n^{+}\}}+\pi_n^{+}(y)\1_{\{x\in\VV_n^{-}\}}
\Bigr].
\ee
Now, only one of the two indicator functions in the right hand side above is non zero so that
by \eqv(3.7),
$
\Delta=(1+\delta_n)2\pi_n(y)
$,
yielding \eqv(3.prop1.2). 
\end{proof}

We now prove Proposition \thv(3.prop3) and Proposition \thv(3.prop4).

\begin{proof}[Proof of Proposition \thv(3.prop3)]
Consider the Ehrenfest chain on state space $\{0,\dots,2n\}$
with one step transition probabilities $r_n(i,i+1)=\frac{i}{2n}$ and $r_n(i, i-1)=1-\frac{i}{2n}$.
Denote by $r_n^{l}(\cdot,\cdot)$ its $l$ steps transition probabilities.
It is well known (see e.g. \cite{BG08}) that
$
p_n^{l}(z,z)=r_n^{l}(0,0)
$
for all $l\geq 0$ and all $z\in\VV_n$.
Hence
$
\sum_{l=1}^{2m}p_n^{l+2}(z,z)=\sum_{l=1}^{2m}r_n^{l+2}(0,0)
$.
It is in turn well known (see \cite{Kem}, p.~25, formula (4.18)) that
\be
r_n^{l}(0,0)=2^{-2n}\sum_{k=0}^{2n}\binom{2n}{k}\left(1-\frac{k}{n}\right)^l\,,\quad l\geq 1\,.
\Eq(3.prop3.2)
\ee
Note that by symmetry, $r_n^{2l+1}(0,0)=0$. Simple calculations yield $r_n^{4}(0,0)=\frac{c_2}{n^2}$, $r_n^{6}(0,0)=\frac{c_3}{n^3}$, and $r_n^{8}(0,0)=\frac{c_4}{n^4}$, for some constants $0<c_i<\infty$, $2\leq i\leq 4$. Thus, if $m\leq 3$,
$
\sum_{l=1}^{2m}r_n^{l+2}(0,0)\leq \frac{c}{n^2}
$
for some numerical constant $0<c<\infty$. If now $m>3$, write
$
\sum_{l=1}^{2m}r_n^{l+2}(0,0)
=r_n^{4}(0,0)+r_n^{6}(0,0)+\sum_{l=6}^{2m}r_n^{l+2}(0,0)
$,
and use that by \eqv(3.prop3.2),
\be
\nonumber
\sum_{l=6}^{2m}r_n^{l+2}(0,0)
=
2^{-2n}\sum_{k=0}^{2n}\binom{2n}{k}\sum_{l=6}^{2m}\left(1-\sfrac{k}{n}\right)^{l+2}
\leq
2^{-2n}\sum_{k=0}^{2n}\binom{2n}{k}\left(1-\sfrac{k}{n}\right)^{8}\sum_{j=0}^{m-1}\left(1-\sfrac{k}{n}\right)^j\,.
\ee
Since $|1-\sfrac{k}{n}|\leq 1$,
$
\sum_{l=6}^{2m}r_n^{l+2}(0,0)
\leq 2^{-2n}\sum_{k=0}^{2n}\binom{2n}{k}\left(1-\sfrac{k}{n}\right)^{8}m
=mr_n^{8}(0,0)
\leq n^2\sfrac{c_4}{n^4}
$,
so that
$
\sum_{l=1}^{2m}r_n^{l+2}(0,0)
\leq \frac{c_2}{n^2}+\frac{c_3}{n^3}+n^2\sfrac{c_4}{n^4}
\leq \sfrac{c}{n^2}
$
for some numerical constant $0<c<\infty$. The lemma is proven.
\end{proof}

\begin{proof}[Proof of Proposition \thv(3.prop4)]
This estimate is proved using  a $d$-dimensional version of the Ehrenfest scheme known as
the lumping procedure, and studied e.g.~in \cite{BG08}.
In what follows we mostly stick to the notations of \cite{BG08}, hoping that this will create
no confusion. Without loss of generality we may take $y\equiv 1$
to be the vertex whose coordinates all take the value 1. Let $\g^{\L}$ be the map
(1.7) of \cite{BG08} derived from the partition of $\L\equiv\{1,\dots,n\}$ into
$d=2$ classes, $\L=\L_1\cup\L_2$, defined through the relation:
$i\in\L_1$ if the $i^{th}$ coordinate of $z$ is 1, and $i\in\L_2$ otherwise.
 The resulting lumped chain $X^{\L}_n\equiv\g^{\L}(J_n)$
has range $\G_{n,2}=\g^{\L}(\VV_n)\subset[-1,1]^2$. Note that the vertices
1 and $y$ of $\VV_n$ are mapped respectively on the corners $1\equiv(1,1)$ and $x\equiv(1,-1)$ of $[-1,1]^2$.
Without loss of generality we may assume that $0\in\G_{n,2}$.
Now, denoting by $\P^{\circ}$ the law of $X^{\L}_n$, we have,
$
p_n^{l+2}(y,z)=\P^{\circ}(X^{\L}_n(l+2)=x \mid X^{\L}_n(0)=1)
$.
Let
$
\t^{x'}_x=\inf\{ k>0 \mid X^{\L}_n(0)=x', X^{\L}_n(k)=x\}
$.
Starting from 1, the lumped chain may visit 0 before it visits x or not. Thus
$
p_n^{l+2}(1,z)
=\P^{\circ}(X^{\L}_n(l+2)=x, \t^{1}_0< \t^{1}_x)+\P^{\circ}(X^{\L}_n(l+2)=x, \t^{1}_0\geq \t^{1}_x)
$.
On the one hand, using Theorem 3.2 of \cite{BG08},
$
\P^{\circ}(X^{\L}_n(l+2)=x, \t^{1}_0\geq \t^{1}_x)
\leq \P^{\circ}(\t^{1}_x\leq \t^{1}_0)
\leq e^{-c_1n}
$
for some constant $0<c_1<\infty$.
On the other hand, conditioning on the time of the last return to 0 before time $l+2$,
and bounding the probability of the latter event by 1, we readily get
\be
\textstyle
\P^{\circ}(X^{\L}_n(l+2)=x, \t^{1}_0< \t^{1}_x)\
\leq (l+2)\P^{\circ}(\t^{0}_x< \t^{0}_0)
=(l+2)\frac{\Q_n(x)}{\Q_n(0)}\P^{\circ}(\t^{x}_0< \t^{x}_x)\,,
\Eq(3.prop4.2)
\ee
where the last line follows from reversibility, and where $\Q_n$, defined in
Lemma 2.2 of \cite{BG08}, denotes the invariant measure of $X^{\L}_n$. Since
$\P^{\circ}(\t^{x}_0< \t^{x}_x)\leq 1$ we are left to estimate the ratio
of invariant masses in \eqv(3.prop4.2). From the assumption that
$\dist(y,z)=\frac{n}{2}(1-o(1))$, it follows
that $\L_1=n-\L_2=\frac{n}{2}(1-o(1))$. Therefore, by (2.4) of \cite{BG08},
$
\frac{\Q_n(x)}{\Q_n(0)}\leq e^{-c_2n}
$
for some constant $0<c_2<\infty$. Gathering our bounds we arrive at
$
p_n^{l+2}(1,z)\leq e^{-c_1n}+(l+2)e^{-c_2n}
$,
which proves the claim of the lemma.
\end{proof}

\section{Ergodic theorems}
\label{S5}

We now capitalize on the estimates of Section 3 and, as a first step towards the verification of Conditions (A1), (A2) and the control of the centering term $M_n(t)$, prove that these conditions can be replaced by simple ones, where all quantities depending on
the jump chain have been averaged out.  We will deal  with with the centering term $M_n(t)$ and with the quantities $\nu_n^{J,t}(u,\infty)$ and $\s_n^{J,t}(u,\infty)$ entering the statements of Conditions (A1)-(A2) separately. 

\subsection{An ergodic theorem for $\nu_n^{J,t}(u,\infty)$.}
\TH(SG4.1)
Setting
\be
\pi_n^{J,t}(x)={k^{-1}_n(t)}\sum_{j=0}^{k_n(t)-1}\1_{\{J_n(j)=x\}}\,,\quad x\in\VV_n,
\Eq(4.7)
\ee
\eqv(3.10) and  \eqv(3.11) can be rewritten as
\bea
\Eq(4.4)
\nu_n^{J,t}(u,\infty)
&\equiv& k_n(t)\sum_{y\in\VV_n}\pi_n^{J,t}(y)h^{u}_n(y),
\\
\s_n^{J,t}(u,\infty)
&\equiv&
k_n(t) \sum_{y\in\VV_n}\pi_n^{J,t}(y)\left(h^{u}_n(y)\right)^2.
\Eq(4.4')
\eea
and by \eqv(G3.lem1.1') of Lemma \thv(G3.lem1) (stated below),
\bea
&&E_{\pi_n}\left[\nu_n^{J,t}(u,\infty)\right]=k_n(t)\sum_{x\in \VV_n}\pi_n(x)h^{u}_n(x)\equiv ({k_n(t)}/{a_n})\nu_n(u,\infty),
\Eq(4.9')
\\
&&E_{\pi_n}\left[\s_n^{J,t}(u,\infty)\right]= k_n(t)\sum_{x\in \VV_n}\pi_n(x)\left(h^{u}_n(x)\right)^2\equiv({k_n(t)}/{a_n})\s_n(u,\infty).
\Eq(4.9'')
\eea

\begin{proposition}
\TH(4.prop1)
Let $\rho_n>0$ be a decreasing sequence satisfying $\rho_n\downarrow 0$ as $n\uparrow\infty$.
There exists a sequence of subsets $\O^{\tau}_{n,0}\subset\O^{\tau}$ with
$
\P\left((\O^{\tau}_{n,0})^c\right)<\frac{\theta_n}{\rho_na_n}
$
and such that on $\O^{\tau}_{n,0}$, the following holds for all large enough $n$: for all $t>0$, all $u>0$, and all $\e>0$,
\be
P_{\pi^\pm_n}\left(\left|\nu_n^{J,t}(u,\infty)-(k_n(t)/a_n)\nu_n(u,\infty)\right|\geq\e\right)
\leq
\e^{-2}[t\Theta^1_n(u)+t^2\Theta^2_n(u)]
\Eq(4.prop1.1)
\ee
 where, for some constant $0<c<\infty$,
\bea
\Theta^1_n(u)
&=&
\s_n(u,\infty)+2\frac{\nu_n^2(u,\infty)}{a_n}
+c\frac{\nu_n(2u,\infty)}{ n^{2}}
+\rho_n\left[\E\nu_n(u,\infty)\right]^2,
\Eq(4.prop1.2)
\\
\Theta^2_n(u)
&=&\frac{\nu_n^2(u,\infty)}{2^{n-1}}.
\Eq(4.prop1.2')
\eea
In addition, for all $t>0$ and all $u>0$,
\be
P_{\pi^\pm_n}\left(\s_n^{J,t}(u,\infty)\geq\e'\right)\leq\frac{2k_n(t)}{\e'\, a_n}\s_n(u,\infty)\,,\quad\forall\e'>0\,.
\Eq(4.prop1'.2')
\ee
\end{proposition}

We first state and prove the following simple lemma.

\begin{lemma}
\TH(G3.lem1) 
For any function $f$ on $\VV_n$, 
\be
E_{\pi^\pm_n}\sum_{i=0}^{k-1}f(J_n(i))=E_{\pi_n}\sum_{i=0}^{k-1}f(J_n(i))+ r^\pm(k)
\Eq(G3.lem1.1)
\ee
where 
$
 r^\pm(k)=\frac{1}{2}\left[E_{\pi^\pm_n}f(J_n(0))-E_{\pi^\mp_n}f(J_n(0))\right]
$
if $k$ is odd and $ r^\pm(k)=0$ else, and
\be
E_{\pi_n}\sum_{i=0}^{k-1}f(J_n(i))= k E_{\pi_n}f(J_n(0)).
\Eq(G3.lem1.1')
\ee
\end{lemma}

\begin{proof}[Proof of Lemma \thv(G3.lem1)] 
Clearly, by \eqv(3.6)-\eqv(3.7), for all $x\in\VV_n$ and $j\in\N$,
\bea
\Eq(G3.lem1.2)
P_{\pi^{\pm}_n}(J(2j)=x)&=&\pi^{\pm}_n(x),\\
\Eq(G3.lem1.3)
P_{\pi^{\pm}_n}(J(2j+1)=x)&=&\pi^{\mp}_n(x)=2\pi_n(x)-\pi^{\pm}_n(x).
\eea
where the last equality is  $\pi_n=\frac{1}{2}(\pi^+_n+\pi^-_n)$.
Now, if $k=2m+1$ for some $m\geq 1$,
\bea
\nonumber
\textstyle\sum_{i=0}^{k-1}f(J_n(i))&=&\textstyle\sum_{j=0}^{m}f(J_n(2j))+\sum_{j=0}^{m-1}f(J_n(2j+1))\\
\Eq(G3.lem1.4)
&=&\textstyle\sum_{j=0}^{m-1}[f(J_n(2j))+f(J_n(2j+1))]+f(2m),
\eea
and by \eqv(G3.lem1.2)-\eqv(G3.lem1.3)
\bea
\textstyle E_{\pi^\pm_n}\sum_{i=0}^{k-1}f(J_n(i))
&=&
(k-1)\textstyle E_{\pi_n}f(J_n(0))+ E_{\pi^\pm_n}f(J_n(0))
\cr
&=&
k\textstyle E_{\pi_n}f(J_n(0))
+\frac{1}{2}[E_{\pi^\pm_n}f(J_n(0))-E_{\pi^\mp_n}f(J_n(0))]
\quad\quad
\Eq(G3.lem1.5)
\eea
Similarly, we get that if $k=2m$ for some $m\geq 1$, 
\be
\textstyle E_{\pi^\pm_n}\sum_{i=0}^{k-1}f(J_n(i))
=
k\textstyle E_{\pi_n}f(J_n(0)).
\Eq(G3.lem1.6)
\ee
Now, since $\pi_n=\frac{1}{2}(\pi^+_n+\pi^-_n)$,
\be
\textstyle
E_{\pi_n}\sum_{i=0}^{k-1}f(J_n(i))
=\frac{1}{2}\left(E_{\pi_n^+}\sum_{i=0}^{k-1}f(J_n(i))+E_{\pi_n^-}\sum_{i=0}^{k-1}f(J_n(i))\right),
\Eq(G3.lem1.7)
\ee
and inserting \eqv(G3.lem1.5), respectively \eqv(G3.lem1.6), in the r.h.s.~of \eqv(G3.lem1.7) yields \eqv(G3.lem1.1').
Plugging \eqv(G3.lem1.1'), in turn, in \eqv(G3.lem1.5) and \eqv(G3.lem1.6) yields \eqv(G3.lem1.1). 
\end{proof}

\begin{proof}[Proof of Proposition \thv(4.prop1)] 
By \eqv(G3.lem1.1) of Lemma \thv(G3.lem1),
\be
E_{\pi^\pm_n}\left[\s_n^{J,t}(u,\infty)\right]=({k_n(t)}/{a_n})\s_n(u,\infty)(1+ \bar r(k_n(t)))
\Eq(4.10'')
\ee
where
$
|\bar r(k)|
\leq k^{-1}
$
for all $k\geq 1$. The upper bound \eqv(4.prop1'.2') now simply follows from  \eqv(4.10'') by a first order
Tchebychev inequality. The proof of \eqv(4.prop1.1) is a little more involved.
Using a second order Tchebychev inequality together with the expressions
\eqv(4.4) and \eqv(4.9') of $\nu_n^{J,t}(u,\infty)$ and $\nu_n(u,\infty)$
the probability in left hand side of \eqv(4.prop1.1) is bounded above by
\bea
&&\hspace{-10pt}
\textstyle
\e^{-2}
E_{\pi^\pm_n}\Bigl[
k_n(t)\sum_{y\in\VV_n}\left(\pi_n^{J,t}(y)-\pi_n(y)\right)h^{u}_n(y)
\Bigr]^2.
\Eq(4.prop1.6)
\\
=&&\hspace{-10pt}
\e^{-2}
\sum_{x\in\VV_n}\sum_{y\in\VV_n}h^{u}_n(x)h^{u}_n(y)
\left[k^2_n(t)E_{\pi^\pm_n}\left(\pi_n^{J,t}(x)-\pi_n(x)\right)\left(\pi_n^{J,t}(y)-\pi_n(y)\right)\right].
\quad
\Eq(4.prop1.7)
\eea
By \eqv(G3.lem1.1) of Lemma \thv(G3.lem1),
$
E_{\pi^\pm_n}\left[\pi_n^{J,t}(y)\right]
=\pi_n(y)(1+ \bar r_y(k_n(t)))
$
where
$
|\bar r_y(k)|
\leq k^{-1}
$
for all $y\in\VV_n$, $k\geq 1$.
Thus, setting
$
\Delta_{ij}(x,y)=P_{\pi^\pm_n}\left(J_n(i)=x, J_n(j)=y\right)-\pi_n(x)\pi_n(y)
$ 
the quantity appearing in square brackets in \eqv(4.prop1.7) may be expressed as
\be
\sum_{i=0}^{k_n(t)-1}\sum_{j=0}^{k_n(t)-1}\Delta_{ij}(x,y)+k^2_n(t)\pi_n(x)\pi_n(y)  (\bar r_x(k_n(t))+\bar r_y(k_n(t))).
\Eq(4.prop1.9)
\ee
For $\theta_n$ defined in \eqv(3.prop1.1), we now break the sum in the r.h.s. of \eqv(4.prop1.9) into three terms:
\bea
(\overline I)&=&2\sum_{0\leq i\leq k_n(t)-1}\sum_{i+\theta_n\leq j\leq k_n(t)-1}\Delta_{ij}(x,y),\\
(\overline{II})&=&\sum_{0\leq i\leq k_n(t)-1}\1_{\{i=j\}}\Delta_{ij}(x,y),\\
(\overline{III})&=&2\sum_{0\leq i\leq k_n(t)-1}\sum_{i<j<i+\theta_n}\Delta_{ij}(x,y).
\Eq(4.prop1.10)
\eea
Consider first $(\overline I)$. 
For $r\geq 0$ and $s\geq 0$ define
\be
\wt\Delta_{rs}(x,y)
=
\left|
\Delta_{(2s)(2r)}(x,y)+\Delta_{(2s)(2r+1)}(x,y)+\Delta_{(2s+1)(2r)}(x,y)+\Delta_{(2s+1)(2r+1)}(x,y)
\right|,
\nonumber
\ee
and note that by \eqv(3.prop1.2) of Corollary \thv(3.cor1) and \eqv(G3.lem1.2)-\eqv(G3.lem1.3), 
$
\wt\Delta_{rs}(x,y)
\leq 
4|\delta_n|\pi_n(x)\pi_n(y)
$.
Hence
\bea
(\overline I)
&\leq&
2\sum_{0\leq s\leq \lceil (k_n(t)-1)/2\rceil}\sum_{\lfloor(i+\theta_n)/2\rfloor\leq l\leq \lceil (k_n(t)-1)/2\rceil}
\wt\Delta_{rs}(x,y)
\\
&\leq & 2|\delta_n|(k_n(t)+1)^2\pi_n(x)\pi_n(y),
\Eq(4.prop1.10')
\eea
where $|\delta_n|\leq 2^{-n}$. 
Turning to the term $(\overline{II})$, we have,
\bea
(\overline{II})
&=&\sum_{0\leq i\leq k_n(t)-1}\Delta_{ii}(x,x)\1_{\{x=y\}}
\\
&=&\sum_{0\leq i\leq k_n(t)-1}\left[P_{\pi^\pm_n}\left(J_n(i)=x\right)-\pi^2_n(x)\right]\1_{\{x=y\}}
\\
&\leq&(k_n(t)+1)\pi_n(x)(1-\pi_n(x))\1_{\{x=y\}},
\Eq(4.prop1.11)
\eea
where the last line follows from  \eqv(G3.lem1.2)-\eqv(G3.lem1.3). In the same way,
\bea
(\overline{III})
&\leq&
2\sum_{i=0}^{k_n(t)-1}\sum_{l=1}^{\theta_n-1}P_{\pi_n}\left(J_n(i)=x, J_n(i+l)=y\right)
\\
&\leq& 
2\sum_{i=0}^{k_n(t)-1}\sum_{l=1}^{\theta_n-1}
P_{\pi^\pm_n}\left(J_n(i)=x\right)P_{\pi^\pm_n}\left(J_n(i+l)=y\mid J_n(i)=x\right)
\\
&=& 2(k_n(t)+1)\pi_n(x)\sum_{l=1}^{\theta_n-1}p_n^{l}(x,y),
\Eq(4.prop1.12)
\eea
where $p_n^{l}(\cdot,\cdot)$ denote the $l$-steps transition matrix of $J_n$. 
Inserting our bounds on $(\overline{I}),(\overline{II})$, and $(\overline{III})$ 
in \eqv(4.prop1.9), and combining with \eqv(4.prop1.7) we get that, for all $\e>0$,
\be
P_{\pi^\pm_n}\left(\left|\nu_n^{J,t}(u,\infty)-(k_n(t)/a_n)\nu_n(u,\infty)\right|\geq\e\right)
\leq\e^{-2}[(I)+(II)+(III)],
\Eq(4.prop1.13)
\ee
where
\bea
(I)
&=&
2\left(|\delta_n|(k_n(t)+1)^2+\bar r_1k^2_n(t)\right)\sum_{x\in\VV_n}\sum_{y\in\VV_n}h^{u}_n(x)h^{u}_n(y)\pi_n(x)\pi_n(y),\\
(II)&=&
(k_n(t)+1)\sum_{x\in\VV_n}\sum_{y\in\VV_n}h^{u}_n(x)h^{u}_n(y)\pi_n(x)(1-\pi_n(x))\1_{\{x=y\}},\\
(III)
&=&2(k_n(t)+1)
\sum_{x\in\VV_n}\sum_{y\in\VV_n}
h^{u}_n(x)h^{u}_n(y)\pi_n(x)\sum_{l=1}^{\theta_n-1}p_n^{l}(x,y).
\Eq(4.prop1.14)
\eea
By \eqv(4.9') and  \eqv(4.9''),
\bea
(I)&\leq&
22^{-n}\left(\sfrac{k_n(t)+1}{a_n}\right)^2\nu_n^2(u,\infty)+ 2a_n^{-1}\sfrac{k_n(t)+1}{a_n}\nu_n^2(u,\infty)
\Eq(4.prop1.15)
\\
(II)&\leq&
\sfrac{k_n(t)+1}{a_n}\s_n(u,\infty).
\Eq(4.prop1.15')
\eea
To further express the term \eqv(4.prop1.14) note that, by \eqv(4.8),
\be
\sum_{y\in\VV_n}p_n^{l}(x,y)h^{u}_n(y)
=\sum_{y\in\VV_n}p_n^{l}(x,y)\sum_{z\in\VV_n}p_n(y,z)e^{-uc_n\l_n(z)}
\leq\sum_{z\in\VV_n}p_n^{l+1}(x,z)e^{-uc_n\l_n(z)},
\Eq(4.prop1.16)
\ee
and
\bea
\sum_{x\in\VV_n}\pi_n(x)h^{u}_n(x)p_n^{l+1}(x,z)
&=&
\sum_{y\in\VV_n}e^{-uc_n\l_n(y)}\sum_{x\in\VV_n}\pi_n(x)p_n(x,y)p_n^{l+1}(x,z)
\\
&\leq&
\sum_{y\in\VV_n}e^{-uc_n\l_n(y)}\pi_n(y)p_n^{l+2}(y,z),
\Eq(4.prop1.17)
\eea
where the last inequality follows by reversibility.
Hence,
\bea
(III)
&\leq&2(k_n(t)+1)\sum_{l=1}^{\theta_n-1}\sum_{z\in\VV_n}\left[\sum_{x\in\VV_n}\pi_n(x)h^{u}_n(x)p_n^{l+1}(x,z)\right]e^{-uc_n\l_n(z)},
\\
&\leq&2\sum_{l=1}^{\theta_n-1}(k_n(t)+1)\sum_{z\in\VV_n}\sum_{y\in\VV_n}\pi_n(y)e^{-uc_n(\l_n(y)+\l_n(z))}p_n^{l+2}(y,z)
\\
&=&2 \frac{(k_n(t)+1)}{a_n}\sum_{l=1}^{\theta_n-1}[(III)_{1,l}+(III)_{2,l}].
\Eq(4.prop1.18)
\eea
where, distinguishing the cases $z=y$ and $z\neq y$,
\bea
(III)_{1,l}
&=&\sum_{z\in\VV_n}a_n\pi_n(z)e^{-2uc_n\l_n(z)}p_n^{l+2}(z,z),
\\
(III)_{2,l}
&=&\sum_{z\in\VV_n}\sum_{y\in\VV_n : y\neq z}a_n\pi_n(y)e^{-uc_n(\l_n(y)+\l_n(z))}p_n^{l+2}(y,z).
\Eq(4.prop1.19)
\eea
One easily checks that $\theta_n\leq 2m$ with $m\leq n^2$. Thus, by Proposition \thv(3.prop3),
\be
\sum_{l=1}^{\theta_n-1}(III)_{1,l}
=\sum_{z\in\VV_n}a_n\pi_n(z)e^{-2uc_n\l_n(z)}\sum_{l=1}^{\theta_n-1}p_n^{l+2}(z,z)
\leq c n^{-2}\nu_n(2u,\infty)\,.
\Eq(4.prop1.20)
\ee
for some constant $0<c<\infty$.

The next lemma is designed to deal with the second sum in the last line of \eqv(4.prop1.18).

\begin{lemma}
\TH(4.lemma1)
Let $\rho_n>0$ be a decreasing sequence satisfying $\rho_n\downarrow 0$ as $n\uparrow\infty$.
There exists a sequence of subsets $\O^{\tau}_{n,0}\subset\O^{\tau}$ with
$
\P\left((\O^{\tau}_{n,0})^c\right)<\frac{\theta_n}{\rho_n a_n}\,,
$
and such that, on $\O^{\tau}_{n,0}$,
\be
\sum_{l=1}^{\theta_n-1}(III)_{2,l}<\rho_n\left[\E\nu_n(u,\infty)\right]^2\,.
\Eq(4.lem1.2)
\ee
\end{lemma}

\begin{proof}[Proof of Lemma \thv(4.lemma1)] 
$
\P\left(\sum_{l=1}^{\theta_n-1}(III)_{2,l}\geq\eta\right)
\leq
\eta^{-1}\sum_{l=1}^{\theta_n-1}\E(III)_{2,l}
$.
Next,  for all $y\neq z\in\VV_n\times \VV_n$, by independence,
$
\E\left[e^{-uc_n(\l_n(y)+\l_n(z))}\right]=\left[a^{-1}_n\E\nu_n(u,\infty)\right]^2
$.
Thus,
\be
\sum_{l=1}^{\theta_n-1}\E(III)_{2,l}
\leq
\frac{1}{a_n}\left[\E\nu_n(u,\infty)\right]^2\sum_{l=1}^{\theta_n-1}\sum_{z\in\VV_n}p_n^{l+2}(y,z)
\leq
\frac{\theta_n}{a_n}\left[\E\nu_n(u,\infty)\right]^2\,,
\ee
yielding
$
\P\left(\sum_{l=1}^{\theta_n-1}(III)_{2,l}\geq\eta\right)
\leq\frac{\theta_n}{\eta a_n}\left[\E\nu_n(u,\infty)\right]^2
$.
The lemma now easily follows.
\end{proof}

Collecting the bounds \eqv(4.prop1.15), \eqv(4.prop1.15'), \eqv(4.prop1.20), and \eqv(4.lem1.2),
and combining them with \eqv(4.prop1.13), we obtain that under the assumptions
and with the notations of Lemma \thv(4.lemma1), on $\O^{\tau}_{n,0}$,
for all $t>0$ and all $u>0$,
\bea
\nonumber
&&P_{\pi^\pm_n}\left(\left|\nu_n^{J,t}(u,\infty)-(k_n(t)/a_n)\nu_n(u,\infty)\right|\geq\e\right)
\\
&\leq&\e^{-2}\left\{
2\left(\sfrac{k_n(t)+1}{a_n}\right)^2\Theta^2_n(u)
+\sfrac{k_n(t)+1}{a_n}\Theta^1_n(u)
\right\}.
\Eq(4.prop1.21)
\eea
for some constant $0<c<\infty$, where $\Theta^2_n(u)$ and $\Theta^2_n(u)$  are defined in \eqv(4.prop1.2) and \eqv(4.prop1.2'), respectively.
Since
$
k_n(t)= \lf a_n t\rf
$
this yields \eqv(4.prop1.1). The proof of Proposition \thv(4.prop1) is done.
\end{proof}

\subsection{An ergodic theorem for  $M_n(t)$.}
\TH(SG4.2)

We turn to the concentration of $M_n(t)$ around $E_{\pi_n}[M_n(t)]$.
By \eqv(1'.theo3.0), \eqv(4.7) and \eqv(F1), writing $g_{1}\equiv g$ and setting
\be
G_n(y)=\sum_{x\in\VV_n}p_n(x,y)g(\g_n(x)),
\ee
we have 
\bea
M_n(t)\hspace{-6pt}&=&\hspace{-6pt}k_n(t)\sum_{x\in\VV_n}\pi_n^{J,t}(x)G_n(x),
\Eq(4.prop1'.3)
\\
E_{\pi_n}[M_n(t)]\hspace{-6pt}&=&\hspace{-6pt}
k_n(t)\sum_{x\in \VV_n}\pi_n(x)g(\g_n(x))
\equiv ({k_n(t)}/{a_n})m_n,
\Eq(4.prop1'.4)
\eea
where we used \eqv(G3.lem1.1') of Lemma \thv(G3.lem1) in the last line and where the last equality defines $m_n$.
Further set
\bea
v_n & = & \frac{a_n}{2^n}\sum_{z\in\mathcal{V}_n}[g(\g_n(z))]^2,\\
\Eq(4.prop1'.2)
w_n & = &  \frac{a_n}{2^n}\sum_{x\in\mathcal{V}_n} \sum_{x'\in\mathcal{V}_n} p_n^2(x,x')g(\g_n(x))g(\g_n(x')).
\Eq(4.prop1'.1)
\eea
The next Proposition is the analogue for $M_n(t)$ of Proposition \thv(4.prop1).

\begin{proposition}
\TH(4.prop1')
Let $\rho_n>0$ be a decreasing sequence satisfying $\rho_n\downarrow 0$ as $n\uparrow\infty$.
There exists a sequence of subsets $\O^{\tau}_{n,5}\subset\O^{\tau}$ with
$
\P\left((\O^{\tau}_{n,5})^c\right)<\frac{\theta_n}{\rho_na_n}
$,
and such that on $\O^{\tau}_{n,0}$, the following holds for all large enough $n$: for all $t>0$ and all $\e>0$,
\be
P_{\pi^\pm_n}\left(\left|M_n(t)-(k_n(t)/a_n)m_n\right|\geq\e\right)
\leq
2\e^{-2}
\left[
t\overline\Theta^1_n+t^2\overline\Theta^2_n
\right]
\Eq(4.prop1'.01)
\ee
where, for some constant $0<c<\infty$,
\be
\overline\Theta^1_n=w_n+2\frac{m_n^2}{a_n}+c\frac{v_n}{ n^{2}}+\rho_n[\E (m_n)]^2,
\quad
\overline\Theta^2_n =\frac{m_n^2}{2^{n}}.
\Eq(4.prop1'.02)
\ee
\end{proposition}

\begin{proof}
This is a simple re-run of the proof of Proposition \thv(4.prop1), substituting $G_n$ for $h^{u}_n$. 
\end{proof}

\section{Laws of large numbers and concentration}
\label{S6}

In this section we collect the laws of large numbers and concentration results that, once combined with Proposition \thv(4.prop1) and Proposition  \thv(4.prop1') respectively, will enable us to establish the validity of Conditions (A1), (A2), (A3) and (A3'). (Note that Condition (A0) reads
\be
{\nu_n(u,\infty)}/{a_n}=o(1),
\Eq(1.A0'bis)
\ee
which will trivially hold true as a by-product of our convergence results for $\nu_n(u,\infty)$.)

\subsection{Laws of large numbers for $\nu_n$ and $\s_n$}
\label{S6.1}
In this subection we study the terms  $\nu_n$ and $\s_n$ defined through \eqv(4.9') and \eqv(4.9'')  that enter the statement of Proposition \thv(4.prop1). In view of \eqv(4.8)  they read
\be
\nu_n(u,\infty)=
\frac{a_n}{2^n}\sum_{x\in\VV_n}
e^{-u/\g_n(x)},
\Eq(5.2)
\ee
\be
\s_n(u,\infty)
=
\frac{a_n}{2^n}\sum_{x,\in\VV_n}\sum_{x'\in\VV_n}
p^{2}_n(x,x')e^{-u/\g_n(x)}e^{-u/\g_n(x')}.
\Eq(5.3)
\ee
where $p^{2}_n(\cdot,\cdot)$ are the two steps transition probabilities of $J_n$.
The following laws of large numbers form the core of this subection.

\begin{proposition}[Intermediate time-scales]
\TH(5.prop6)
Given $0<\varepsilon\leq 1$ let $c_n$ be an intermediate scale. Let $\nu^{int}$ be defined in \eqv(1.prop1.1)
and assume that $\b\geq \b_c(\varepsilon)$.

\item{i)} If $\sum_{n}a_n/2^n<\infty$ then there exists a subset $\O^{\tau}_2\subset\O^{\tau}$
with $\P(\O^{\tau}_2)=1$ such that, on $\O^{\tau}_2$, the following holds: for all $u>0$
\bea\Eq(5.prop6.1)
\lim_{n\rightarrow\infty}\nu_n(u,\infty)&=&\nu^{int}(u,\infty),
\nonumber\\
\lim_{n\rightarrow\infty}n\s_n(u,\infty)&=&\nu^{int}(2u,\infty).
\eea
\item{ii)} If $\sum_{n}a_n/2^n=\infty$ then there exists a sequence of subsets $\O^{\tau}_{n,3}\subset\O^{\tau}$ with $\P(\O^{\tau}_{n,3})\geq 1-o(1)$ and such that for all $n$ large enough, on $\O^{\tau}_{n,3}$, the following holds: for all $u>0$ 
\bea
&
\left|\nu_{n}(u,\infty)-\nu^{int}(u,\infty)\right|=o(1),
\Eq(5.prop6.2)
\\
&
\left|n\s_{n}(u,\infty)-\nu^{int}(2u,\infty)\right|=o(1).
\Eq(5.prop6.2bis) 
\eea
\end{proposition}

The proofs of  Proposition \thv(5.prop6), which is given at the end of this subsection, rely on the following three lemmata.

\begin{lemma}
\TH(5.lemma7)
Under the assumptions  and with the notation of Proposition \thv(5.prop6),
\be
\lim_{n\rightarrow\infty}\E[\nu_n(u,\infty)]=\nu^{int}(u,\infty),\quad\forall u>0.
\Eq(5.lem7.1)
\ee
Furthermore,
\be
\displaystyle
\E[\s_n(u,\infty)]
=\frac{\E[\nu_n(2u,\infty)]}{n}+\frac{(\E[\nu_n(u,\infty)])^2}{a_n}\frac{n-1}{n}.
\Eq(5.lem8'.0)
\ee
\end{lemma}

\begin{lemma}\TH(5.lemma8')
For all $\kappa\geq 0$ such that $a_n\kappa/2^n=o(1)$,
\be
\P\left(\left|\nu_n(u,\infty)-\E[\nu_n(u,\infty)]\right|
\geq 2\sqrt{a_n\kappa/2^n}\sqrt{\E[\nu_n(2u,\infty)]}\right)
\leq e^{-\kappa}.
\Eq(5.lem7.2)
\ee
\end{lemma}

\begin{lemma}
\TH(5.lemma8)
Under the assumptions of Proposition \thv(5.prop6),  for all $\kappa>0$,
\be
\P\left(\left|\s_n(u,\infty)-\E[\s_n(u,\infty)]\right|
\geq n^{-1}\sqrt{a_n\kappa/2^n}\sqrt{\E\left[\nu_n(u,\infty)\right]}\right)
\leq \kappa^{-1}.
\Eq(5.lem8.2)
\ee
\end{lemma}

\begin{proof}[Proof of Lemma  \thv(5.lemma7)] 
We first prove \eqv(5.lem7.1).
For fixed $u>0$ set $f(y)=e^{-u/y}$.
By \eqv(5.2), integrating by part and using \eqv(2.13),
\be
\E[\nu_n(u,\infty)]
=
a_n\frac{|\AA_n|}{2^n}\int_{0}^{\infty}f'(y)\P\left(\g_n(0)\geq y\right)dy
=
(1+o(1))\int_{0}^{\infty}f'(y)h_n(y)dy,
\nonumber
\ee
since by assumption $2^{-n}|A_n|\uparrow 1$  as $n\uparrow\infty$.
Set  $I_n(a,b)=\int_{a}^{b}f'(y)h_n(y)dy$, $a\leq b$, and, given $0<\hat\zeta<1$ and $\zeta>1$
break the last integral above into
\be
I_n\bigl(0,c_n^{-1/2}\bigr)+I_n\bigl(c_n^{-1/2},\hat\zeta\bigr)+I_n\bigl(\hat\zeta,\zeta\bigr)+I_n(\zeta,\infty).
\Eq(5.lem7.3')
\ee
Let us now establish that,  as $n\uparrow\infty$, for small enough $\hat\zeta$ and large enough $\zeta$, 
the leading contribution to \eqv(5.lem7.3') comes from $I_n\bigl(\hat\zeta,\zeta\bigr)$.
By \eqv(1.1') and the rough upper bound
$
h_n(y)\leq a_n
$,
$
\textstyle
I_n\bigl(0,c_n^{-1/2}\bigr)\leq a_n\int_{0}^{1/\sqrt{c_n}}f'(y)dy={e^{-u\sqrt{c_n}}}/{\P(\t_n(x)\geq c_n)}
$,
and, together with the gaussian tail estimates of \eqv(2.11),  this yields
\be
\lim_{n\rightarrow\infty}I_n\bigl(0,c_n^{-1/2}\bigr)=0.
\Eq(5.lem7.3'')
\ee
Next, by Lemma \thv(2.lemma6), (ii), with $\d=1/2$,
$
I_n\bigl(c_n^{-1/2},\hat\zeta\bigr)
\leq 2(1+o(1))\int_{0}^{\hat\zeta}f'(y)y^{-(3/4)\a_n}dy
$
for all $0<\hat\zeta<1$, where $0\leq \a_n=\a(\varepsilon)+o(1)$.
Now, there exists $\zeta^*\equiv\zeta^*(u)>0$ such that, for all $\hat\zeta<\zeta^*$,
$f'(y)y^{-(3/4)\a_n}$ is strictly increasing on $[0,\hat\zeta]$. Hence,
for all $\hat\zeta<\min(1,\zeta^*)$,
$
I_n\bigl(c_n^{-1/2},\hat\zeta\bigr)
\leq 2(1+o(1))u\hat\zeta^{-1+(3/4)[\a(\varepsilon)+o(1)]}e^{-{u}/{\hat\zeta}}
$,
implying that
\be
\lim_{n\rightarrow\infty}I_n\bigl(c_n^{-1/2},\hat\zeta\bigr)
\leq 2u\hat\zeta^{-1+(3/4)\a(\varepsilon)}e^{-{u}/{\hat\zeta}},\quad \hat\zeta<\min(1,\zeta^*).
\Eq(5.lem7.3''')
\ee
To deal with $I_n\bigl(\hat\zeta,\zeta\bigr)$ note that
by Lemma \thv(2.lemma6), (i), $h_n(y)\rightarrow y^{-\a(\varepsilon)}$, $n\rightarrow\infty$,
where the convergence is uniform in  $\hat\zeta\leq y\leq \zeta$ since, for each $n$,
$h_n(y)$ is a monotone function, and since the limit, $y^{-\a(\varepsilon)}$,
is continuous. Hence,
\be
\lim_{n\rightarrow\infty}I_n\bigl(\hat\zeta,\zeta\bigr)
=\lim_{n\rightarrow\infty}\int_{\hat\zeta}^{\zeta}f'(y)h_n(y)dy
=\int_{\hat\zeta}^{\zeta}f'(y)y^{-\a(\varepsilon)}dy.
\Eq(5.lem7.4)
\ee
It remains to bound $I_n(\zeta,\infty)$. By
\eqv(2.lem6.1) of Lemma \thv(2.lemma6),
$
I_n(\zeta,\infty)
=\int_{\zeta}^{\infty}f'(y)h_n(y)dy
=(1+o(1))\int_{\zeta}^{\infty}f'(y)y^{-\a_n}dy,
$
where again $0\leq \a_n=\a(\varepsilon)+o(1)$.
Thus, for $0<\d<1$ arbitrary we have, for  large enough $n$, that for all $y\geq\zeta>1$,
$
f'(y)y^{-\a_n}
\leq f'(y)y^{-\a(\varepsilon)+\d}
\leq u/y^{2-\d}
$.
Therefore
$
I_n(\zeta,\infty)
\leq(1+o(1))\frac{1}{1-\d}\zeta^{-(1-\d)}
$,
and, choosing e.g.~$\d=1/2$,
\be
\lim_{n\rightarrow\infty}I_n(\zeta,\infty)\leq 2u\zeta^{-1/2}.
\Eq(5.lem7.5)
\ee
Collecting \eqv(5.lem7.3'')-\eqv(5.lem7.5) and passing to the limit $\hat\zeta\rightarrow 0$ and $\zeta\rightarrow\infty$, we finally get
\be
\lim_{n\rightarrow\infty}\E[\nu_n(u,\infty)]=\int_{0}^{\infty}f'(y)y^{-\a(\varepsilon)}dy
=u^{-\a(\varepsilon)}\a(\varepsilon)\G(\a(\varepsilon)),
\Eq(5.lem7.8)
\ee
where we used that $\a(\varepsilon)>0$ since by assumption $\varepsilon>0$. This proves \eqv(5.lem7.1).
We skip the elementary proof of \eqv(5.lem8'.0).
\end{proof}

\begin{proof}[Proof of Lemma  \thv(5.lemma8')]
The proof relies on Bennett's bound \cite{Ben} for the tail behavior of sums of random variables,
which states that if $(X(x),\, x\in A)$ is a family of i.i.d.~centered random variables
that satisfies $\max_{x\in A}|X(x)|\leq \bar a$ then, setting
$\tilde b^2=\sum_{x\in A}\E X^2(x)$,  for all $\bar b^2\geq \tilde b^2$,
$
\textstyle
\P\left(\left|\sum_{x\in A}X(x)\right|>t\right)
\leq\exp\bigl\{
\frac{t}{\bar a}-\bigl(\frac{t}{\bar a}+\frac{\bar b^2}{\bar a^2}\bigr)\log\bigl(1+\frac{\bar at}{\bar b^2}\bigr)
\bigr\}
$,
$t\geq 0$.
This implies in particular that for $t<\bar b^2/(2\bar a)$, 
\be
\textstyle{\P\left(\left|\sum_{x\in A}X(x)\right|\geq t\right)}
\leq\exp\bigl\{-{t^2}/{4\bar b^2}\bigr\}.
\Eq(5.lem7.11)
\ee
Now take $X(x)=e^{-u/\g_n(x)}-\E e^{-u/\g_n(x)}$, $x\in\AA_n$, so that for all $\theta>0$,
\be
\P\left(\left|\nu_n(u,\infty)-\E[\nu_n(u,\infty)]\right|\geq \theta\right)
=\textstyle{\P\left(\left|\sum_{x\in\AA_n}X(x)\right|\geq 2^n a_n^{-1}\theta\right)}.
\Eq(5.lem7.9)
\ee
Since $|X(x)|\leq 1$ and
$
\sum_{x\in\AA_n}\E X^2(x) \leq |\AA_n|\E e^{-2u/\g_n(x)}= 2^na_n^{-1}\E[\nu_n(2u,\infty)]
$,
we can apply Bennett's bound with $\bar a=1$ and $\bar b^2=2^na_n^{-1}\E[\nu_n(2u,\infty)]$, and 
by  \eqv(5.lem7.11), choosing $\theta^2= a_n \kappa 2^{-n+2}\E[\nu_n(2u,\infty)]$ in \eqv(5.lem7.9),
\be
\P\left(\left|\nu_n(u,\infty)-\E[\nu_n(u,\infty)]\right|
\geq 2\sqrt{a_n\kappa/2^n}\sqrt{\E[\nu_n(2u,\infty)]}\right)
\leq e^{-\kappa},
\Eq(5.lem7.15)
\ee
 for all $\kappa>0$.
This choice of $\theta$ is permissible provided that
$\theta\leq \E[\nu_n(2u,\infty)]/2$. In view of \eqv(5.lem7.1)
this will be verified for all $n$ large enough whenever $\theta\downarrow 0$
as $n\uparrow\infty$, i.e.~whenever $a_n\kappa/2^n=o(1)$.
This concludes the proof of
\eqv(5.lem7.2), and of the lemma.
\end{proof}

\begin{proof}[Proof of Lemma   \thv(5.lemma8)]
For $u>0$ and $l\geq 1$ set
\be
\bar\s_n^l(u,\infty)=a_n\sum_{y\in\AA_n}\pi_n(y)\left(h^{u}_n(y)\right)^l.
\Eq(5.lem8.1)
\ee
Thus  $\nu_n(u,\infty)=\bar\s_n^1(u,\infty)$ and $\s_n(u,\infty)=\bar\s_n^2(u,\infty)$.
If for $l=1$, $\bar\s_n^l(u,\infty)$
is a sum of independent random variables, this is no longer true when $l=2$. In this case
we simply use a second order Tchebychev inequality to write
\bea
\P\left(\left|\bar\s_n^l(u,\infty)-\E[\bar\s_n^l(u,\infty)]\right|\geq t\right)
\leq t^{-2}[\theta_1+\theta_2],
\Eq(5.lem8.4)
\eea
where
\bea
\theta_1&=&\left(\frac{a_n}{2^n}\right)^2\sum_{y\in\AA_n}\E\left[(h^{u}_n(y))^l-\E(h^{u}_n(y))^l\right]^2,
\Eq(5.lem8.5)
\\
\theta_2&=&\left(\frac{a_n}{2^n}\right)^2\sum_{{y,y'\in\AA_n\times\AA_n}\atop {y\neq y'}}
\E\left\{\left[(h^{u}_n(y))^l-\E(h^{u}_n(y))^l\right]\left[(h^{u}_n(y'))^l-\E(h^{u}_n(y'))^l\right]\right\}.\quad\quad
\Eq(5.lem8.5bis')
\eea
On the one hand
\be
\theta_1
\leq \frac{a_n}{2^n}\E[\bar\s_n^{2l}(u,\infty)].
\Eq(5.lem8.5')
\ee
On the other hand, after some lengthy but simple calculations, we obtain that
\bea
\theta_2
&\leq&
\textstyle
\frac{n(n-1)}{2^{n+1}}\Bigl[
\frac{a_n}{n^{2l}}\E\left[\nu_n(u,\infty)\right]
+2\frac{(\E[\nu_n(u,\infty)])^2}{n^l}\left(\frac{\E[\nu_n(u,\infty)]}{a_n}+\frac{2}{n}\right)^{l-1}
\cr
&&
\textstyle
+\frac{1}{a_n}(\E[\nu_n(u,\infty)])^3\left(\frac{\E[\nu_n(u,\infty)]}{a_n}+\frac{1}{n}\right)^{2(l-1)}
\Bigr].
\Eq(5.lem8.20)
\eea
Since on intermediate scales ${n^{2l}}/{a_n}=o(1)$ for any $l<\infty$, it follows from \eqv(5.lem8.20) and \eqv(5.lem7.1)
that for all $n$ large enough
\be
\theta_1+\theta_2\leq\frac{\E\left[\nu_n(u,\infty)\right]}{n^{2(l-1)}}\frac{a_n}{2^{n}}.
\Eq(5.lem8.24)
\ee
Inserting \eqv(5.lem8.24) in \eqv(5.lem8.4) and choosing
$
t=n^{-(l-1)}\sqrt{({a_n\kappa}/{2^n})\E\left[\nu_n(u,\infty)\right]}
$
yields
\be
\P\left(\left|\bar\s_n^l(u,\infty)-\E[\bar\s_n^l(u,\infty)]\right|
\geq n^{-(l-1)}\sqrt{a_n\kappa/2^n}\sqrt{\E\left[\nu_n(u,\infty)\right]}\right)
\leq \kappa^{-1},
\Eq(5.lem8.2')
\ee
and taking $l=2$ in \eqv(5.lem8.2') gives \eqv(5.lem8.2).
The proof of Lemma \thv(5.lemma8) is complete.
\end{proof}

\begin{proof}[Proof of Proposition \thv(5.prop6)]
By definition of an intermediate time-scale, any sequence $a_n$ must satisfy $a_n/2^n=o(1)$. 
Let us first assume that $\sum_{n}a_n/2^n<\infty$.
This implies in particular that $(a_n\log n)/2^n=o(1)$ and $n/a_n=o(1)$.
Thus, using Lemma \thv(5.lemma8') with $\kappa=2\log n$,
it follows from Borel-Cantelli Lemma that
\be
\lim_{n\rightarrow\infty}\nu_n(u,\infty)=\nu^{int}(u,\infty)\,\,\,\text{$\P$-almost surely.}
\Eq(5.prop6.3)
\ee
Together with the monotonicity of $\nu_n$ and the continuity of the limiting
function $\nu^{int}$, \eqv(5.prop6.3) entails the existence of a subset 
$\O^{\tau}_{2,1}\subset\O^{\tau}$
with the property that $\P(\O^{\tau}_{2,1})=1$, and such that, on $\O^{\tau}_{2,1}$,
\be
\lim_{n\rightarrow\infty}\nu_n(u,\infty)=\nu^{int}(u,\infty),\quad\forall\, u>0.
\Eq(5.prop6.6)
\ee
Similarly, using \eqv(5.lem8.2) of Lemma \thv(5.lemma8) with $\kappa=2^n/a_n$,
it follows from
\eqv(5.lem8'.0) and Borel-Cantelli Lemma that
\be
\lim_{n\rightarrow\infty}n\s_n(u,\infty)=\nu^{int}(2u,\infty)\,\,\,\text{$\P$-almost surely.}
\Eq(5.prop6.4)
\ee
This and the monotonicity of $\s_n$ allows us to conclude that
there exist a subset $\O^{\tau}_{2,2}\subset\O^{\tau}$ of full measure such that, on $\O^{\tau}_{2,2}$,
\be
\lim_{n\rightarrow\infty}n\s_n(u,\infty)=\nu^{int}(2u,\infty),\quad\forall\, u>0.
\Eq(5.prop6.7)
\ee
Assertion i) of the proposition now follows by taking $\O^{\tau}_2=\O^{\tau}_{2,1}\cap\O^{\tau}_{2,2}$.

If now $\sum_{n}a_n/2^n=\infty$ our estimates do not guarantee almost sure convergence of $\nu_n(u,\infty)$ and
$n\s_n(u,\infty)$ but still yield almost sure convergence along sub-sequences. Using the characterisation of convergence of probability in terms of almost sure convergence of sub-sequences (see e.g. \cite{RW}, Sect.~II.~19).
This allows us to reduce the proof in this case to the case of almost sure convergence treated in the proof of Assertion i).
\end{proof}

\subsection{Concentration of $m_n$ and of terms appearing in the ergodic theorem for $M_n(t)$}
As we will later make use of Lemma \thv(4.1'') under the  condition that $\lim_{n\to \infty}\sqrt{n}\beta- \frac{\log c_n}{\sqrt{n}\beta}=\theta$, which by \eqv(2.4.4) of Lemma \thv(Lem2.4.1) implies that $\E (m_n)=C\sqrt{n}$ for some constant $C$, we need to control all quantities appearing in \eqv(4.prop1'.02) of Proposition \thv(4.prop1') with an extra  multiplicative factor $\sqrt{n}$.

\begin{lemma}\TH(4.1'')
 (i) If $\sum_{n}a_n/2^n<\infty$ then there exists a subset $\O^{\tau}_6\subset\O^{\tau}$
with $\P(\O^{\tau}_6)=1$ such that, on $\O^{\tau}_6$
 \bea
\lim_{n\to\infty} \sqrt{n}  \max\left\{ \frac{m_n^2}{2^{n}},\frac{m_n^2}{a_n}, c n^{-2} v_n\right\}&=& 0,\nonumber\\
\lim_{n\to\infty}  \max\left\{ \left \vert m_n -\E m_n\right\vert,\; \sqrt{n}\left \vert w_n -\E w_n\right\vert\right\} & = & 0.
\eea
(ii) If $\sum_{n}a_n/2^n=\infty$ there exists 
$\O^{\tau}_{n,6}\subset\O^{\tau}$ with $\P\left(\O^{\tau}_{n,6}\right)\geq 1-o(1)$ such that for $n$ large enough, on $\O^{\tau}_{n,6}$\bea
\sqrt{n}  \max\left\{ \frac{m_n^2}{2^{n}},\frac{m_n^2}{a_n}, c n^{-2} v_n\right\}& \leq &(a_n/2^n)^{1/4}, \nonumber\\
\sqrt{n}\max\left\{ \left \vert m_n -\E m_n\right\vert, \left \vert w_n -\E w_n\right\vert\right\} & \leq &(a_n/2^n)^{1/4}.
\eea
 (iii) Moreover, for all sequences $\rho_n$ such that  $\rho_n\left(\frac{a_n e^{n\beta^2/2}}{c_n}\right)^2=o(1)$ we have
 \be
 \lim_{n\to\infty}  \max\left\{\E\frac{m_n^2}{2^{n}},\E\frac{m_n^2}{a_n},  c n^{-2} \E v_n,\E w_n,\rho_n[\E (m_n)]^2\right\}=0.
 \Eq(4.1''.3)
 \ee
\end{lemma}

\begin{proof} Let us first compute the expected  values of each of the terms appearing in \eqv(4.1''.3).
Consider first  $w_n$.
\be
\Eq(Bonn.2)
\E w_n =\frac{a_n}{n} \E\left(g(\g_n(x))^2 \right)
+ \frac{a_n}{2^n}\sum_{x\in\mathcal{V}_n} \sum_{\substack{x'\in\mathcal{V}_n, \\x\neq x'}} p_n^2(x,x') 
\E\left(g(\g_n(x)) g(\g_n(x'))\right).
\ee
By \eqv(2.4.5) the first summand in \eqv(Bonn.2) is equal to $ C \frac{1}{n}(1+o(1))$ for some constant $C>0$.  By \eqv(2.4.4') the second summand in \eqv(Bonn.2) is  equal to 
\be
\frac{n-1}{n} a_n\left(\E\left(g(\g_n(x))\right)\right)^2
\leq \frac{n-1}{n}a_n\left(c_n^{-1}{e^{\b^2 n/2}}\right)^2.
\ee
Thus
\be
\E w_n \leq  C \frac{1}{n}(1+o(1))+\frac{n-1}{n}a_n\left(c_n^{-1}{e^{\b^2 n/2}}\right)^2.
\ee
Similarly, recalling from \eqv(4.prop1'.4) that
$m_n=({a_n}/{2^n})\sum_{x\in\mathcal{V}_n}g(\g_n(x))$, 
we have by \eqv(2.4.4') that $\rho_n [\E(m_n)]^2\leq \rho_n\left({a_nc_n^{-1}e^{n\beta^2/2}}\right)^2$ and 
\be
\Eq(NY.100)
\E\left(\frac{m_n^2}{2^{n}}\right)\leq C\frac{a_n}{2^{2n}}+ \frac{a_n^2 e^{\b^2n}}{2^nc_n^2}, 
\quad\E\left(\frac{m_n^2}{a_n}\right)\leq C\frac{1}{2^{n}}+ \frac{a_n e^{\b^2n}}{c_n^2},
\quad\E\left(c \frac{v_n}{n^{2} }\right) \leq  C\frac{1}{n^{2} }.
\ee
Collecting the above bounds readily proves part (iii) of Lemma \thv(4.1'').
Throughout the rest of the proof let $C>0$ be a generic constant that is large enough  to fulfill all desired  inequalities. 
The proofs of the first two parts follow from  part (iii)  and Markov inequalities.
To prove concentration of $m_n$ around its mean value we use that by a second order Tchebychev inequality, 
for all $\e>0$,  
$
\P\left( \left| m_n-\E m_n\right|>\e \right)
\leq 
\e^{-2}\left(\E\left(m_n^2\right)- 
\left(
\E m_n
\right)^2\right)
$.
By \eqv(4.prop1'.1), 
$
\E\left(m_n^2\right)=({a_n}/{2^n})(v_n+w_n)
$.
Now it follows from the calculations in   \eqv(2.4.5) that 
$
v_n<c_2
$
whereas
$
({a_n}/{2^n})w_n -\left(\E m_n\right)^2 = - ({a_n^2}/{2^n})\left(\E
g(\g_n(x))\right)^2 < 0
$.
 Thus 
\be\Eq(2.4.8)
 \P\left( \left| m_n-\E m_n\right|>\e \right)
\leq \e^{-2} c_2 ({a_n}/{2^n}).
\ee
We next prove concentration of $w_n$. Using again a second order Tchebychev inequality we have, for all $\e>0$ 
$
\mathbb{P}\left(\left \vert w_n-\E\left(w_n\right)\right \vert > \e\right) \leq \e^{-2}\left(\theta_1 +\theta_2\right)
$,
where
\bea
\theta_1 & = & \left(\frac{a_n}{2^n}\right)^2 \sum_{y\in\mathcal{V}_n} \E\left(G_n(y)^2-\E\left(G_n(y)\right)^2\right)^2,\\
\theta_2 & = & \left(\frac{a_n}{2^n}\right)^2 \sum_{\substack{y, y'\in\mathcal{V}_n,\\ y\neq y'}}\E\left(G_n(y)^2-\E\left(G_n(y)\right)^2\right)\big(G_n(y')^2-\E\left(G_n(y')\right)^2\big).
\eea
Hence we observe that the expectation with respect to the random environment of all terms appearing converges to $0$ as $n\to\infty$. 
First, we bound $\theta_1$ from above by
\be
\theta_1
\leq \left(\frac{a_n}{2^n}\right)^2 \sum_{y\in\mathcal{V}_n}\E\left(G_n(y)^4\right)-\left(\E\left(G_n(y)\right)^2\right)^2.
\ee
Expanding 
$
\left({a_n}/{2^n}\right)^2 
\sum_{y\in\mathcal{V}_n}\E\left(G_n(y)^4\right)
$,
we bound $\theta_1$ from above by $\left({a_n}/{2^n}\right)^2 $ times
\bea
& &   \sum_{y,x\in\mathcal{V}_n} p_n(y,x)^4 \E\left(g(\g_n(x))^4 \right) 
+C 
\sum_{\substack{y,x,x'\in\mathcal{V}_n,\\ x\neq x'}}p_n(y,x) p_n(y,x')^3\E\left(g(\g_n(x))g(\g_n(x')^3)\right) \nonumber\\
&&+  C  \sum_{y\in\mathcal{V}_n} \sum_{x\in\mathcal{V}_n} \sum_{\substack{x'\in\mathcal{V}_n,\\ x\neq x'}}p_n(y,x)^2 p_n(y,x')^2\E\left(g(\g_n(x))^2g(\g_n(x'))^2\right)\nonumber\\
&&+   C
 \sum_{\substack{y,x,x'\in\mathcal{V}_n,\\ x\neq x'}}\sum_{\substack{x''\in\mathcal{V}_n,\\ x''\neq x', x''\neq x}}p_n(y,x)^2 p_n(y,x') p_n(y,x'') \E\left((g(\g_n(x))^2g(\g_n(x'))g(\g_n(x''))\right) \nonumber\\
& &+ 
\sum_{\substack{y,x_0,x_1,x_2,x_3\in\mathcal{V}_n,\\ x_0,x_1,x_2,x_3\\\; pairwise\; distinct}}
 \;\;\prod_{i=0}^3p_n(y,x_i) 
\E\left(g(\g_n(x))g(\g_n(x'))g(\g_n(x''))g(\g_n(x'''))\right)  
\eea
Using the calculations in Lemma \thv(Lem2.4.1) on the behavior of $g(\g_n(x))$ and the independence of $\g_n(x)$ and $\g_n(y)$ when $x\neq y$, we arrive at
\be\Eq(NY.102)
\theta_1\leq \frac{a_n}{2^n}\left(\frac{1}{n^3}+C\frac{e^{n\b^2/2}}{n^2 c_n}+  \frac{C}{n^2 2^n a_n}+ C\frac{e^{n\b^2}}{n c_n^2} + \frac{e^{2n\b^2}a_n}{c_n^4}\right).
\ee
Expanding the expression $\theta_2$, we obtain 
\be\Eq(NY.101)
\theta_2=\left(\frac{a_n}{2^n}\right)^2\sum_{y\in\mathcal{V}_n}\sum_{\substack{y'\in\mathcal{V}_n,\\ y\neq y'}}\E\left(G_n(y)^2G_n(y')^2\right)-\E\left(G_n(y)^2\right)\E\left(G_n(y')^2\right).
\ee
We observe that when expanding the terms of \eqv(NY.101), some of the resulting terms cancel each other so that $ \theta_2$ is bounded from above by  $C\left({a_n}/{2^n}\right)^2$ times
\bea
&&  
\sum_{\substack{y,y',x\in\mathcal{V}_n,\\ y\neq y'}}
p_n(y,x)^2p_n(y',x)^2\E\left(g(\g_n(x))^4 \right)\nonumber\\
&& +    
\sum_{\substack{y,y'\in\mathcal{V}_n,\\ y\neq y'}}
  \sum_{\substack{x,x'\in\mathcal{V}_n,\\ x'\neq x}}p_n(y,x)p_n(y,x')p_n(y',x)^2 
\E\left(g(\g_n(x))^3g(\g_n(x'))\right)
\nonumber\\
&& 
+  \left.  \sum_{\substack{y,y'\in\mathcal{V}_n,\\ y\neq y'}}   \sum_{\substack{x,x'\in\mathcal{V}_n,\\ x'\neq x}}\sum_{\substack{z'\in\mathcal{V}_n,\\ z'\neq x, z'\neq x'}}
p_n(y,x) p_n(y,x') p_n(y',x)p_n(y',z')\right. \nonumber\\
&&\qquad\qquad\qquad \qquad \qquad  \times\left.\E\left(g(\g_n(x))^2g(\g_n(x'))g(\g_n(z')\right)\right).
\eea
Using once more Lemma \thv(Lem2.4.1) we get that
\be\Eq(NY.103)
\theta_2  \leq 
 \frac{a_n}{2^n}\left(\frac{1}{n^2}+C\frac{e^{n\b^2/2}}{n c_n}+C\frac{e^{n\b^2}}{c_n^2}\right).
\ee
Collecting the bounds of \eqv(NY.102) and \eqv(NY.103) gives
\be
\mathbb{P}\left(\left \vert w_n-\E\left(w_n\right)\right \vert > \e \right) 
\leq C\frac{a_n}{2^n}\left(
 \frac{e^{2n\b^2 }a_n}{c_n^4}+\frac{1}{n^2}+\frac{e^{n\b^2/2}}{n c_n}+\frac{e^{n\b^2}}{c_n^2}\right).
\ee
Choosing $\e={\e_0}/{\sqrt{n}}$ for some $\e_0>0$, the claim of Part  (i)  of Lemma \thv(4.1'') follows from Borel-Cantelli Lemma as  the bounds of \eqv(2.4.8) are summable if $\sum a_n/2^n<\infty$. If $\sum a_n/2^n=\infty$, Part (ii) of Lemma \thv(4.1'') follows.
\end{proof}

\subsection{Verification of Conditions (A3) and (A3')} Recall the definition of $g_\d(u)$ and $f_\d(u)$  from \eqv(F1) and \eqv(F2), respectively, and define  the key  quantities
\be
\l_{\d,n} = \frac{a_n}{2^n} \sum_{x\in \VV_n}g_\d(\g_n(x)), \quad {\bar\l}_{\d,n}  \equiv 
\frac{a_n}{2^n}\sum_{x\in\mathcal{V}_n}
f_\d(\g_n(x)).
\Eq(7.1)
\ee
Observe that the quantity appearing in \eqv(G1.A3) in Condition (A3) is equal to $(k_n(t)/a_n) \l_{\d,n}$. Similarly, the quantity in \eqv(G1.A3') in Condition (A3') is equal to $(k_n(t)/a_n) {\bar\l}_{\d,n} $.
\begin{lemma}\TH(A3) 
\item{(a)} Let $c_n$ be an intermediate time-scale and assume that  $0<\varepsilon\leq 1$ and $0<\beta<\infty$ are such that  
$0<\a(\varepsilon)< 1$.
\item{(a-1)} If $\sum_{n}a_n/2^n<\infty$ then there exists $\Omega_8^\t\subset\Omega^\t$ with $\P\left(\Omega_8^\t\right)=1$ such that, on $\Omega_8^\t$ 
\be
 \lim_{\d\to 0}\lim_{n\to\infty} \l_{\d,n}
 = 0.
 \Eq(4.L.1')
\ee
\item{(a-2)}  If $\sum_{n}a_n/2^n=\infty$  then there exists $\Omega_{n,8}^\t\subset\Omega^\t$ with $\P\left(\Omega_{n,8}^\t\right)\geq 1-o(1)$ such that for $n$ large enough, on $\Omega_{n,8}^\t$ 
\be\Eq(2.3.5')
\left \vert  \l_{\d,n}- \E\left( \l_{\d,n}\right)\right \vert \leq \left(\frac{a_n}{2^n}\right)^{1/4},
\ee
and
$
\lim_{\d\to 0 }\lim_{n\to\infty}\E\left( \l_{\d,n}\right) = 0.
$
\item{(b)}  Let $c_n$ be an intermediate time-scale and take $\b=\b_c(\ve)$ with $0<\ve\leq 1$. Then the statement of assertion (a)
above holds with $\l_{\d,n}$ replaced by ${\bar\l}_{\d,n}$.
\end{lemma}

\begin{proof}
To prove Part (a) note first that by \eqv(2.3.3'), $\E\left(\l_{\d,n}\right) \leq c_3 \d^{1-\a}$ for all large enough  $n$,   so that
$
\lim_{\d\to 0 }\lim_{n\to\infty}\E\left( \l_{\d,n}\right) = 0
$,
and use next that  by a second order Tchebychev inequality and Lemma \thv(Lem2.4.0),
\be\Eq(8.1)
\P\left(\left \vert \l_{\d,n}-\E\left(\l_{\d,n}\right)\right \vert >\e\right)
\leq \e^{-2}\left(\frac{a_n}{2^n}\right)^2\sum_{x\in \VV_n}\E\left(g_\d(\g_\d(x))^2\right)<\e^{-2}c_4 \frac{a_n}{2^n}.
\ee
Part (b) is proved in a  similar way. By \eqv(2.3.3),  for all large enough  $n$, $\E\left( {\bar\l}_{\d,n}\right)\leq c_0\delta$, implying that
$
\lim_{\d\to 0 }\lim_{n\to\infty}\E\left( {\bar\l}_{\d,n}\right) = 0
$,
while by  a second order Tchebychev inequality, for all $\e>0$
\be\Eq(A.3'.1)
\mathbb{P}\left(\left \vert {\bar\l}_{\d,n}
-  \E\left(  {\bar\l}_{\d,n}\right)\right \vert >\e\right)
\leq \e^{-2}\frac{a_n}{2^n} a_n\E\left(f_\d(\g_n(x))^2\right) \leq \e^{-2} c_1 \frac{a_n}{2^n},
\ee
where by independence of $f_\d(\g_n(x))$ and $f_\d(\g_n(x'))$ if $x \neq x' $  and \eqv(2.3.4).  
Based on \eqv(8.1) and \eqv(A.3'.1) the proof of Lemma \thv(A3) is concluded by arguing as in the proof of Proposition \thv(5.prop6).\end{proof}
 
\section{Proof of Theorem \thv(1'.theo1) and Proposition \thv(LCor.3)}
\label{S7}

Using the results of the two previous sections, we are now finally in the position to prove Theorem \thv(1'.theo1).

\begin{proof}[Proof of Theorem \thv(1'.theo1)] We first prove Assertion (i).
Choose $\nu=\nu^{int}$ as in \eqv(1.prop1.1) in Conditions (A1) (see \eqv(G1.A1)).
By Proposition \thv(4.prop1) and the estimates of Proposition \thv(5.prop6),  Conditions (A1), (A2) and (A0) are
satisfied $\P$-almost surely if $\sum_{n}a_n/2^n<\infty$ and  in $\P$-probability  if $\sum_{n}a_n/2^n=\infty$. 
By  Assertion (a) of Lemma \thv(A3),  when $\b>\b_c(\ve)$, Condition (A3) is satisfied 
$\P$-almost surely if $\sum_{n}a_n/2^n<\infty$ and  in $\P$-probability  if $\sum_{n}a_n/2^n=\infty$.
Thus Assertion (i) of Theorem \thv(1.theo3) implies that, under the same conditions and w.r.t.~the 
same convergence mode as above, $S_n\Rightarrow  S^{int}$ as $n\rightarrow\infty$, 
where $S^{int}$ is the subordinator with L\'evy measure $\nu^{int}$. 
This proves Assertion (i) of Theorem \thv(1'.theo1).

We now turn to Assertion (ii). If $\b=\b_c(\ve)$, reasoning as in the proof of Assertion (i),
Conditions (A1) and (A2) are satisfied $\P$-almost surely if $\sum_{n}a_n/2^n<\infty$ and 
in $\P$-probability  if $\sum_{n}a_n/2^n=\infty$. Furthermore, by Assertion (b) of Lemma \thv(A3) 
Condition (A3') is satisfied $\P$-almost surely if $\sum_{n}a_n/2^n<\infty$ and  in $\P$-probability
if $\sum_{n}a_n/2^n=\infty$. Thus Assertion  (ii) of Theorem \thv(1.theo3) implies that, under the 
same conditions and w.r.t.~the same convergence mode as above,
$S_n-\frac{M_n}{c_n}\Rightarrow  S^{crit}$, proving \eqv(1'.theo3.1').

We now  assume that $\lim_{n\to \infty} \sqrt{n}\beta-\frac{\log c_n}{\sqrt{n}\beta}=\theta$ for some $\theta\in(-\infty,\infty)$. 
To prove \eqv(1'.theo3.2') we proceed as follows. 
First, observe that $M_n(t) $ is an increasing process and that, by \eqv(2.4.4) of Lemma \thv(Lem2.4.1), 
$
(k_n(\cdot)c_n/a_n)e^{-n\beta^2/2}\E(m_n)
$
converges to a continuous limit. 
Hence, $K\equiv L(a_n/c_n)e^{n\beta^2/2}$ for any $L>0 $ control points suffice to establish the desired convergence.
To make this more precise let $t_1, \dots, t_K$ be an equidistant partition of $[0,T]$. Then for any $t\in[0,t]$ there exists 
$1\leq i\leq K$ such that $t_i\leq t\leq t_{i+1}$,
\be
\nonumber
M_n(t_i)\leq M_n(t)\leq M_n(t_{i+1}) \,\,\,\mbox{and}\,\,\, \vert({k_n(t_i)}/{a_n})\E m_n- ({k_n(t_{i+1})}/{a_n})\E m_n\vert \leq \sfrac{T}{L}.
\ee
Hence it suffices to prove that
\be \Eq(Bonn.3)
\mathcal{P}\left(\exists i \in \{1, \dots, K\} :\left \vert M_n(t_i) -({k_n(t_i)}/{a_n})\E m_n\right\vert\geq \epsilon \right)
\ee
converges to zero as $n \to \infty$ $\mathbb{P}$-a.s.~(resp.~in $\mathbb{P}$-probability). By the linearity of $({k_n(t_i)}/{a_n}) m_n$ and $({k_n(t_i)}/{a_n})\E m_n$ (see \eqv(As.1)), and Lemma \thv(4.1'')  it suffices to consider
\be \Eq(Bonn.4)
\mathcal{P}\left(\exists i \in \{1, \dots, K\} :\left\vert  M_n(t_i) -({k_n(t_i)}/{a_n})m_n\right\vert \geq \epsilon \right).
\ee
Using a union bound the probability in \eqv(Bonn.3) is bounded from above by
\be\Eq(Bonn.5)
\sum_{i=1}^K \mathcal{P}\left( \left \vert M_n(t_i) -({k_n(t_i)}/{a_n}) m_n \right\vert \geq \epsilon \right)
\ee
Under the assumption $\lim_{n\to\infty}\sqrt{n}\beta-\frac{\log c_n}{\sqrt{n}\beta}=\theta$, by Lemma \thv(Th.Scale), 
$K=CL\sqrt{n}(1+o(1))$.
The claim now follows from Proposition \thv(4.prop1') and Lemma \thv(4.1'').
As before, this convergence holds either $\P$-a.s.~or in $\P$-probability depending on whether $\sum_{n}a_n/2^n$ if finite or not.
This completes the proof of Theorem \thv(1'.theo1).  
\end{proof}

\begin{proof}[Proof of Proposition \thv(LCor.3)]
We start with proving Part (i) of Proposition \thv(LCor.3). Proceeding as in the proof of \eqv(1'.theo3.2') in Assertion (ii) of Theorem \thv(1'.theo1), we observe that the expectation with respect to the random environment of
$
\frac{k_n(\cdot)}{a_n} \frac{ m_n(\cdot)}{\E m_n(\cdot)} 
$  
converges to $t$ as $n\uparrow \infty$, which is obviously continuous in $t$, and that $\frac{M_n(t)}{\E(m_n)}$ is increasing in $t$. Hence, to establish \eqv(1'.theo3.2'.1)  it suffices to prove  convergence of the finite dimensional distributions. This follows from Proposition \thv(4.prop1') and Lemma \thv(4.1''). As $\E(\EE(M_n(t))$ diverges as $n\uparrow \infty$, by Lemma \thv(Lem.As), the second part of Proposition \thv(LCor.3) follows from Assertion (ii) of Theorem \thv(1'.theo1).

Next, we turn to Part (ii) of Proposition \thv(LCor.3). We rewrite $S_n(t)$ as
\be\Eq(Mainz.110)
\sum_{i=1}^{[a_nt]} \g_n(J_n(i))e_{n,i}\1_{\{\g_n(J_n(i))e_{n,i}\leq 1\}} +\sum_{i=1}^{[a_nt]} \g_n(J_n(i))e_{n,i}\1_{\{\g_n(J_n(i))e_{n,i}> 1\}}\equiv S_{n,1}(t)+S_{n,2}(t)
\ee
Using Markovs inequality we bound the probability that the second summand in \eqv(Mainz.110) is larger than $\e$ by
\be\Eq(Mainz.111)
\e^{-1}E\left(S_{n,2}(t)\right).
\ee
Using again Markovs inequality (this time with respect to $\P$ we bound the probability that the expectation in \eqv(Mainz.111) is larer than $\e^2$ by
\be\Eq(Main.112)
\e^{-2}\E E\left(S_{n,2}(t)\right)=\e^{-2} (1+o(1))\frac{ e^{n\b^2/2}}{c_n\sqrt{2\pi}}\int^{\infty}_{\frac{\log n}{\sqrt{n}\b} -\frac{\log c_n}{\b \sqrt{n}}+\sqrt{n}\b}e^{-y^2/2}\mbox{d}y,
\ee
see the computations preceding \eqv(NY.113). As $\b<\b_c$ it follows from \eqv(Main.112) and  Gaussian tail-bounds that $S_{n,2}(t){c_n e^{n\b^2/2}}{a_n}$ converges to zero $\P$ a.s. as $n\uparrow \infty$. 

Turning to $S_{n,1}(t)$ we observe that
\be \Eq(Mainz.100)
\E\left(\sum_{i=1}^{[a_nt]} \g_n(J_n(i))e_{n,i}\1_{\{\g_n(J_n(i))e_{n,i}\leq 1\}}\right)=\EE(M_n(t))
\ee
and $\log\frac{a_ne^{\beta^2}}{c_n}=n(\beta-\beta_c)^2/2(1+o(1))$. The analogous statement to Proposition \thv(4.prop1') holds for $\alpha>1$ with $\epsilon$ replaced by $\frac{\epsilon}{a_ne^{n\b^2}/c_n}$. 
Noting that the bounds used in the first  moment computation in the proof of Lemma \thv(4.1'') (iii) still hold, we see that $S_n^1(t)$ concentrates $P$-almost surely around its expectation with respect to $P$. The almost sure concentration of $m_n$ with respect to the random environment $\P$ follows from a second moment computation as in the proof of Lemma \thv(4.1'') (i) as long as $2\b>\b_c$ (as \eqv(Mainz.114) is used which requires this condition).
For $2\b<\b_c$ a similar computation works as one can use the truncation $\1_{\g_n(J_n(i))e_{n,i}\leq e^{(\b+\delta)\beta n}}$ for some $\delta>0$ instead of  $\1_{\{\g_n(J_n(i))e_{n,i}\leq 1\}}$ in \eqv(Mainz.110). We omit details as the computations are a rerun of the first and second moment computation done for the other truncation.
The claim of Proposition \thv(LCor.3) (ii) now follows from the above estimates as $S_n(t)$ is increasing in $t$ and the limit $t$ is obviously continuous in $t$.
 \end{proof}

\section{Proofs of Theorem \thv(1.theo1) and Theorem \thv(TCor.2) on correlation functions.}
\label{S8}

In this section we give the proofs of the results of Section \thv(S1.3) that are obtained on intermediate scales. Those obtained on extreme scales are given in Subsection \thv(S9.4).

\begin{proof}[Proof of Theorem \thv(1.theo1)]
This is a direct consequence of  \eqv(1.prop1.2) of \thv(1'.theo1) and Dynkin-Lamperti Theorem in continuous time (see e.g. Theorem 1.8 in  \cite{G12}) since under the assumptions of Theorem \thv(1.theo1), $S^{int}$ is a stable subordinator of index $0<\a(\varepsilon)<1$.
\end{proof}

 \begin{proof}[Proof of Theorem \thv(hightemp)]
This is a direct consequence of  the control of $S_{n,2}$ in the proof of Proposition \thv(LCor.3). In particular of \eqv(Mainz.111) and \eqv(Main.112) which show that the contribution of the jumps larger than $c_n$ to the clock process $S_n$ converges to zero. 
\end{proof}

Let us outline the proof of Theorem \thv(TCor.2). The main idea is to compute for each $k$ the probability that the size of the $(k+1)$-th jump is large enough to straddle over the desired interval. If $k$ is too large, the sum of the small jumps up to the $k$-th one is already larger than $c_nt$, hence we can exclude this possibility. For the other jumps, we use that Proposition \eqv(LCor.3) provides a precise control on the clock process up to the $(k+1)$-th jump (in the supremum norm) together with the fact, which essentially follows from Condition (A1), that we know how likely such a big jump is.
 
\begin{proof}[Proof of Theorem \thv(TCor.2)]  We define an auxiliary time-scale $\widetilde{a}_{n}$ by
\be\Eq(Cor.2.1')
\widetilde{a}_{n}e^{n\b^2/2}\Phi(\theta)= c_n,
\ee
A crucial quantity is the ratio $\widetilde{a}_n/a_n$ which is by Lemma \thv(Th.Scale) given by
\be\Eq(Cor.2.1'')
\frac{\widetilde{a}_n}{a_n}= \frac{e^{-\theta^2/2}}{\Phi(\theta)\b\sqrt{2\pi n}} (1+o(1)).
\ee
Set
\be\Eq(NY.108)
A_n(t) \equiv \PP\left(\sup_{k\in\{1,\dots,\lfloor\widetilde{a}_nt\rfloor\}}\left \vert\widetilde{S}_n(k)-\frac{c_nk}{\widetilde{a}_n}\right \vert >c_n\e\right).
\ee
Fix a realization of the random environment such that for all $t,T>0$, for all $x>s$ uniformly in $x$ and for all $\e>0$
\be\Eq(C2) 
\lim_{n\to\infty}A_n(t) =  0,
\ee
\be\Eq(C3)
\lim_{n\to\infty}  \sup_{k\geq \theta_n}\left \vert \lfloor a_nt\rfloor\PP\left(\t_n(J_n(k+1))e_{n,k+1}> c_n x\right)-\frac{t}{x}\right \vert =0,
\ee
and to take care of the first $\theta_n$ jumps
\be\Eq(C5)
\lim_{n\to\infty}\frac{a_n}{\widetilde{a}_n}\sum_{l=1}^{\theta_n}\PP\left(\t_n(J_n(k+1)e_{n,k+1}>c_n x\right)=0
\ee
and
\be\Eq(C1)
\lim_{n\to\infty} \sqrt{n}\PP\left(\left\vert\sum_{k=1}^{\lfloor\widetilde{a}_nt\rfloor}\PP(\t_n(J_n(k+1))e_{n,k+1}>c_ns|J_n(k))-\frac{\widetilde{a}_nt}{a_ns}\right\vert>\frac{\widetilde{a}_n\e}{a_n\sqrt{A_n(t)}}\right) =0,
\ee
\be\Eq(C4)
\lim_{n\to\infty}\sqrt{n}\PP\left(  M_n(\lfloor\widetilde{a}_nt(1+\e')/a_n\rfloor)  < c_n\right) =0.
\ee
The terms \eqv(C5), \eqv(C1) and \eqv(C4), which resemble terms studied earlier but depend 
on the auxiliary time-scale, are studied in Appendix \ref{app.3}. Rewriting the correlation function gives
\be\Eq(Cor.2.1)
\CC_n (t,s) =  \PP\left(\bigcup_{k>0} \widetilde{S}_n(k) < c_nt ,  \widetilde{S}_n(k+1)>c_n(t+s)\right).
\ee 
Let  $Q_k$ be the event that the $k$-th jump has the desired height, namely 
\be\Eq(NY.105)
Q_k= \{\tau_n(J_n(k+1))e_{n,k+1}> c_n(t+s) -\widetilde{S}_n(k)\}.
\ee
Then we can rewrite \eqv(Cor.2.1) as
\be\Eq(NY.106)
  \PP\left(  \bigcup_{k\leq \widetilde{a}_{n}t(1+\e') }\left\{\widetilde{S}_n(k)<c_nt\right\}\cap Q_k \right)  + \PP\left(  \bigcup_{k> \widetilde{a}_{n}t(1+\e') }\left\{\widetilde{S}_n(k)<c_nt\right\}\cap Q_k  \right).
\ee
 Using that  $\widetilde{S}_n(k)$ is an increasing  process we have for all $\e>0$ 
\be\Eq(Cor.2.2)
\PP \left( \bigcup_{k>\widetilde{a}_{n}t(1+\e')}\left\{\widetilde{S}_n(k)\leq c_nt\right\}\right)
\leq \PP\left(\frac{\widetilde{S}_n(\lfloor\widetilde{a}_nt(1+\e')\rfloor) }{c_n }< t\right) . 
\ee
By only counting summands smaller than $\delta$ 
we bound \eqv(Cor.2.2) from above by
\be\Eq(Cor.2.2.1)
\PP\left(  M_n(\lfloor\widetilde{a}_nt(1+\e')/a_n\rfloor)\leq t\right),
\ee
which is of order $o(1/\sqrt{n})$ by \eqv(C4).
Using this together with a union bound we can rewrite \eqv(Cor.2.1) as 
\be\Eq(NY.107)
\sum_{k\in T} \PP\left(\left\{\widetilde{S}_n(k)<c_nt\right\}\cap Q_k\right)+o(1/\sqrt{n}).
\ee
We rewrite the first summand in \eqv(NY.107) as
\bea
\Eq(Cor.2.3)
&& \sum_{k\in T}\left\{ \EE\left(\1_{\left\{\widetilde{S}_n(k)<c_nt\right\}\cap Q_k}
 \1_{\left\{\left|\widetilde{S}_n(k)-\frac{c_nk}{\widetilde{a}_n}\right|>c_n\e\right\}}  \right)\right.\cr
&&\quad \quad+
\left.\EE\left(\1_{\left\{\widetilde{S}_n(k)<c_nt\right\}\cap Q_k}
\1_{\left\{\left|\widetilde{S}_n(k)-\frac{c_nk}{\widetilde{a}_n}\right|\leq c_n\e\right\}}   \right)\right\} 
\eea
where $T=\{0,\dots,\lfloor\widetilde{a}_n(1+\e)\rfloor\}$.
We want to show the first summand in \eqv(Cor.2.3) are of order $o(1/\sqrt{n})$.  By using the worst bound on $\widetilde S_n(k)$ to bound $Q_k$ and then dropping the condition $\widetilde S_n(k)\leq c_nt$ we bound the first part of \eqv(Cor.2.3) by
\bea\Eq(Cor.2.4)
 &&\sum_{k\in T}\EE\left(\1_{\left\{\left|\widetilde{S}_n(k)-\frac{c_nk}{\widetilde{a}_n}\right|>c_n\e\right\}} 
 \1_{\{\tau_n(J_n(k+1))e_{n,k+1}> c_n(t + s) -\widetilde{S}_n(k)\}}\right) \\
&\leq &
\EE\left(\1_{\left\{\sup_{k\in [0,\widetilde{a}_n(1+\e)]\cap \N}   \left \vert\widetilde{S}_n(k)-\frac{c_nk}{\widetilde{a}_n}\right \vert >c_n\e\right\}}\sum_{k\in T}\PP\left(\t_n(J_n(k+1))e_{n,k+1}>c_ns|J_n(k)\right)\right)\nonumber ,
\eea
where we conditioned on the process up to time $k$ and used that $J_n$ is a Markov chain.
Using \eqv(C1) to bound the conditional probability given $J_n(k)$, \eqv(Cor.2.4) is bounded above by
\bea\Eq(Cor.2.5')
&&\frac{\widetilde{a}_n}{a_n}\left(\frac{t}{s}+\frac{1}{\sqrt{A_n(t)}}\e\right)\PP\left(\sup_{k\in [0,\dots,\widetilde{a}_nt(1+\e)]} \left \vert\widetilde{S}_n(k)-\frac{c_n k}{\widetilde{a}_n}\right \vert >c_n\e\right)+o(1/\sqrt{n})\nonumber\\
&=&\frac{\widetilde{a}_n}{a_n}\frac{t}{s}A_n(t(1+\e))+ \frac{\widetilde{a}_n}{a_n}\sqrt{A_n(t(1+\e))}\e+o(1/\sqrt{n}),
\eea
where $A_n(t(1+\epsilon))$ is exactly defined to be the probability appearing in the first line of \eqv(Cor.2.5').
By \eqv(C2) we have that \eqv(Cor.2.5') is of order $o(\widetilde{a}_n/a_n)=o(1/\sqrt{n})$. We can bound the second summand in \eqv(Cor.2.3) from above by
\bea\Eq(Cor.2.5)
&&\sum_{k\in [0,\dots, \lfloor\widetilde{a}_nt(1-\e)\rfloor]}\PP\left(\t_n(J_n(k+1)e_{n,k+1}> c_n(t+s)-c_n\left(\frac{k}{\widetilde{a}_n}+\e\right)\right)\nonumber\\
&&\quad+2\e \widetilde{a}_n \PP\left(\t_n(J_n(k+1))e_{n,k+1}> c_ns\right).
\eea
Define the functions 
\be\Eq(NY.110)
F_{\pm}(k)= \frac{1}{t+s-(k\pm \epsilon) }.
\ee
Using \eqv(C5) for the first $\theta_n$ summands and \eqv(C3) for $k\geq \theta_n$ we can rewrite \eqv(Cor.2.5) as
\be\Eq(Cor.2.6)
\sum_{k\in [\theta_n,\dots, \lfloor\widetilde{a}_nt(1-\e)\rfloor]}\frac{1}{a_n}  F_+\left(\frac{k}{\widetilde a_n}\right)+2\e \frac{\widetilde{a}_n}{a_n} +\frac{\widetilde{a}_n}{a_n}o(1).
\ee
In the same way we obtain the following lower bound on \eqv(Cor.2.3),
\be\Eq(Cor.2.7)
\sum_{k\in [\theta_n,\dots, \lfloor\widetilde{a}_nt(1+\e)\rfloor]}\frac{1}{a_n} F_-\left(\frac{k}{\widetilde a_n}\right)-2\e \frac{\widetilde{a}_n}{a_n} -\frac{\widetilde{a}_n}{a_n}o(1).
\ee
By a Riemann sum argument we have
\be
\sum_{k=0}^{\lfloor\widetilde{a}_nt(1-\e)\rfloor\}}\frac{1}{\widetilde{a}_n} F_+\left(\frac{k}{\widetilde a_n}\right)-\frac{1}{\widetilde{a}_n } \frac{1}{t+s}
\geq \int_0^{t(1-\e)}F_+(u)\mbox{d}u
= \log\left(\frac{s+t(1-\e)}{s}\right),
\ee
respectively,
\be
\sum_{k=0}^{\lfloor\widetilde{a}_nt(1+\e)\rfloor}\frac{1}{\widetilde{a}_n}F_-\left(\frac{k}{\widetilde a_n}\right)- \frac{1}{\widetilde{a}_n}\frac{ 1}{s }
\leq \int_0^{t(1+\e)}F_-(u)\mbox{d}u
=\log\left(\frac{s+t(1+\e)}{s}\right).
\ee
Noting that 
\be
\sum_{k=0}^{\theta_n}\frac{1}{\widetilde{a}_n}F_{\pm}\left(\frac{k}{\widetilde a_n}\right)\leq\frac{\theta_n}{\widetilde{a}_n} F_{\pm}\left(\frac{\theta}{\widetilde a_n}\right)\leq \frac{\theta_n}{\widetilde{a}_n} \frac{1}{s},
\ee 
and $\frac{\theta_n}{\widetilde{a}_n}\ll \frac{\widetilde{a}_n}{a_n}$, we can bound \eqv(Cor.2.6) from above by
\be
\frac{\widetilde{a}_n}{a_n}\log\left(\frac{s+t(1-\e)}{s}\right)-\frac{1}{\widetilde{a}_n} \frac{1}{1+s}+2\e \frac{\widetilde{a}_n}{a_n} +\frac{\widetilde{a}_n}{a_n}o(1),
\ee
and \eqv(Cor.2.7) from below by
\be
\frac{\widetilde{a}_n}{a_n}\log\left(\frac{t+s+\e}{s}\right)- \frac{1}{\widetilde{a}_n}\frac{ 1}{s }-2\e \frac{\widetilde{a}_n}{a_n} -\frac{\widetilde{a}_n}{a_n}o(1).
\ee
Putting those estimates together we obtain
\be
\CC_n(t,s)= \frac{\widetilde{a}_n}{a_n}\log\left(1+\frac{t}{s}\right)+\frac{\widetilde{a}_n}{a_n}o(1).
\ee
By \eqv(Cor.2.1') we have
\be\Eq(Cor.2.8)
\CC_n(t,s)= \frac{e^{-\theta^2/2}}{\Phi(\theta)}\log\left(1+\frac{1}{s}\right)\frac{1}{\b\sqrt{2\pi n}}(1+o(1)). 
\ee
So far we worked on a fixed realization of the random environment. 
Observe that the probability in \eqv(C3) is equal to $\nu_n(x,\infty)$ 
and hence, 
by Proposition \thv(3.prop0), for all $k\geq \theta_n$ and large enough $n$ we have, for $\nu_n(u,\infty)$ defined in \eqv(4.9'),
\be 
 \lfloor a_nt\rfloor\PP\left(\t_n(J_n(k+1))e_{n,k+1}> c_n x\right)=(1+\delta_n) \nu_n(x,\infty),
 \Eq(temp)
\ee
which is a monotone function in $x$ and its limit $1/x$ is continuous for $x>0$, which implies that \eqv(C3) holds $\P$-a.s. if $\sum a_n/2^n<\infty$ and in $\P$-probability if $\sum a_n/2^n<\infty$.
By Proposition \thv(LCor.3) and Lemma \thv(LCor.4), \eqv(C1) and \eqv(C2) hold either $\P$-a.s.~or in $\P$-probability.  Eq.~\eqv(C5) holds $\P$-a.s.~by Lemma \thv(Lem.cor5) and \eqv(C4) holds either $\P$-a.s.~or in $\P$-probability by Lemma \thv(Lem.cen5). Arguing as in the proof of Theorem 1.3 in \cite{G12}, we have that if \eqv(C1),\eqv(C2), \eqv(C3) and \eqv(C5) and \eqv(C4) hold $\P$-a.s., respectively in $\P$-probability, \eqv(Cor.2.8) holds with respect to the same convergence mode with respect to the random environment. This concludes the proof of Theorem \thv(TCor.2).
\end{proof}

\section{Extreme scales.}
\label{S9}

This section closely follows Section 6 of  \cite{G12} where an approach known as ``the method of common probability space''  was first implemented to bypass the fact that on extreme time-scales, when $a_n\sim 2^n$, the convergence properties of sums such as \eqv(5.2) or \eqv(5.3) can no longer follow from classical laws of large numbers; instead, one aims at replacing the sequence of re-scaled landscapes $(\g_n(x), x\in\VV_n)$, $n\geq 1$,
by a new sequence with identical distribution and  almost sure convergence properties.
In Subsection \thv(S9.1), we give an explicit representation of such a  re-scaled
landscape which is valid for all extreme scales (Lemma \thv(6.lemma1))
and show that, in this representation, all random variables of interest
have an almost sure limit (Proposition \thv(6.prop4)). In Subsection \thv(S9.2)
we consider the model obtained by substituting the representation for
the original landscape.
For this model we state and prove the analogue of the ergodic theorem
of Section \thv(S5) (Proposition \thv(6.prop1)) and the associated chain independent
estimates of Section \thv(S6) (Proposition \thv(6.prop2)). Thus equipped we are
ready, in Subsection \thv(S9.3), to prove the results of Section \thv(S1) obtained on extreme scales.

\subsection{ A representation of the landscape.\hfill}
\label{S9.1}

The representation we now introduce
is due to Lepage {\it et al.} \cite{LWZ} and relies on an elementary property of
order statistics. We use the following notations.
Set $N=2^n$.
Let
$\bar\t_{n}(\bar x^{(1)})\geq\dots\geq\bar\t_{n}(\bar x^{(N)})$
and
$\bar\g_{n}(\bar x^{(1)})\geq\dots\geq\bar\g_{n}(\bar x^{(N)})$
denote, respectively, the landscape and re-scaled landscape variables
$\g_n(x)=c_n^{-1}\t_n(x)$, $x\in\VV_n$, arranged in decreasing order of magnitude.
As in Section 2, set $G_n(v)=\P(\tau_n(x)>v)$, $v\geq 0$, and denote by
$
G_n^{-1}(u):=\inf\{v\geq 0 : G_n(v)\leq u\}
$, $u\geq 0$, its inverse.
Also recall that $\a=\b_c/\b$ and assume that $\b>\b_c$.

Let $(E_i, i\geq 1)$ be a sequence of i.i.d.~mean one exponential random variables
defined on a common probability space $(\O, \FF, \boldsymbol{P})$.
For $k\geq 1$ set
\bea\Eq(6.4)
\G_k&=&\sum_{i=1}^k E_i,\nonumber\\
\g_k&=&\G_k^{-1/\a},
\eea
and, for $1\leq k\leq N$, $n\geq 1$, define
\be
\g_{n}(x^{(k)})=c_n^{-1}G_n^{-1}(\G_k/\G_{N+1}),
\Eq(6.5)
\ee
where $\{x^{(1)},\dots,x^{(N)}\}$ is a randomly chosen labelling of the $N$ elements of $\VV_n$,
all labellings being equally likely.

\begin{lemma}{\TH(6.lemma1)} 
 For each $n\geq 1$,
$
(\bar\g_{n}(\bar x^{(1)}),\dots,\bar\g_{n}(\bar x^{(N)}))\overset{d}=(\g_{n}(x^{(1)}),\dots,\g_{n}(x^{(N)})).
$
\end{lemma}

\begin{proof} Note that $G_n$ is non-increasing and right-continuous so that $G_n^{-1}$
is non-increasing and right-continuous. It is well known that if the random variable $U$
is a uniformly distributed on $[0,1]$ we may write $\t_n(0)\overset{d}=G_n^{-1}(U)$
(see e.g. \cite{Re}, page 4).
In turn it is well known (see \cite{F}, Section III.3) that if $(U(k), 1\leq k\leq N)$
are independent random variables uniformly distributed on $[0,1]$ then, denoting by
$\bar U_{n}(1)\leq\dots\leq \bar U_{n}(N)$ their ordered statistics,
$
(\bar U_{n}(1),\dots,\bar U_{n}(N))\overset{d}=(\G_1/\G_{N+1},\dots,\G_{N}/\G_{N+1})
$.
Combining these two facts readily yields the claim of the lemma since, by independence of
the landscape variables $\t_n(x)$, all arrangements of the $N$ variables $\G_k/\G_{N+1}$
on the $N$ vertices of $\VV_n$ are equally likely.
\end{proof}

Next, let $\Upsilon$  be the point process in $M_P(\R_+)$ which has counting function
\be
\Upsilon([a,b])=\sum_{i=1}^{\infty}\1_{\{\g_k\in [a,b]\}}.
\Eq(6.6)
\ee
\begin{lemma}{\TH(6.lemma2)} 
$\Upsilon$ is a Poisson random measure on $(0,\infty)$
with mean measure $\mu$ given by \eqv(1.theo2.0).
\end{lemma}

\begin{proof}
The point process 
$
\G=\sum_{i=1}^{\infty}\1_{\{\G_k\}}
$
defines a homogeneous Poisson random measure on $[0,\infty)$ and thus,
by the mapping theorem (\cite{Re}, Proposition 3.7), setting $T(x)=x^{-1/\a}$ for $x>0$,
$\Upsilon=\sum_{i=1}^{\infty}\1_{\{T(\G_k)\}}$ is Poisson random measure on $(0,\infty)$
with mean measure $\mu(x,\infty)=T^{-1}(x)$.
\end{proof}

We thus established that both the ordered landscape variables and the point process $\Upsilon$
can be expressed in terms of the common sequence $(E_i, i\geq 1)$ and thus,
on the common probability space $(\O, \FF, \boldsymbol{P})$. 
As shown by the next proposition, on that space, the random variables of interest will have an almost sure limit.

\begin{proposition}{\TH(6.prop4)} 
Assume that $\a<1$.
Let $c_n$ be an extreme time-scale. Let $f:(0,\infty)\rightarrow[0,\infty)$ be
a continuous function that obeys
\be
\int_{(0,\infty)}\min(f(u), 1)d\mu(u)<\infty.
\Eq(6.prop4.1)
\ee
Then, $\boldsymbol{P}$-almost surely,
\be
\lim_{n\rightarrow\infty}\sum_{k=1}^{N}f(\g_{n}(x^{(k)}))=\sum_{k=1}^{\infty}f(\g_{k})<\infty.
\Eq(6.prop4.2)
\ee
\end{proposition}

\begin{proof}[Proof of Proposition \thv(6.prop4)]
The proof of Proposition \thv(6.prop4) closely follows that of Proposition 7.3 of \cite{G12}, which itself is
strongly inspired from the proof of Proposition 3.1 of \cite{FIN}. We omit the details.
\end{proof}

\subsection{Preparations to the verification of Conditions (A1), (A2) and (A3).\hfill}
\label{S9.2}

Consider the model obtained by substituting the representation
$(\g_{n}(x^{(i)}), 1\leq i\leq N)$ for the original re-scaled landscape $(\g_n(x), x\in\VV_n)$.
The aim of this subsection is to prove
the homologue, for this model, of the ergodic theorem (Proposition \thv(4.prop1))
and chain independent estimates (Proposition \thv(5.prop6)) of Section \thv(S5) and Section \thv(S6).

In order to distinguish the quantities $\nu_n^{J,t}(u,\infty)$, $\s_n^{J,t}(u,\infty)$,
$\nu_n(u,\infty)$ and $\s_n(u,\infty)$, expressed in \eqv(4.7)--\eqv(4.4') in the original
landscape variables, from their expressions in the new ones , we call the latter
$\boldsymbol{\nu}_n^{J,t}(u,\infty)$, 
$(\boldsymbol{\sigma}_n^{J,t})(u,\infty)$,
$\boldsymbol{\nu}_n(u,\infty)$ and $\boldsymbol{\sigma}_n(u,\infty)$ respectively. Their definition is otherwise unchanged.

\begin{proposition}{\TH(6.prop1)} 
There exists a subset $\O_{0}\subset\O$ such that
$\boldsymbol{P}(\O_{0})=1$
and such that, on $\O_{0}$, for all large enough $n$,
the following holds: for all $t>0$ and all $u>0$,
\be
P_{\pi_n}\left(\left|\boldsymbol{\nu}_n^{J,t}(u,\infty)-E_{\pi_n}\left[\boldsymbol{\nu}_n^{J,t}(u,\infty)\right]\right|\geq\e\right)
\leq
\e^{-2}\Theta_n(t,u),\quad\forall\e>0,
\Eq(6.prop1.1)
\ee
where
\bea\Eq(6.prop1.2)
\Theta_n(t,u)&=&
\left(\frac{k_n(t)}{a_n}\right)^2\frac{\boldsymbol{\nu}_n^2(u,\infty)}{2^n}+
\frac{k_n(t)}{a_n} \left(\boldsymbol{\sigma}_n(u,\infty) +c_1\frac{\boldsymbol{\nu}_n(2u,\infty)}{ n^{2}}\right)
\nonumber\\
&+&\frac{k_n(t)}{a_n}\left(3\theta_ne^{-u/\delta_n}\boldsymbol{\nu}_n(u,\infty)+\frac{2^n}{a_n}\boldsymbol{\nu}_n^2(u,\infty)e^{-c_2u}\right),
\eea
for some constants $0<c_1,c_2<\infty$, where $\delta_n\leq n^{-\a(1+o(1))}$, and where $\theta_n$ is defined as in \thv(3.prop1.1).
In addition, for all $t>0$ and all $u>0$,
\be
P_{\pi_n}\left((\boldsymbol{\sigma}_n^{J,t})(u,\infty)\geq\e'\right)\leq\frac{k_n(t)}{\e'\, a_n}\boldsymbol{\sigma}_n(u,\infty),\quad\forall\e'>0.
\Eq(6.prop1.3)
\ee
\end{proposition}

\begin{proposition}{\TH(6.prop2)} 
Let $c_n$ be an extreme time-scale. Assume that  $\b\geq\b_c$ and  let $\nu^{ext}$ be defined in \eqv(1.prop2.2).
There exists a subset $\O_{1}\subset\O$ such that $\boldsymbol{P}(\O_{1})=1$ and such that, on $\O_{1}$,
the following holds: for all $u>0$,
\bea
&&\lim_{n\rightarrow\infty}\boldsymbol{\nu}_n(u,\infty)=\nu^{ext}(u,\infty)<\infty,
\\
&&\lim_{n\rightarrow\infty}\boldsymbol{\sigma}_n(u,\infty)=0,
\Eq(6.prop2.1)
\\
 &&\lim_{\d\to 0}\lim_{n\to\infty}\frac{a_n}{2^n} \sum_{k=1}^{N}g_\d(\g_{n}(x^{(k)}))=0.
 \Eq(5.prop6.1bis)
\eea
\end{proposition}

\noindent Proposition \thv(6.prop2) is a straightforward application of Proposition \thv(6.prop4) and Lemma \thv(1.lemma5)
whose proof we skip (see also \cite{G12}, (6.32)-(6.35) for a pattern of proof).

\begin{proof}[Proof of Proposition \thv(6.prop1)] This  is a rerun of the proof
Proposition \thv(4.prop1). The only difference is in the treatment of the term \eqv(4.prop1.19).
In the new landscape variables, Lemma \eqv(4.lemma1) is not true, and its method of proof is
unadapted. To bound \eqv(4.prop1.19) we proceed as follows.
Let $T_n:=\{x^{(k)}, 1\leq k\leq n\}\subset\VV_n$ be the set of the $n$ vertices
with largest $\g_{n}(x)$. The next two lemmata collect elementary properties of $T_n$.

\begin{lemma}{\TH(6.lemma3)} 
There exists a subset $\O_{0,1}\subset\O$ with
$\boldsymbol{P}(\O_{0,1})=1$ such that, for all $\o\in\O_{0,1}$, for all large enough $n$, the following holds:
for all $x, x'\in T_n$, $x\neq x'$,
$
\dist(x,x')=\frac{n}{2}(1-\rho_n)
$
where  $\rho_n=\sqrt{\frac{8\log n}{n}}$.
\end{lemma}

\begin{proof}  Given $t>0$ consider the event
$
\O_{0,1}(n)=\left\{\exists_{1\leq k\neq k'\leq n}: \left|\dist\bigl(x^{(k)},x^{(k')}\bigr)-\frac{n}{2}\right|\geq t\right\}
$.
By construction, the elements of $T_n$ are drawn at random from $\VV_n$,
independently and without replacement. Hence
\be
\textstyle
\boldsymbol{P}\left(\O_{0,1}(n)\right)
\leq n^2\boldsymbol{P}\left(\left|\dist\bigl(x^{(1)},x^{(2)}\bigr)-\frac{n}{2}\right|\geq t\right)
\sim n^2 P\left(\left|\sum_{i=1}^n\varepsilon_i-\frac{n}{2}\right|\geq t\right),
\Eq(6.prop1.4)
\ee
where $(\varepsilon_i, 1\leq i\leq n)$ are i.i.d. r.v.'s taking value 0 and 1 with probability $1/2$.
A classical exponential Tchebychev inequality yields
$
P\left(\left|\sum_{i=1}^n\varepsilon_i-\frac{n}{2}\right|\geq t\right)\leq e^{-\frac{t^2}{2n}}
$.
Choosing $t=\sqrt{8n\log n}$, and plugging into \eqv(6.prop1.4),
$
\boldsymbol{P}\left(\O_{0,1}(n)\right)\leq n^{-2}
$.
Setting $\O_{0,1}=\cup_{n_0}\cap_{n>n_0}\O_{0,1}(n)$,
the claim of the lemma follows from an application of Borel-Cantelli Lemma.
\end{proof}

\begin{lemma}{\TH(6.lemma4)} 
There exists a subset $\O_{0,2}\subset\O$ with
$\boldsymbol{P}(\O_{0,2})=1$ such that, for all $\o\in\O_{0,2}$, for all large enough $n$,
$
\sup\{\g_{n}(x),x\in\VV_n\setminus T_n\}\leq \delta_n
$
where $\delta_n=(1+o(1))n^{-\a(1+o(1))}$.
\end{lemma}

\begin{proof} Clearly
$
\sup\{\g_{n}(x),x\in\VV_n\setminus T_n\}
=\sup\{\g_{n}(x^{(k)}),k>n\}
=\g_{n}(x^{n+1})
$,
and by \eqv(6.5), 
$
\g_{n}(x^{n+1})
=c_n^{-1}G_n^{-1}\bigl(\frac{\G_{n+1}}{\G_{N+1}}\bigr)
$.
By the strong law of large numbers applied to both $\G_{n+1}$ and $\G_{N+1}$,
we deduce that there exists a subset $\O_{0,2}\subset\O$
of full measure such that, for all  $n$ large enough and all  $\o\in\O_{0,2}$,
$
\g_{n}(x^{n+1})
=c_n^{-1}G_n^{-1}\bigl((n/b_n)(1+\l_n)\bigr)
$.
By definition of $h_n(v)$ (see \eqv(2.13)),
$
c_n^{-1}G_n^{-1}(h_n(v))=v
$,
and by Lemma \thv(2.lemma6),
$
\g_{n}(x^{n+1})=(1+o(1))n^{-\a(1+o(1))}
$.
\end{proof}

We are now equipped to bound $(III)_{2,l}$. Set $\O_{0}=\O_{0,1}\cap\O_{0,2}$.
Writing $T_n^c\equiv\VV_n\setminus T_n$, and setting
$
f(y,z)=k_n(t)\pi_n(y)e^{-u[\g^{-1}_n(y)+\g^{-1}_n(z)]}p_n^{l+2}(y,z)
$,
we may decompose $(III)_{2,l}$ it into four terms,
\be
\sum_{z\in T_n^c, y\in T_n^c:y\neq z}f(y,z)
+\sum_{z\in T_n^c, y\in T_n}f(y,z)
+\sum_{z\in T_n, y\in T_n^c }f(y,z)
+\sum_{z\in T_n, y\in T_n:y\neq z}f(y,z).
\Eq(6.prop1.5)
\ee
To bound the first sum above we use that, by Lemma \thv(6.lemma4), for $y\in T_n^c$,
$
e^{-u[\g^{-1}_n(z)+\g^{-1}_n(y)]}\leq e^{-u/\g_n(z)}e^{-u/\delta_n}
$.
Thus,
\bea\Eq(6.prop1.6)
\sum_{z\in T_n^c, y\in T_n^c:y\neq z}f(y,z)
&\leq&
e^{-u/\delta_n}\sum_{z\in T_n^c}
a_n\pi_n(z) e^{-u/\g_n(z)}\sum_{y\in T_n^c:y\neq z}p_n^{l+2}(y,z)
\nonumber\\
&\leq&
e^{-u/\delta_n}\sum_{z\in T_n^c}
a_n\pi_n(z) e^{-u/\g_n(z)}
\\
&\leq&
e^{-u/\delta_n}\boldsymbol{\nu}_n(u,\infty).\nonumber
\eea
The second and third sums of \eqv(6.prop1.5) are bounded just in the same way.
To deal with the last sum we use that in view of Lemma \thv(6.lemma3) the assumptions of
Proposition \thv(3.prop4) are satisfies. Consequently
\bea\Eq(6.prop1.7)
\sum_{l=1}^{\theta_n-1}\sum_{z\in T_n, y\in T_n:y\neq z}f(y,z)
&\leq &\frac{2^n}{a_n}\Bigl[a_n\sum_{z\in T_n}\pi_n(z)e^{-u/\g_n(z)}\Bigr]^2\sum_{l=1}^{\theta_n-1}p_n^{l+2}(y,z),
\nonumber\\
&\leq& e^{-cn}\frac{2^n}{a_n}(\boldsymbol{\nu}_n(u,\infty))^2,
\eea
for some constant $0<c<\infty$.
Collecting \eqv(6.prop1.5), \eqv(6.prop1.6) and \eqv(6.prop1.7), and summing over $l$, we finally get,
\be
\sum_{l=1}^{\theta_n-1}(III)_{2,l}\leq 3\theta_n e^{-u/\delta_n}\boldsymbol{\nu}_n(u,\infty)+
e^{-cn}\frac{2^n}{a_n}(\boldsymbol{\nu}_n(u,\infty))^2.
\Eq(6.prop1.8)
\ee

\noindent Proposition \thv(6.prop1) is now proved just as Proposition \thv(4.prop1),
using the bound \eqv(6.prop1.8) instead of the bound  \eqv(4.lem1.2) of Lemma \eqv(4.lemma1).
\end{proof}

\subsection{Proofs of the results of Section 1: the case of extreme scales.\hfill}
\label{S9.3}
We now prove the results of Section 1 that are concerned with extreme scales, namely,
Theorem \thv(1.theo2), Theorem \thv(1'.theo2) and Lemma \thv(1.lemma5).
Again our key tool will be Theorem \thv(1.theo3) of Subsection 1.4.

We assume throughout this section that $c_n$ is an extreme time-scale and that  $\b>\b_c$.
\begin{proof}[Proof of Theorem \thv(1'.theo2)]
Consider the model obtained by substituting the representation\\
$(\g_{n}(x^{(i)}), 1\leq i\leq N)$ for the original landscape $(\g_n(x), x\in\VV_n)$.
Let $\wt{\bold S}_n$, $\boldsymbol{\sigma}_n$, and  $\boldsymbol{\cal C}_{n}(t,s)$
denote, respectively, the clock process \eqv(1.1.6), the re-scaled clock process \eqv(G1.3.2'),
and the time correlation function \eqv(1.1.8) expressed in the new landscape variables.
Choose $\nu=\nu^{ext}$ in Conditions (A1),  (A2) and (A3)
(that is, in \eqv(G1.A1), \eqv(G1.A2) and \eqv(G1.A3), expressed of course in the new landscape variables).
By Proposition \thv(6.prop1) and Proposition \thv(6.prop2),
there exists a subset $\O_{2}\subset\O$ with
$\boldsymbol{P}(\O_{2})=1$, such that, on $\O_{2}$,  Conditions (A1), (A2), (A3), and (A0') are satisfied.
By  \eqv(G1.3.theo1.1) of Theorem \thv(1.theo3) we thus have that, on $\O_{2}$,
$\boldsymbol{\sigma}_n\Rightarrow  S^{ext}$
where $S^{ext}$ is the (random) subordinator of L\'evy measure
$\nu^{ext}$. This proves Theorem \thv(1'.theo2).
\end{proof}
 
\noindent It now remains to prove Lemma \thv(1.lemma5).

\begin{proof}[Proof of Lemma \thv(1.lemma5)] To ease the notation set $\bar\varepsilon=1$.
Set $u^{-\a}=M$ and $f(x)=e^{-1/x}$. By \eqv(1.prop2.2) we may write
\be
\textstyle
u^{\a}\nu^{ext}(u,\infty)=\frac{1}{M}\sum_{k=1}^{\infty}f(M^{1/\a}\g_k).
\Eq(1.lem5.3)
\ee
An easy re-run of the proof of Lemma 3.10 in \cite{G12} then yields that
\be
\textstyle
\lim_{M\rightarrow\infty}\frac{1}{M}\sum_{k=1}^{\infty}f(M^{1/\a}\g_k)=\a\G(\a)\,\,\,\text{$\boldsymbol{P}$-almost surely,}
\Eq(1.lem5.4)
\ee
\end{proof}

\subsection{Proof of  Theorem \thv(1.theo1.ext) and Theorem \thv(1.theo2).}
\label{S9.4}
We are now ready to give the proofs of the results of Section \thv(S1.3) that are obtained on extreme time-scales.

\begin{proof}[Proof of  Theorem \thv(1.theo1.ext)]
To prove Theorem \thv(1.theo1.ext) first note that by Lemma \thv(6.lemma1),
\be
\CC_n(t,s)\overset{d}=\boldsymbol{\cal C}_{n}(t,s)\,\,\,\text{for all $n\geq 1$ and all $t,s>0$.}
\Eq(1.prop1.p4)
\ee
Next, by \eqv(1.prop2.1) of Theorem \thv(1'.theo2) we have that, on $\O_{2}$,
\be
\lim_{n\rightarrow\infty}\boldsymbol{\cal C}_{n}(t,s)=\CC^{ext}_{\infty}(t,s)\quad\forall\, 
t,s>0,
\Eq(1.prop1.p1)
\ee
where
$
\CC^{ext}_{\infty}(t,s)=\PP\left(\left\{ S^{ext}(u),u>0\right\}
\cap (t, t+s)=\emptyset\right)
$.
By Lemma \thv(1.lemma5) there exists a subset $\O_{3}\subset\O$ with
$\boldsymbol{P}(\O_{3})=1$, such that, on $\O_{3}$, $\nu^{ext}$ is
regularly varying at infinity with index $-\a$. Thus, by Dynkin-Lamperti Theorem in continuous time
applied for fixed $\o\in\O_{3}$  (see e.g. Theorem 1.8 in  \cite{G12})
we get that,
\be
\lim_{t\rightarrow 0+}\CC^{ext}(t,\rho t)=\asl_{\a}(1/1+\rho)\quad\forall\,\rho>0.
\Eq(1.prop1.p2)
\ee
By \eqv(1.prop1.p4) with  $s=\rho t$, using in turn \eqv(1.prop1.p1)
and \eqv(1.prop1.p2) to pass to the limit $n\rightarrow\infty$ and $t\rightarrow 0+$,
we obtain that for all $\rho>0$,
$
\lim_{t\rightarrow 0+}\lim_{n\rightarrow\infty}\CC^{ext}_n(t,\rho t)\overset{d}=\asl_{\a}(1/1+\rho)
$.
Since convergence in distribution to a constant implies convergence in probability,
the claim of Theorem \thv(1.theo1), (iii) follows.
\end{proof}

\begin{proof}[Proof of Theorem \thv(1.theo2)] This is a re-run of the proof of Theorem 3.5 of \cite{G12} (setting $a=0$).
Note indeed that for all  $\b>\b_c$, $0< \a< 1$, which implies that
$
\int_{0}^{\infty}\nu^{ext}(u,\infty)du=\sum_{k=1}^{\infty}\g_k<\infty$
$\boldsymbol{P}$-almost surely.
We are thus in the realm of ``classical'' renewal theory,
in the so-called ``finite mean life time'' case. The second and first assertions of
Theorem \thv(1.theo2) then follow, respectively, from Assertion (ii) of Theorem 1.8 of \cite{G12} and Assertion (ii) of Theorem 7.3  of \cite{G12} on delayed subordinator.
Their proofs use the following two elementary facts: firstly
\be
\frac{\int_{s}^{\infty}\nu^{ext}(u,\infty)du}{\int_{0}^{\infty}\nu^{ext}(u,\infty)du}={\CC}^{sta}_{\infty}(s),\quad u>0,
\Eq(1.lem6.1)
\ee
where ${\CC}^{sta}_{\infty}$ is defined in \eqv(1.theo2.1); secondly, setting
\be
1-F_n(v):=\sum_{x\in\VV_n}\GG_{\a,n}(x)e^{-vc_n\l_n(x)}
=\sum_{k}\frac{\g_{n}(x^{(k)})}{\sum_{l}\g_{n}(x^{(l)})}e^{-s/\g_{n}(x^{(l)})},
\ee
a simple application of Proposition \thv(6.prop4) yields,
$
\lim_{n\rightarrow\infty}(1-F_n(v))=(1-F^{sta}(v)):={\CC}^{sta}_{\infty}(s)
$
$\boldsymbol{P}$-almost surely. 
We skip the details.
\end{proof}

\section{Proof of Theorem \thv(TTCF.theo1)}
\label{S10}

\begin{proof}[Proof of Theorem \thv(TTCF.theo1)]
Given $\rho \in (0,1)$, let $A_n^\rho (t,s)$ be the event
\be
A_n^\rho (t,s)
=\left\{n^{-1}\big(X_n(c_n t), X_n(c_n (t+s)\big)\geq 1-\rho\right\}.
\Eq(TTCF.2)
\ee
Observe that since $n^{-1}(x,x')=1-2n^{-1}\dist(x,x')$ (see \eqv(graph.dist))
\be
A_n^\rho (t,s)=\left\{ \dist(X_n(c_n t),X_n(c_n (t+s)))< \rho n/2\right\}.
\Eq(TTCF.5)
\ee
Denote by $\RR_n$ the range of the  rescaled
clock process $S_n$ of \eqv(G1.3.2') and write
\bea
\PP_{\pi_n}\left(A_n^\rho(s,t)\right)
&=&\PP_{\pi_n}\left(A_n^\rho(s,t)\cap \{\RR_n\cap (s,t)=\emptyset\}\right)
\Eq(TTCF.3)
\\
&+&\PP_{\pi_n}\left(A_n^\rho(s,t)\cap \{\RR_n\cap (s,t)\neq\emptyset\}\right).
\Eq(TTCF.4)
\eea
Because $A_n^\rho(s,t)\supset \{\RR_n\cap (s,t)=\emptyset\}$, the probability in the right-hand side of \eqv(TTCF.3) is
${\CC}_n(t,s)$. Thus, in order to prove \eqv(TTCF.theo1.1) (respectively, \eqv(TTCF.theo1.2)), we are left to establish that the  probability appearing in \eqv(TTCF.4) vanishes (respectively, vanishes faster than $1/\sqrt{n}$) as $n$ diverges. 
We distinguish the cases $\a(\varepsilon)< 1$ and $\a(\varepsilon)= 1$.

\smallskip
\paragraph{The case $\a(\varepsilon)< 1$.}
Consider the set 
\be
\TT_n(\d)\equiv\{x\in\VV_n\mid \t_n(x)>\d c_n\},\quad \d>0.
\Eq(TTCF.8)
\ee 
It follows from Theorem \thv(1'.theo1) on intermediate time-scales and from Theorem \thv(1'.theo2) on extreme time-scales that if $\RR_n\cap(t,t+s)\neq\emptyset$ then, with a probability that tends to one  as $n\uparrow \infty$ and $\d\downarrow 0$, the points $t$ and $t+s$ lie in disjoint constancy intervals of the clock  process, and such intervals are produced, asymptotically, by visits to  $\TT_n(\d)$. That is to say that there exists  $u_-< u_+$ such that  $c_n^{-1}\wt S_n(k_n(u_-))<t<c_n^{-1}\wt S_n(k_n(u_-)+1)$
and $c_n^{-1}\wt S_n(k_n(u_+))<t+s<c_n^{-1}\wt S_n(k_n(u_+)+1)$  and  these two (disjoint for large enough $n$) clock process increments correspond to visits of the jump chain $J_n$ to vertices  in $\TT_n(\d)$, which we denote by $z_-$ and $z_+$, respectively. We thus must establish that for all $0<\rho<1$ 
\be
\lim_{\d\rightarrow 0}\lim_{n\rightarrow\infty}\PP\left(z_+\in B_\rho(z_-)\cap \TT_n(\d)\right)=0
\Eq(TTCF.12)
\ee
either $\P$-a.s.~or in $\P$-probability, where
$
B_\rho(z)=\{x\in\VV_n\mid\dist(z,x)\leq \rho n/2\}
$ 
is the ball of radius $\rho n/2$ centered at $z$. We treat the intermediate and extreme time-scales separately. 

\smallskip
\noindent{\bf\emph{Intermediate time-scales.}} Denote by
$
H_n(A)=\inf\left\{i\geq 0 : J_n(i)\in A\right\}
$
the hitting time of $A\subset\VV_n$.
To prove \eqv(TTCF.12) it suffices to establish that after leaving $z_-$,  the chain $J_n$ does not visit any vertex in
$B_\rho(z_-)\cap \TT_n(\d)$ in the following $k_n(u_+)$ steps, which happens if 
\be
H(B_\rho(z_-)\cap \TT_n(\d))\gg a_n.
\Eq(TTCF.10)
\ee
The next lemma enables us to prove that \eqv(TTCF.10) holds true. Let
\be
\textstyle
m_n \equiv 2^n/\bigl[a_n^{-1}{n \choose \rho n/2}\bigr].
\Eq(TTCF.11)
\ee

\begin{lemma}
\TH(TTCF.lem1) 
Fix $z\in\VV_n$. There exists a constant $c>0$ such that if $a_n^{-1}{n \choose \rho n/2}>c \log n$  then there exists a subset 
$\O^{\tau}_{9}\subset\O^{\tau}$ with $\P\left(\O^{\tau}_{9}\right)=1$ such that on $\O^{\tau}_{9}$, for $n$ large enough, for all $u>0$
\be
\max_{x\in\VV_n}\left|
P_x
\left[H_n((B_\rho(z)\cap \TT_n(\d))\setminus \{x\})\geq u m_n \right]
-e^{-u} 
\right|
<C(\d,\varepsilon, \rho)\frac{1}{n}
\Eq(TTCF.lem1.1)
\ee
for some $0<C(\d,\varepsilon, \rho)<\infty$ independent of $n$. 
\end{lemma}

\begin{proof} The proof uses Theorem 1.3  of \cite{CG08}. Proceeding as in the proof of  (4.1) of \cite{CG08}  we get that there exists a constant $c>0$ such that if $a_n^{-1}{n \choose \rho n/2}>c \log n$ then there exists $\O^{\tau}_{10}\subset\O^{\tau}$ with $\P\left(\O^{\tau}_{10}\right)=1$ such that on $\O^{\tau}_{10}$, for all large enough $n$, both
\be
\textstyle
|B_\rho(z)\cap \TT_n(\d)|=a_n^{-1}{n \choose \rho n/2}\d^{-\a(\varepsilon)}(1+o(1)),\quad
|\TT_n(\d)|=a_n^{-1}2^n\d^{-\a(\varepsilon)}(1+o(1)).
\Eq(TTCF.7)
\ee
A simple rerun of the proof of Theorem 1.1  of \cite{CG08} then yields the claim of the lemma.\end{proof}

Note that the probability in \eqv(TTCF.12) being an increasing function of $\rho$, it suffices to prove \eqv(TTCF.12) for all large enough $\rho\in (0,1)$. Also note that when $c_n$ is an intermediate time-scale with $0<\varepsilon< 1$, one may always find 
$\rho'\in (0,1)$ such that for all $\rho\in [\rho',1)$, $a_n^{-1}{n \choose \rho n/2}\gg \log n$. Furthermore, for all $0<\rho<1$ there exists $\zeta>0$ such that ${n \choose \rho n/2}/2^n<2^{-\zeta n}$. Thus in particular
$
m_n \gg a_n
$.
Hence by Lemma \thv(TTCF.lem1), for all $0<\varepsilon< 1$ and $0<\rho<1$, \eqv(TTCF.12) holds true $\P$-a.s.

When on the contrary $c_n$ is an intermediate time-scale with $\varepsilon= 1$ then, for all $0<\rho<1$ and large enough $n$,
$
a_n^{-1}{n \choose \rho n/2}
\equiv r^2_n
\leq c\frac{2^n}{a_n}e^{-n\left\{\ln 2+ \frac{\rho}{2}\ln \frac{\rho}{2} + (1-\frac{\rho}{2})\ln (1-\frac{\rho}{2})\right\}}
\leq e^{-c(\rho)n}
$ 
where by \eqv(1.2) $c(\rho)>0$, and by a first order Tchebychev inequality
\be
\textstyle
\P\left[|(B_\rho(z)\cap \TT_n(\d))\setminus z|>r_n\right]
\leq \d^{-\a(\varepsilon)}r_n.
\Eq(TTCF.13)
\ee
In that case $B_\rho(z_-)\cap \TT_n(\d)$ reduces to the singleton $\{z_-\}$. By Theorem 7.5 of \cite{BG08}, $H_n(z_-)$ is asymptotically exponentially distributed with mean value $2^n$, and since $ 2^n\gg a_n $ we get that  for $\varepsilon=1$ 
and all $0<\rho<1$, \eqv(TTCF.12) holds true $\P$-a.s..
This concludes the case of intermediate time-scales.

\smallskip
\noindent{\bf\emph{Extreme time-scales.}} In that case $|\TT_n(\d)|$ is asymptotically finite. Indeed, replacing the  variables $(\g_n(x),\, x\in\VV_n)$ by the representation \eqv(6.5), it follows from Lemma \thv(6.lemma1), Lemma \thv(6.lemma2) and Proposition \thv(6.prop4) that $\lim_{n\rightarrow\infty}|\TT_n(\d)|=\Upsilon([\d,\infty])$ where $\Upsilon([\d,\infty])$ is a Poisson random variable whose mean value $M(\d)$ obeys
$
M(\d) \sim \d^{-\a}
$  
as
$\d\rightarrow 0$ $\boldsymbol{P}$-a.s.~by  Lemma \thv(1.lemma5).
Here again
$
B_\rho(z_-)\cap \TT_n(\d)=\{z_-\}
$
$\boldsymbol{P}$-a.s.~(see Lemma 2.12 of  \cite {BBG1}) and  by Theorem 7.5 of \cite{BG08}, $H_n(z_-)$ is asymptotically exponentially distributed with mean value $2^n(1+o(1))\sim a_n$. Thus, in contrast to \eqv(TTCF.10), the jump chain has a positive probability to revisit any element of $\TT_n(\d)$ many times during its first $k_n(u_+)$ steps. However, by Corollary 1.5 of \cite{BG08}, at each re-entrance in $\TT_n(\d)$ starting from $\VV_n\setminus \TT_n(\d)$, the jump chain enters $\TT_n(\d)$ with a uniform distribution, uniformly in its starting point.
Combining these observations yields
$
\lim_{n\rightarrow\infty}\PP\left(z_+\in B_\rho(z_-)\cap \TT_n(\d)\right)=\Upsilon^{-1}([\d,\infty])\sim \d^{\a}
$  
as
$\d\rightarrow 0$. Thus \eqv(TTCF.12) holds true  $\boldsymbol{P}$-a.s. (that is to say, in $\P$-probability).

\smallskip
\paragraph{The case $\a(\varepsilon)= 1$.} It follows from Theorem \thv(1'.theo1), (ii), and Proposition \thv(LCor.3) that
the leading contributions to $S_n(t)$ no longer come from visits to extremes, as in the case $\a(\varepsilon)< 1$, but from visits to the set of ``typical" increments, 
\be
\MM_n(\d)\equiv\{x\in\VV_n\mid \d \E\t_n(x) <\t_n(x)< \d^{-1} \E\t_n(x)\},\quad \d>0.
\Eq(TTCF.14)
\ee 
Clearly, this set is much more dense than $\TT_n(\d)$. As in the case $\a(\varepsilon)< 1$ with $0<\varepsilon<1$, there exists a constant $c>0$ such that if $a_n^{-1/4}{n \choose \rho n/2}>c \log n$ then there exists $\O^{\tau}_{11}\subset\O^{\tau}$ with $\P\left(\O^{\tau}_{11}\right)=1$ such that on $\O^{\tau}_{11}$, for all large enough $n$, both
\be
\textstyle
|B_\rho(z)\cap \MM_n(\d)|=  \frac{1}{\sqrt n}{n \choose \rho n/2}a_n^{-1/4}
(1+o(1)),\,\,\,
|\MM_n(\d)|=\frac{1}{\sqrt n}2^na_n^{-1/4}
(1+o(1)),
\Eq(TTCF.15)
\ee
where we used that 
$
0<\b=\b(\varepsilon)\leq\beta_c(1)
$. 
By Theorem 1.1 of \cite{CG08}, 
$H(B_\rho(z_-))$ is asymptotically exponentially distributed with mean value
$
m_n^{-}\equiv\sqrt na_n^{1/4}\ll a_n
$.
Thus, along trajectories of length $\sim a_n$, the jump chain typically revisits the element of $\MM_n(\d)$ very many times. However, proceeding again as in the case $\a(\varepsilon)< 1$ with $0<\varepsilon<1$, we see that for all  $0<\varepsilon\leq 1$, one may choose $\rho'\in (0,1)$ such that for all $\rho\in [\rho',1)$, 
by Lemma \thv(TTCF.lem1),
$H(B_\rho(z_-)\cap \MM_n(\d))$ also is asymptotically exponentially distributed, with mean value
\be
\textstyle
2^{n(1-\z)}\ll  m_n^{+}\equiv 2^n/\bigl[a_n^{-1/4}{n \choose \rho n/2}\bigr]\ll 2^n(\log n)^{-1}.
\Eq(TTCF.16)
\ee
From these two results we deduce that for all large enough n and all $x\in\VV_n\setminus \MM_n(\d)$,
\be
P_x
\left[H_n((B_\rho(z)\cap \MM_n(\d)))\leq  H_n(\MM_n(\d))\right]
\leq 
c(\d)\frac{m_n^{-}}{m_n^{+}}+ \OO(1/n)\ll \sqrt{n}^{-1}
\Eq(TTCF.18)
\ee
for some constant $c(\d)>0$.
Since this holds true uniformly for all starting point in $\VV_n\setminus \MM_n(\d)$, one readily gets that for all $0<\rho<1$ and $0<\varepsilon\leq 1$, $\P$-a.s.
\be
\lim_{\d\rightarrow 0}\lim_{n\rightarrow\infty}\sqrt{n}\PP\left(z_+\in B_\rho(z_-)\cap \MM_n(\d)\right)
=0
\Eq(TTCF.17)
\ee
where $z_+$ and $z_-$ have the same meaning as in \eqv(TTCF.12). This proves \eqv(TTCF.theo1.2) and concludes the proof of Theorem \thv(TTCF.theo1).
\end{proof}

\appendix
\section{Calculations}
\TH(A)
This appendix contains calculatory results on the moments of $f_\d(\g_n(x))$ and $g_\d(\g_n(x))$ that are needed in several places in the proofs. Our first lemma provides  asymptotic bounds on $a_n\E\left(f_\d(\g_n(x))\right)$ and $a_n\E\left(f_\d(\g_n(x))^2\right)$ needed in the verification of Condition (A3').
\begin{lemma}\TH(Lem2.3.1)
For all $t>0$ and $\d>0$ and for $n$ large enough there exist constants $0<c_0,c_1<\infty$ such that
\be\Eq(2.3.3)
a_n\E\left(f_\d(\g_n(x))\right) \leq c_0  \d  
\ee
\be\Eq(2.3.4)
a_n\E\left(f_\d(\g_n(x))^2\right)\leq \d^4+c_1.
\ee
\end{lemma}
\begin{proof}
We observe that
\be
f_\d(u)\leq \d^2 \quad \forall u\in (0,\infty).
\ee
We decompose $a_n\E\left(f_\d(\g_n(x))\right)$ in the following way
\be
a_n\E\left(f_\d(\g_n(x))\right)= a_n\E\left(f_\d(\g_n(x))\1_{\{\g_n(x)>\d\}}\right)+a_n\E\left(f_\d(\g_n(x))\1_{\{\g_n(x)\leq \d\}}\right)\equiv (1) + (2)
\ee
We start by controlling the behavior of $(1)$.
\be
(1)\leq \d^2a_n\P\left(\g_n(x)>1\right) \sim \d,
\ee
where we used the definition of $a_n$ and $c_n$. 
Turning to $(2)$ we have
\bea\Eq(2.3.2)
(2) & \leq & a_n\E\left(\g_n(x)^2\1_{\{\g_n(x)\leq \d\}}\right)\nonumber\\
& = & \frac{a_n e^{2n\b^2}}{c_n^2}\int_{-\infty}^{\frac{\log (c_n\d)}{\sqrt{n}\b}-2\sqrt{n}\b} \frac{e^{-u^2/2}}{\sqrt{2\pi}} \mbox{d}u\nonumber\\
& \sim & \frac{a_n e^{2n\b^2}}{c_n^2} \left(\sqrt{2\pi}\left(-\frac{\log (c_n\d)}{\sqrt{n}\b}+2\sqrt{n}\b\right)\right)^{-1}e^{-\frac{1}{2}\left(\frac{\log (c_n\d)}{\sqrt{n}\b}-2\sqrt{n}\b\right)^2},
\eea
where we used that by \eqv(2.23''.2) $\frac{\log (c_n\d)}{\sqrt{n}\b}-2\sqrt{n}\b\to -\infty$ as $n\to\infty$. Now we start to expand the exponent of the exponential function and plug in \eqv(2.23''.2).
\be
\eqv(2.3.2)  = a_n\left(\sqrt{2\pi}\left(-\frac{\log (c_n\d)}{\sqrt{n}\b}+2\sqrt{n}\b\right)\right)^{-1} \d^2 e^{-\frac{1}{2}\left(\frac{\log c_n\d}{\sqrt{n}\b}\right)^2}
  = c_0'  \d(1+o(1)),
\ee
where $0<c_0'<\infty$. Putting our estimates together we have that for $n$ large enough there exists a constant $0<c_0<\infty$ such that
\be 
a_n\E\left(f_\d(\g_n(x))\right) \leq c_0 \d .
\ee
In a similar way we treat $a_n\E\left(f_\d(\g_n(x))^2\right)$. This time we truncate at one, namely
\be
a_n\E\left(f_\d(\g_n(x))^2\right) = a_n\E\left(f_\d(\g_n(x))^2\1_{\{\g_n(x)>1\}}\right)+ a_n\E\left(f_\d(\g_n(x))^2\1_{\{\g_n(x)\leq 1\}}\right).
\ee
 For the first summand we use again the bound on $f$ and the definition of the time-scale to bound it by $\d^4$. And for the second summand we use the same method as for (2): applying Gaussian estimates, expanding the resulting term and plugging in the exact representation of $c_n$. The bound we obtain is a constant. Putting these estimates together we have for $n$ large enough
\be\textstyle
a_n\E\left(f_\d(\g_n(x))^2\right)\leq \d^4+c_1.
\ee
\end{proof}
In the verification of Condition (A3) a slightly different function $g_\d$ appeared.
In the forthcoming lemma we control the first and second moment of $g_\d(\g_n(x))$ when $0<\a(\varepsilon)<1$.
\begin{lemma} \TH(Lem2.4.0)Let $c_n$ be an intermediate time-scale and $0<\b<\infty$ and $0<\a(\varepsilon)<1$. Then there exists constants $c_3$ and $c_4$ such that the following holds for $n$ large enough
\be\Eq(2.3.3')
a_n\E\left(g_\d (\g_n(x))\right) \leq c_3 \d^{1-\a(\varepsilon)},
\ee
\be\Eq(2.2.4')
a_n\E\left(g_\d (\g_n(x))^2\right) \leq c_4.
\ee
\end{lemma}
\begin{proof}
We observe that
\be\Eq(2.2.5'')
g_\d(u) \leq \d \quad\forall u>0.
\ee
As in the proof of the previous lemma we write
\be
a_n\E\left(g_\d (\g_n(x))\right) = a_n\E\left(g_\d (\g_n(x))\1_{\{\g_n(x)>\d\}}\right) +a_n\E\left(g_\d (\g_n(x))\1_{\{\g_n(x)\leq\d\}}\right) \equiv (1) + (2).
\nonumber
\ee
The first summand $(1)$ we control by
\be
(1) \leq \d a_n\P\left(\g_n(x)>\d\right) \sim \d^{1-\a(\varepsilon)}.
\ee
For $(2)$ we use Gaussian estimates.
\bea\Eq(2.2.5')
(2) & \leq & a_n\E\left(\g_n(x) \1_{\g_n(x)\leq \d}\right) 
 =  \frac{a_ne^{n\b^2/2}}{c_n}\int_{-\infty}^{\frac{\log(c_n\d)}{\sqrt{n}\b}-\sqrt{n}\b} \frac{e^{-u^2/2}}{\sqrt{2\pi}}\mbox{d}u\nonumber\\
& \sim &  \frac{a_ne^{n\b^2/2}}{c_n} \left(\sqrt{2\pi}\left(\sqrt{n}\b-\frac{\log(c_n\d)}{\sqrt{n}\b}\right)\right)^{-1}e^{-\frac{1}{2}\left(\frac{\log(c_n\d)}{\sqrt{n}\b}-\sqrt{n}\b\right)^2},
\eea
where we used that since $\b>\b_c(\varepsilon)$ we have by \eqv(2.23''.2) $\frac{\log(c_n\d)}{\sqrt{n}\b}-\sqrt{n}\b \to -\infty $ as $n\to \infty$. Now, expanding the terms and inserting the exact representation of $c_n$, \eqv(2.2.5') is equal to
\be
\d a_n\left(\sqrt{2\pi}\left(\sqrt{n}\b-\frac{\log(c_n\d)}{\sqrt{n}\b}\right)\right)^{-1} e^{-\frac{1}{2}\left(\frac{\log(c_n\d)}{\sqrt{n}\b}\right)^2}
 \leq  c_3' \d^{1-\a(\varepsilon)},
\ee
for some constant $0<c_3'<\infty$. Putting the estimates on $(1)$ and $(2)$ together we get that there exists a constant $0<c_3<\infty$ such that
\be
a_n\E\left(g_\d (\g_n(x))\right) \leq c_3 \d^{1-\a(\varepsilon)}.
\ee
To control $a_n\E\left(g_\d (\g_n(x))^2\right)$ one proceeds in exactly the same way.
\end{proof}
To study the behavior of $M_n(t)$, and in particular to check Condition (B1), we needed a control on the moments of $g_1(\g_n(x))$ when $\b=\b_c(\varepsilon)$ which is done in the next lemma.

\begin{lemma}\TH(Lem2.4.1)
Let $c_n$ be an intermediate scale.
\item{(i)} Let $\b=\b_c(\varepsilon)$. Then 
\be\Eq(2.4.4')
\E\left(g_1(\g_n(x))\right)  \leq \frac{e^{n\b^2/2}}{c_n}(1+o(1)).
\ee
Moreover, if $\lim_{n\to\infty}\sqrt{n}\b-\frac{\log c_n}{\beta\sqrt{n}} = \theta$ for some $\theta\in(-\infty,\infty)$. Then,
\be\Eq(2.4.4)
a_n\E\left(g_1(\g_n(x))\right)  =   \Phi(\theta) \frac{a_ne^{n\b^2/2}}{c_n}(1+o(1))
=
\Phi(\theta)\beta \sqrt{2n\pi} e^{\theta^2/2}(1+o(1)).
\ee
\item{(ii)} Let  $\b=\b_c(\varepsilon)$. For $n$ large enough there exists a constant $0<c_2<\infty$ such that 
\be\Eq(2.4.5)
a_n \E\left(g_1(\g_n(x))^l\right)  \leq c_2,  \quad 2\leq l\leq 4.
\ee
\item{(iii)} Let $\beta<\beta_c$. Then \be\Eq(Mainz.101)
\E\left(g_1(\g_n(x))\right)  = \frac{e^{n\b^2/2}}{c_n}(1+o(1)).
\ee
If $\beta>\beta_c/2$, then
\be\Eq(Mainz.102)
a_n \E\left(g_1(\g_n(x))^2\right)  \leq c_2 .
\ee
Otherwise $a_n \E\left(g_1(\g_n(x))^2\right)  \leq a_ne^{2n\beta^2}/c_n^2$ and 
\be\Eq(Mainz.103) 
 \log \frac{a_n \E\left(g_1(\g_n(x))^2\right)}{a_n^2 e^{n\beta^2}/c_n^2}= n(2\beta-\beta_c)/2.
 \ee
\end{lemma}
\begin{proof}
Recall that $g_1(u) \leq 1$, $ \forall u>0$.
To prove assertion (i) we rewrite $\E\left(g_1(\g_n(x))\right)$ as
\bea\Eq(A.21)
& &\frac{e^{n\b^2/2}}{\sqrt{2\pi}c_n} \int_{-\infty}^{\infty} e^{\sqrt{n}\b z }\left(1-e^{-c_n  e^{-\sqrt{n}\b z}}\right)e^{-z^2/2}\mbox{d}z\nonumber\\
& = & \frac{e^{n\b^2/2} }{c_n} - \frac{ e^{n\b^2/2}}{c_n \b \sqrt{2 \pi n}} \int_{-\infty}^{\infty} e^{y+\log c_n}e^{-\left(\frac{y}{\b\sqrt{n}}+\frac{\log c_n}{\b\sqrt{n}}\right)^2- e^{-y}}\mbox{d}y.
\eea
Now one can cut the domain of integration into different pieces. Observe that in the region $y>\log n$ the integral is equal to
\bea\Eq(NY.113)
&&(1+o(1))\frac{e^{n\b^2/2}}{c_n \b\sqrt{2\pi n}} \int^{\infty}_{\log n} e^{y+\log c_n}e^{-\left(\frac{y}{\b\sqrt{n}}+\frac{\log c_n}{\b\sqrt{n}}\right)^2/2}\mbox{d}y\nonumber\\
 &= & (1+o(1))\frac{ e^{n\b^2/2}}{c_n\sqrt{2\pi}}\int^{\infty}_{\frac{\log n}{\sqrt{n}\b} -\frac{\log c_n}{\b \sqrt{n}}+\sqrt{n}\b}e^{-y^2/2}\mbox{d}y.
\eea
  If $\sqrt{n}\b-\frac{\log c_n}{\beta\sqrt{n}} \to \theta$ for some constant $\theta$ as $n\to \infty$  we have that \eqv(A.21) is equal to $ (1+o(1))\frac{e^{n\b^2/2}}{c_n}(1-\Phi(\theta))$.
Proceeding as in \eqv(NY.113) one can bound the integral in \eqv(A.21) on the  domain of integration $|y|<\log n$ by $o(1)\frac{ e^{n\b^2/2}}{c_n}$. For $y<-\log n$, $e^{-y}>n$ which implies that the on that part of the domain of integration the integral in \eqv(A.21) is equal to $o(1)\frac{ e^{n\b^2/2}}{c_n}$. This yields the first equality in \eqv(2.4.4), and as the Gaussian integral is always between zero and one, this also implies \eqv(2.4.4'). The second inequality in \eqv(2.4.4) follows from the first by \eqv(scale.1) of Lemma \thv(Th.Scale).
We now turn to assertion (ii) and consider $\E\left(g_1(\g_n(x))^2\right)$. We will split this term into two terms:
\be
a_n \E\left(g_1(\g_n(x))^2\right)  = a_n\E\left(g_1(\g_n(x))^2\1_{\{\g_n(x)>1\}}\right) +a_n \E\left(g_1(\g_n(x))^2\1_{\{\g_n(x)\leq 1\}}\right) \equiv (1) + (2).
\nonumber
\ee
For (1) we use the definition of the scaling $a_n$ and $c_n$ and the bound \eqv(2.2.5'')
\be
(1) \leq a_n\P\left(\g_n(x)>1\right) = 1.
\ee
For Term (2) we use exact Gaussian estimates to bound
\bea\Eq(2.4.2)
(2) & \leq & \frac{a_n}{c_n^2} \int_{-\infty}^{\frac{\log c_n}{\sqrt{n}\b}} e^{2\b\sqrt{n} u}\left(1-e^{- c_ne^{-\sqrt{n}\b u}}\right)^2 \frac{e^{-u^2/2}}{\sqrt{2\pi}} \mbox{d}u\nonumber\\
& \leq & \frac{a_n e^{2n\b^2}}{c_n^2}\int_{-\infty}^{\frac{\log c_n}{\sqrt{n}\b}-2\sqrt{n}\b} \frac{e^{-r^2/2}}{\sqrt{2\pi}} \mbox{d}r\nonumber\\
& \sim & \frac{a_n e^{2n\b^2}}{c_n^2} \left(\sqrt{2\pi}\left(-\frac{\log c_n}{\sqrt{n}\b}+2\sqrt{n}\b\right)\right)^{-1}e^{-\left(\frac{\log c_n}{\sqrt{n}\b}-2\sqrt{n}\b\right)^2/2},
\eea
where we use that by \eqv(2.23'), $\frac{\log c_n}{\sqrt{n}\b}-2\sqrt{n}\b\to -\infty$ as $n\to \infty$. Plugging in \eqv(2.23') yields
\be
\eqv(2.4.2) =  a_n \left(\sqrt{2\pi}\left(-\frac{\log c_n}{\sqrt{n}\b}+2\sqrt{n}\b\right)\right)^{-1} e^{-\left(\frac{\log c_n}{\sqrt{n}\b}\right)^2/2}
=  
c_2'(1+o(1)),
\ee
where $0<c_2'<\infty$. Putting both estimates together we get that for $n$ large there exists a constant $0< c_2<\infty$ such that
\be \Eq(Mainz.114)
a_n \E\left(g_1(\g_n(x))^2\right)  \leq c_2  .
\ee
Proceeding in exactly the same way  with $a_n \E\left(g_1(\g_n(x))^3\right) $ and $a_n \E\left(g_1(\g_n(x))^4\right) $, one readily obtains \eqv(2.4.5) for $l=3$ and $l=4$.

Part (iii) follows from  computations similar to those of (i) and (ii). \eqv(Mainz.103) follows from \eqv(2.23'').
\end{proof}

\section{The centering term $M_n(t)$ at criticality}
\TH(B)

In this appendix we collect the fine asymptotics needed to control the centering term $M_n(t)$ on the critical line 
 $\b=\b_c(\varepsilon)$,  $0<\varepsilon\leq 1$. Computing $\E\left( \EE \left(M_n(t) \right)\right)$ at $\b=\b_c(\varepsilon)$ gives
\bea\Eq(As.1)
\E\left( \EE \left(M_n(t) \right)\right)& = &  \E\left(\sum_{i=1}^{[a_nt]} \EE \left(c_n^{-1}{\tau_n(J_n(i)) e_{n,i}}1_{\left \{0<c_n^{-1}{\tau_n(J_n(i)) e_{n,i}}<1 \right \}} \right)\right)\nonumber\\
& = &  \lfloor a_nt\rfloor\E\left(g_1(\g_n(x))\right)  =   c t\frac{a_ne^{n\b^2/2}}{c_n}(1+o(1)),
\eea
 where by the first equality in \eqv(2.4.4) $c$ is some constant $>0$.
The following lemma proves the general diverging behavior of $\E(\EE(M_n(1))$. Recall the notation 
\eqv(1.1')-\eqv(1.theo1.0') of Definition \thv(1.def1).
\begin{lemma}\TH(Lem.As) 
Given $0<\varepsilon\leq 1$, let  $a_n$ and $c_n$ be sequences satisfying \eqv(1.1') and \eqv(1.2) and let 
$\b=\b_c(\varepsilon)$. Then
\be
\lim_{n\to\infty} \frac{a_ne^{n\b^2/2}}{c_n} = \infty.
\ee
\end{lemma}
\begin{proof}
By \eqv(2.23'') with $\b=\b_c(\varepsilon)$,  $\log a_n  = \frac{1}{2}(n\b^2+f(n))$ for some sequence $f(n)$ such that $\frac{f(n)}{n\b^2}=o(1)$. Furthermore, by \eqv(2.23''.2),
 $\log(\log a_n) =  \log(\frac{n\b^2+f(n)}{2})$ and $\sqrt{2\log a_n}= \sqrt{n\b^2+f(n)}$. Note that due to the asymptotic behavior of $f(n)$, $\log(\log a_n)$ is positive for $n$ large enough. Hence it suffices to show that
\be\Eq(Lem.As.2)
\frac{\log(a_n e^{n\b^2/2})}{\sqrt{n}\b} \geq \sqrt{2\log a_n}.
\ee
Plugging in the expressions for $\log a_n$, \eqv(Lem.As.2) reads
\be\Eq(scale.2)
\sqrt{n}\b + \frac{f(n)}{2\sqrt{n}\b}\geq \sqrt{n\b^2 + f(n)},
\ee
which is always satisfied and equality holds if and only if $f(n)=0$.
\end{proof}
\begin{lemma}\TH(Th.Scale)
If in addition to the assumptions of Lemma \thv(Lem.As),
$\lim_{n\to\infty}\frac{\log c_n}{\sqrt{n}\beta}-\sqrt{n}\beta=\theta$ for some $\theta\in(-\infty,\infty)$, then
\be
\Eq(scale.1)
\lim_{n\to\infty}\sqrt{n}\frac{c_n}{a_ne^{n\beta^2/2}}=\frac{1}{\beta \sqrt{2\pi} } e^{-\theta^2/2}.
\ee
\end{lemma}
\begin{proof}
Using the notation of the proof of Lemma \thv(Lem.As), \eqv(scale.1) follows from \eqv(2.23''.2) with $\lim_{n\to\infty} \frac{f(n)}{2\sqrt{n}\beta}=\theta$. Namely, under the assumption of the lemma, \eqv(2.23''.2) may be written as
\be\Eq(NY.112)
c_n=\frac{1}{\beta \sqrt{2\pi n}} e^{n\beta^2+\frac{f(n)}{2}-\frac{1}{8}\left(\frac{f(n)}{\beta\sqrt{n}}\right)^2+o(1)}
\ee
by Taylor expansion of the square root. Eq.~\eqv(NY.112) also implies that $\lim_{n\to\infty}\frac{\log c_n}{\sqrt{n}\beta}-\sqrt{n}\beta=\theta$ if and only if $\lim_{n\to\infty} \frac{f(n)}{2\sqrt{n}\beta}=\theta$.
\end{proof}

 \section{Auxiliary Lemmas needed in the proof of Theorem \thv(TCor.2) }
 \label{app.3}
 Recall that $\widetilde{a}_n$ and $A_n(t)$ are defined in \eqv(Cor.2.1') and  \eqv(NY.108), respectively.
\begin{lemma}\TH(LCor.4)  Let $c_n$ be an intermediate scale with $\lim_{n\to\infty}\frac{\log c_n}{\sqrt{n}\beta}-\sqrt{n}\beta=\theta$ for some $\theta\in(-\infty,\infty)$ and $\b=\b_c(\ve)$ with $0<\ve\leq 1$. 
If $\sum a_n/2^n<\infty$ we have for all $t,s>0$ and for all $\e>0$,  $\P$-a.s.
\be
\lim_{n\to\infty}\sqrt{n}
\PP\left(
\left\vert\sum_{k=1}^{\lfloor\widetilde{a}_nt\rfloor}
\PP
\Bigl(\t_n(J_n(k+1))e_{n,k+1}>c_ns|J_n(k)\Bigr)
-\frac{\widetilde{a}_nt}{a_ns}\right\vert>\frac{\widetilde{a}_n\e}{a_n\sqrt{A_n(t)}}
\right) = 0
\ee
If $\sum a_n/2^n=\infty$ the same holds in $\P$-probability.
\end{lemma}
\begin{proof}
Proceeding as in the
the proof of  Proposition \thv(4.prop1) one readily establishes that
\bea
&&\PP\left(\left\vert\sum_{k=1}^{\lfloor\widetilde{a}_nt\rfloor}
\PP\Bigl(\t_n(J_n(k+1))e_{n,k+1}>c_ns|J_n(k)\Bigr)
-\frac{\lfloor\widetilde{a}_nt\rfloor}{a_n}\nu_n(s,\infty)\right\vert>\frac{\widetilde{a}_n\e}{a_n\sqrt{A_n(t)}}\right)\nonumber\\
&&\leq \left(\frac{a_n \sqrt{A_n(t)}}{\widetilde{a}_n \e}\right)^2
\left(
\left(\frac{\lfloor\widetilde{a}_nt\rfloor}{a_n}\right)^2\frac{\nu_n^2(u,\infty)}{2^{n-1}}
+\frac{\lfloor\widetilde{a}_nt\rfloor}{a_n}\Theta^1_n(u)\right).
\eea
where $\Theta^1_n(u)$ is defined in \eqv(4.prop1.2).
Using Proposition \thv(5.prop6) and \eqv(Cor.2.1'') yields the claim of Lemma \thv(LCor.4).
\end{proof}

\begin{lemma} \TH(Lem.cor5)Let $c_n$ be an intermediate scale with $\lim_{n\to\infty}\frac{\log c_n}{\sqrt{n}\beta}-\sqrt{n}\beta=\theta$ for some $\theta\in(-\infty,\infty)$ and $\b=\b_c(\ve)$ with $0<\ve\leq 1$.
Then we have for all $x>0$ that $\P$-a.s.
\be\Eq(C5.1)
\lim_{n\to\infty}\frac{a_n}{\widetilde{a}_n}\sum_{l=1}^{\theta_n}\PP\left(\t_n(J_n(k+1)e_{n,k+1}>c_n x\right)=0.
\ee
\end{lemma}
\begin{proof}Using a first order Tchebychev inequality we have
\be\Eq(C5.2)
\P\left(\frac{a_n}{\widetilde{a}_n}\sum_{l=1}^{\theta_n}\PP\left(\t_n(J_n(k+1)e_{n,k+1}>c_nx\right)>\e\right)\leq \e^{-1} \frac{a_n}{\widetilde{a}_n}\frac{\theta_n}{a_n}\E\nu_n(x,\infty).
\ee
In view of  Lemma \thv(5.lemma7)  and since $\sum \frac{\theta_n}{\widetilde{a}_n}<\infty$, the claim of Lemma \thv(Lem.cor5) follows.
\end{proof}
\begin{lemma}\TH(Lem.cen5)
 Let $c_n$ be an intermediate scale with $\lim_{n\to\infty}\frac{\log c_n}{\sqrt{n}\beta}-\sqrt{n}\beta=\theta$ for some $\theta\in(-\infty,\infty)$ and $\b=\b_c(\ve)$ with $0<\ve\leq 1$. 
If $\sum a_n/2^n<\infty$ we have for all $t,s>0$ and for all $\e'>0$ that $\P$-a.s.
\be\Eq(NY.114)
\lim_{n\to\infty}\sqrt{n}\PP\left(  M_n(\lfloor\widetilde{a}_nt(1+\e')/a_n\rfloor)  < t\right) =0.
\ee
If $\sum a_n/2^n=\infty$ the same holds in $\P$-probability.
\end{lemma}
\begin{proof} Observe first that  $M_n(\lfloor\widetilde{a}_nt(1+\e')/a_n\rfloor)=M_n(ct/\sqrt{n})$ for some constant $c$. We know from Proposition \eqv(4.prop1') and Lemma \thv(4.1'')  that it concentrates around $\EE M_n(\lfloor\widetilde{a}_nt(1+\e')/a_n\rfloor)$ either $\P$-a.s.~or in $\P$-probability and the bounds are in the worst case linear in $t$. Moreover by linearity of $\EE M_n(\lfloor\widetilde{a}_nt(1+\e')/a_n\rfloor)$ and Lemma \thv(4.1'') the claim of Lemma \thv(Lem.cen5) follows.
\end{proof}


\end{document}